\documentclass{amsproc}
\usepackage{amssymb}
\usepackage{amsmath}
\usepackage{amsfonts}

\theoremstyle{plain}

\newtheorem{definition}{Definition}
\newtheorem{example}{Example}

\newtheorem{proposition}{Proposition}

\newtheorem{theorem}{Theorem}
\numberwithin{equation}{section}
\input{tcilatex}

\begin{document}
\title{A new approach to soft set through applications of cubic set }
\author{Saleem Abdullah}
\address{Department of Mathematics, Quaid-i-Azam University, Islamabad
Pakistan}
\email{saleemabdullah81@yahoo.com}
\author{Imran Khan}
\address{Department of Mathematics, Quaid-i-Azam University, Islamabad
Pakistan}
\email{imran\_maths2002@yahoo.com}
\author{Muhammad Aslam}
\address{Department of Mathematics, Quaid-i-Azam University, Islamabad
Pakistan}
\email{draslamqau@yahoo.com}
\keywords{Fuzzy set, Cubic set, Soft set, Cubic soft, Internal cubic soft
set, External cubic soft set}

\begin{abstract}
The notions of (internal, external) cubic soft sets,
P-(R-)order,P-(R-)union, P-(R-)intersection and P-OR, R-OR, P-AND and R-AND
are introduced, and related properties are investigated. We show that the
P-union and the P-intersection of internal cubic soft sets are also internal
cubic soft sets. We provide conditions for the P-union (resp.
P-intersection) of two external cubic soft sets to be an internal cubic soft
set. We give conditions for the P-union (resp. R-union and R-intersection)
of two external cubic soft sets to be an external cubic soft set. We
consider conditions for the R-intersection (resp.P-intersection) of two
cubic sof sets to be both an external cubic soft set and an internal cubic
soft set.
\end{abstract}

\maketitle

\section{INTRODUCTION}

In order to deal with many complicated problems in the fields of
engineering, social science, economics, medical science etc involving
uncertainties, classical methods are found to be inadequate in recent times.
Molodstov \cite{8} pointed out that the important existing theories viz.
probability theory, fuzzy set theory, intuitionistic fuzzy set theory, rough
set theory etc, which can be considered as mathematical tools for dealing
with uncertainties, have their own difficulties. He further pointed out that
the reason for these difficulties is, possibly, the inadequacy of the
parameterization tool of the theory. In 1999 he proposed a new mathematical
tool for dealing with uncertainties which is free of the difficulties
present in these theories. He introduced the novel concept of soft sets and
established the fundamental results of the new theory. He also showed how
soft set theory is free from parameterization inadequacy syndrome of fuzzy
set theory, rough set theory and probability theory etc. Many of the
established paradigms appear as special cases of soft set theory. In 2003,
P. K .Maji, R. Biswas and A. R. Roy \cite{7} studied the theory of soft sets
initiated by Molodstov. They defined equality of two soft sets, subset and
super set of a soft set, complement of a soft set, null soft set, and
absolute soft set with examples. Soft binary operations like AND, OR and
also the operations of union, intersection were also defined. Pei and Miao 
\cite{9} and Chen et al. \cite{3} improved the work of Maji et al. \cite{5,
7}. In 2009, M. Irfan Ali et al., \cite{2} gave some new notions such as the
restricted intersection, the restricted union, the restricted difference and
the extended intersection of two soft sets along with a new notion of
complement of a soft set. Sezgin and Atag\"{u}n \cite{op} studied on soft
set operations. Babitha and Sunil introduced soft set relations and
functions \cite{babit}. Majumdar and Samanta, worked on soft mappings \cite%
{majum} were proposed and many related concepts were discussed too.
Moreover, the theory of soft sets has gone through remarkably rapid strides
with a wide-ranging applications especially in soft decision making as in
the following studies: \cite{cag-dec,cag-redef} and some other fields such
as \cite{fee,fee2,f1,f2}. Since its inception, it has received much
attention in the mean of algebraic structures. In \cite{2A}, Akta\d{s} and 
\d{C}a\u{g}man applied the concept of soft set to groups theory and
introduced soft group of a group. Feng et.al, studied soft semirings by
using soft set \cite{3A}. Recently, Acar studied soft rings \cite{4A}. Jun
et. al, applied the concept of soft set to BCK/BCI-algebras \cite%
{jun1,jun2,jun3}. Sezgin and Atag\"{u}n initiated th concept of normalistic
soft groups \cite{op1}. Zhan et.al, worked on soft ideal of BL-algebras \cite%
{zhan}. In \cite{kazanc}, Kazanc\i\ et. al, used the concept of soft set to
BCH-algebras. Sezgin et. al, studied soft nearrings \cite{sez-near2}. Atag%
\"{u}n and Sezgin \cite{atag} defined the concepts of soft subrings and
ideals of a ring, soft subfields of a field and soft submodules of a module
and studied their related properties with respect to soft set operations
also union soft substructues of nearrings and nearring modules are studied
in \cite{aslii2}. \d{C}a\u{g}man et al. defined two new types of group
actions on a soft set, called group $SI$-action and group $SU$-action \cite%
{cag-fil}, which are based on the inclusion relation and the intersection of
sets and union of sets, respectively.Recently, Sezgin described a new view
of ring through soft intersection properties and discussed some fundamental
results \cite{sezgin}. The concept of soft equality and some related
properties are derived by Qin and Hong

In recent times, researches have contributed a lot towards fuzzification of
soft set theory. Maji et al. \cite{6} introduced some properties of fuzzy
soft set. These results were further revised and improved by Ahmad and
Kharal \cite{1}. This theory has proven useful in many different fields such
as decision making \cite{ASDF,AAA,BBB,CCC,FFF,CEE,XGZ}. In \cite{4},Y. B.
Jun et al., introduced a new notion, called a (internal, external) cubic set
by using a fuzzy set and an interval-valued fuzzy set, and investigate
several properties. They also defined P-union, P-intersection, R-union and
R-intersection of cubic sets, and investigate several related properties.

In this paper, using a fuzzy set and an interval-valued fuzzy set, we
introduce new notions, is called (internal, external) cubic soft sets,
P-(R-)order,P-(R-)union, P-(R-)intersection and P-OR, R-OR, P-AND and R-AND
are introduced, and related properties are investigated. We show that the
P-union and the P-intersection of internal cubic soft sets are also internal
cubic soft sets. We provide conditions for the P-union (resp.
P-intersection) of two external cubic soft sets to be an internal cubic soft
set. We give conditions for the P-union (resp. R-union and R-intersection)
of two external cubic soft sets to be an external cubic soft set. We
consider conditions for the R-intersection (resp.P-intersection) of two
cubic sof sets to be both an external cubic soft set and an internal cubic
soft set.

\section{Preliminaries}

We introduce below necessary notions and present a few auxiliary results
that will be used throughout the paper. Recall first the basic terms and
definitions from the cubic set theory.

A map $\lambda :X\rightarrow \left[ 0,1\right] $ is called a fuzzy subset of 
$X.$ For any two fuzzy subsets $\lambda $ and $\mu $ of $X$, $\lambda \leq
\mu $ means that, for all $x\in X$, $\lambda (x)\leq \mu (x)$. The symbols $%
\lambda \wedge \mu $, and $\lambda \vee \mu $ will mean the following fuzzy
subsets of $X$

\begin{center}
$(\lambda \wedge \mu )(x)=\lambda (x)\wedge \mu (x)$

$(\lambda \vee \mu )(x)=\lambda (x)\vee \mu (x)$
\end{center}

for all $x\in X.$

Let $X$ be a non-empty set. A function $A$ $:X\rightarrow \lbrack I]$ is
called an interval-valued fuzzy set (shortly, an IVF set) in $X$. Let $%
[I]^{X}$ stands for the set of all IVF sets in $X$. For every $A\in \lbrack
I]^{X}$ and $x\in X$, $A(x)=[A^{-}(x),A^{+}(x)]$ is called the degree of
membership of an element $x$ to $A$, where $A^{-}:X\rightarrow I$ and $%
A^{+}:X\rightarrow I$ are fuzzy sets in $X$ which are called a lower fuzzy
set and an upper fuzzy set in $X$, respectively. For simplicity, we denote $%
A=[A^{-},A^{+}]$. For every $A,B\in \lbrack I]^{X}$, we define $A\subseteq B$
if and only if $A(x)\preceq B(x)$ for all $x\in X$. \cite{LA23}

\begin{definition}
\cite{LA23}Let $A=[A^{-},A^{+}],$ and $B=[B^{-},B^{+}]$ be two interval
valued fuzzy sets in $X$ (briefly, IVF sets). Then, we define $r\min
\{A(x),B(x)\}=[\min \{A^{-}(x),B^{-}(x)\},\min \{A^{+}(x),B^{+}(x)\}],$

$r\max \{A(x),B(x)\}=[\max \{A^{-}(x),B^{-}(x)\},\max
\{A^{+}(x),B^{+}(x)\}]. $
\end{definition}

\begin{definition}
\cite{LA23}Let $A=[A^{-},A^{+}],$ and $B=[B^{-},B^{+}]$ be two interval
valued fuzzy sets in $X$ (briefly, IVF sets). Then, we define
\textquotedblleft\ $\preceq $\textquotedblright , and \textquotedblleft\ $%
\succeq $\textquotedblright , as $A(x)\preceq B(x)$ if and only if $%
A^{-}(x)\leq B^{-}(x)$ and $A^{+}(x)\leq B^{+}(x)$. Similarly, $A(x)\succeq
B(x)$ if and only if $A^{-}(x)\geq B^{-}(x)$ and $A^{+}(x)\geq B^{+}(x)$ for
all $x\in X.$
\end{definition}

\begin{definition}
\cite{4}Let $X$ be a non-empty set. By a cubic set in $X,$ we mean a
structure $\mathcal{A}=\{<x,A(x),\lambda (x)>:x\in X\}$ in which $A$ is an
interval valued fuzzy sets in $X$ (briefly, IVF set) and $\lambda $ is a
fuzzy set in $X$. A cubic set $\mathcal{A}=\{<x,A(x),\lambda (x)>:x\in X\}$
is simply denoted by $\mathcal{A}=<A,\lambda >.$ The collection of all cubic
sets in $X$ is denoted by $CP\left( X\right) $. A cubic set $\mathcal{A}=<A,$
$\lambda >$ in which for all $x\in X$ $A(x)=0$ and $\lambda (x)=1$
(respectively $A(x)=1$ and $\lambda (x)$ $=0$) for all $x\in X$ is denoted
by $\overset{{\LARGE ..}}{0}$ (respectively $\overset{{\LARGE ..}}{1}$). A
cubic set $\mathcal{B}=<B,$ $\mu >$, in which $B(x)=0$ and $%
\mu
(x)=0$ (respectively $B(x)=1$ and $%
\mu
(x)=1$) is denoted by $\overset{\wedge }{0}$ (respectively $\overset{\wedge }%
{1}$)$.$
\end{definition}

From now on, $X$ refers to an initial universe, $E$ is a set of parameters, $%
P(X)$ is the power set of $X$ and $A,B,C\subset E$. A soft set $\lambda _{A}$
over $X$ is a set defined by%
\begin{equation*}
\lambda _{A}:E\longrightarrow P(X)\text{ such that }\lambda _{A}\left(
x\right) =\emptyset \text{ if }x\notin A.
\end{equation*}%
Here $\lambda _{A}$ is also called an approximate function. A soft set over $%
X$ can be represented by the set of ordered pairs%
\begin{equation*}
\lambda _{A}=\left\{ \left( x,\lambda _{A}\left( x\right) \right) :x\in
E,\lambda _{A}\left( x\right) \in P(X)\right\}
\end{equation*}%
It is clear to see that a soft set is a parametrized family of subsets of
the set $X$. Note that the set of all soft sets over $X$ will be denoted by $%
S(X)$ \cite{8}. Let $A$ be a set of parameters. $\{\lnot a:a\in A\}$ is
called the Not-set of $A$ denoted by $\lnot A$, where $\lnot a$ means not $a$
for each $a\in A$. It is necessary to assume that the Not-set of each set of
parameters is a subset of $E,$ where $E$ is the set of parameters. For
further notions and concepts of fuzzy set, Interval valued fuzzy set cubic
set and soft set see \cite{2,4,8,2A,op,LA24,LA23}.

\section{CUBIC SOFT SETS}

In this section we apply cubic set to soft set. We define a new extension of
fuzzy soft by using cubic set. We defined two types of cubic soft set and
related properties.

\begin{definition}
A pair $\left( \widetilde{F,}\text{ }I\right) $ is called cubic soft set
over $X$ if and only if $\widetilde{F}$ is a mapping of $I(\subseteq E)$
into the set of all cubic sets in $X,$ i.e., $\widetilde{F}:I\longrightarrow
CP\left( X\right) $ where $I$ is any subset of parameter's set $E$, $X$ is
an initial universe set and $CP\left( X\right) $ is the collection of all
cubic sets in $X$. Here we denote and define cubic soft set as $\left( 
\widetilde{F,}\text{ }I\right) =\left\{ \widetilde{F}(e_{i})=\mathcal{A}%
_{i}=\left\{ <x,\text{ }A_{e_{i}}(x),\text{ }\lambda _{e_{i}}(x)>:\text{ }%
x\in X\right\} \text{ }e_{i}\in I\right\} $ in this set corresponding to
each $e_{i}\in I$, $\mathcal{A}_{i}=\left\{ <x,\text{ }A_{e_{i}}(x),\text{ }%
\lambda _{e_{i}}(x)>:\text{ }x\in X\right\} $ is a cubic set in $X$ in which 
$A_{e_{i}}(x)$ is an interval valued fuzzy set (briefly, an IVF set) and $%
\lambda _{e_{i}}(x)$ is a fuzzy set.
\end{definition}

\begin{example}
Let $X=\left\{ p_{1},\text{ }p_{2},\text{ }p_{3},\text{ }p_{4}\right\} $ be
the set of cricket players under consideration and $E=\left\{ e_{1},\text{ }%
e_{2},\text{ }e_{3},\text{ }e_{4}\right\} $ be the set of parameters, where $%
e_{1,}$ $e_{2},$ $e_{3,\text{ }}e_{4}$ represent fitness, good current form,
good domestic cricket record and good moral character, respectively. Let $%
I=\left\{ e_{1},\text{ }e_{2},\text{ }e_{3}\right\} \subseteq E.$ Then, the
cubic soft set

$\left( \widetilde{F,}\text{ }I\right) =\left\{ \widetilde{F}(e_{i})=%
\mathcal{A}_{i\text{ }}=\left\{ <p,\text{ }A_{_{e_{i}}\text{ }}(p),\text{ }%
\lambda _{e_{i}\text{ }}(p)>:\text{ }p\in X\right\} \text{ }e_{i}\in
I,i=1,2,3.\right\} $ in $X$ is%
\begin{equation*}
\begin{tabular}{|l|l|l|l|}
\hline
$p$ & $%
\begin{array}{c}
\widetilde{F}(e_{1})=\mathcal{A}_{1}= \\ 
<A_{e_{1}}(p),\text{ \ \ \ \ \ }\lambda _{e_{1}}(p)>%
\end{array}%
$ & $%
\begin{array}{c}
\widetilde{F}(e_{2})=\mathcal{A}_{2}= \\ 
<A_{e_{2}}(p),\text{ \ \ \ \ \ }\lambda _{e_{2}}(p)>%
\end{array}%
$ & $%
\begin{array}{c}
\widetilde{F}(e_{3})=\mathcal{A}_{3}= \\ 
<A_{e_{3}}(p),\text{ \ \ \ \ \ }\lambda _{e_{3}}(p)>%
\end{array}%
$ \\ \hline
$p_{1}$ & $\ \left[ 0.3,0.7\right] ,\text{\ \ \ \ \ \ \ }0.5$ & $\ \left[
0.6,1\right] ,\text{ \ \ \ \ \ \ \ \ \ \ }0.3$ & $\ \left[ 0.3,0.7\right] ,$
\ \ \ \ \ \ \ $0.5$ \\ \hline
$p_{2}$ & $\ \left[ 0.4,0.8\right] ,\text{ \ \ \ \ \ \ }0.2$ & $\ \left[
0.1,0.6\right] ,\text{ \ \ \ \ \ \ \ }0.6$ & $\ \left[ 0.5,0.9\right] ,\text{
\ \ \ \ \ \ \ }0.9$ \\ \hline
$p_{3}$ & $\ \left[ 0.1,0.6\right] ,\text{ \ \ \ \ \ \ }0.6$ & $\ \left[
0.3,0.7\right] ,\text{ \ \ \ \ \ \ \ \ \ }1$ & $\ \left[ 0.4,8\right] ,\text{
\ \ \ \ \ \ \ \ \ }0.7$ \\ \hline
$p_{4}$ & $\ \left[ 0.5,0.9\right] ,\text{ \ \ \ \ \ \ }0.5$ & $\ \left[
0.2,0.6\right] ,\text{ \ \ \ \ \ \ \ }0.7$ & $\ \left[ 0.6,1\right] ,\text{\
\ \ \ \ \ \ \ \ \ }0.9$ \\ \hline
\end{tabular}%
\end{equation*}
\end{example}

\subsection{Types of cubic soft sets}

\subsubsection{Internal cubic soft set (ICSS)}

\begin{definition}
A cubic soft set $\left( \widetilde{F,}\text{ }I\right) $ is said to be an
internal cubic soft set (ICSS), if for all $e_{i}\in I\subseteq E$ ($E$ is
set of parameters) $F(e_{i})=\mathcal{A}_{i}$ is so that $%
A_{e_{i}}^{-}(x)\leq \lambda _{e_{i}}(x)\leq A_{e_{i}}^{+}(x)$ for all $%
e_{i}\in I$ and for all $x\in X$.
\end{definition}

\begin{example}
Let $X=\left\{ p_{1},\text{ }p_{2},\text{ }p_{3},\text{ }p_{4}\right\} $ be
the set of cricket players under consideration and $E=\left\{ e_{1},\text{ }%
e_{2},\text{ }e_{3},\text{ }e_{4}\right\} $ be the set of parameters, where $%
e_{1,}$ $e_{2},$ $e_{3,\text{ }}e_{4}$ represent fitness, good current form,
good domestic cricket record and good moral character, respectively. Let $%
I=\left\{ e_{1},\text{ }e_{2},\text{ }e_{3}\right\} \subseteq E.$ Then, the
cubic soft set

$\left( \widetilde{F,}\text{ }I\right) =\left\{ \widetilde{F}(e_{i})=%
\mathcal{A}_{i\text{ }}=\left\{ <p,\text{ }A_{_{e_{i}}\text{ }}(p),\text{ }%
\lambda _{e_{i}\text{ }}(p)>:\text{ }p\in X\right\} \text{ }e_{i}\in I,\text{
}i=1,2,3.\right\} $ in $X$ is%
\begin{equation*}
\begin{tabular}{|l|l|l|l|}
\hline
$p$ & $%
\begin{array}{c}
\widetilde{F}(e_{1})=\mathcal{A}_{1}= \\ 
<A_{e_{1}}(p),\text{ \ \ \ \ \ }\lambda _{e_{1}}(p)>%
\end{array}%
$ & $%
\begin{array}{c}
\widetilde{F}(e_{2})=\mathcal{A}_{2}= \\ 
<A_{e_{2}}(p),\text{ \ \ \ \ \ }\lambda _{e_{2}}(p)>%
\end{array}%
$ & $%
\begin{array}{c}
\widetilde{F}(e_{3})=\mathcal{A}_{3}= \\ 
<A_{e_{3}}(p),\text{ \ \ \ \ \ }\lambda _{e_{3}}(p)>%
\end{array}%
$ \\ \hline
$p_{1}$ & $\ \left[ 0.2,0.5\right] ,\text{\ \ \ \ \ \ \ }0.35$ & $\ \left[
0.1,0.4\right] ,\text{\ \ \ \ \ \ \ \ }0.3$ & $\ \left[ 0.4,0.7\right] ,%
\text{\ \ \ \ \ \ \ \ }0.5$ \\ \hline
$p_{2}$ & $\ \left[ 0.5,0.8\right] ,\text{\ \ \ \ \ \ \ }0.65$ & $\ \left[
0.3,0.6\right] ,\text{\ \ \ \ \ \ \ \ }0.5$ & $\ \left[ 0.2,0.5\right] ,%
\text{\ \ \ \ \ \ \ \ }0.3$ \\ \hline
$p_{3}$ & $\ \left[ 0.4,0.7\right] ,\text{\ \ \ \ \ \ \ }0.6$ & $\ \left[
0.3,0.6\right] ,\text{\ \ \ \ \ \ \ \ }0.5$ & $\left[ 0.5,0.8\right] ,\text{%
\ \ \ \ \ \ \ \ }0.6$ \\ \hline
$p_{4}$ & $\ \left[ 0.3,0.6\right] ,\text{\ \ \ \ \ \ \ }0.3$ & $\ \left[
0.7,1\right] ,\text{\ \ \ \ \ \ \ \ \ \ \ }0.85$ & $\ \left[ 0.6,0.9\right] ,%
\text{\ \ \ \ \ \ \ \ }0.75$ \\ \hline
\end{tabular}%
\end{equation*}%
an internal cubic soft set (ICSS) in $X$.
\end{example}

\subsubsection{External Cubic Soft sets}

\begin{definition}
A cubic soft set $\left( \widetilde{F,}\text{ }I\right) $ is said to be an
external cubic soft set if for all $e_{i}\in I\subseteq E$ ($E$ is set of
parameters), $F(e_{i})=\mathcal{A}_{i}$ is so $\lambda _{e_{i}\text{ }%
}(x)\notin (A_{e_{i}}^{-}(x),A_{e_{i}}^{+}(x))$ for all $e_{i}\in I$ and for
all $x\in X$.
\end{definition}

\begin{example}
Let $X=\left\{ p_{1},\text{ }p_{2},\text{ }p_{3},\text{ }p_{4}\right\} $ be
the set of cricket players under consideration and $E=\left\{ e_{1},\text{ }%
e_{2},\text{ }e_{3},\text{ }e_{4}\right\} $ be the set of parameters, where $%
e_{1,}$ $e_{2},$ $e_{3,\text{ }}e_{4}$ represent fitness, good current form,
good domestic cricket record and good moral character, respectively. Let $%
I=\left\{ e_{1},\text{ }e_{2},\text{ }e_{3}\right\} \subseteq E.$ Then, the
cubic soft set

$\left( \widetilde{F,}\text{ }I\right) =\left\{ \widetilde{F}(e_{i})=%
\mathcal{A}_{i\text{ }}=\left\{ <p,\text{ }A_{_{e_{i}}\text{ }}(p),\text{ }%
\lambda _{e_{i}\text{ }}(p)>:\text{ }p\in X\right\} \text{ }e_{i}\in I,\text{
}i=1,2,3.\right\} $ in $X$ is%
\begin{equation*}
\begin{tabular}{|l|l|l|l|}
\hline
$p$ & $%
\begin{array}{c}
\widetilde{F}(e_{1})=\mathcal{A}_{1}= \\ 
<A_{e_{1}}(p),\text{ \ \ \ \ \ }\lambda _{e_{1}}(p)>%
\end{array}%
$ & $%
\begin{array}{c}
\widetilde{F}(e_{2})=\mathcal{A}_{2}= \\ 
<A_{e_{2}}(p),\text{ \ \ \ \ \ }\lambda _{e_{2}}(p)>%
\end{array}%
$ & $%
\begin{array}{c}
\widetilde{F}(e_{3})=\mathcal{A}_{3}= \\ 
<A_{e_{3}}(p),\text{ \ \ \ \ \ }\lambda _{e_{3}}(p)>%
\end{array}%
$ \\ \hline
$p_{1}$ & $\ \left[ 0.1,0.5\right] ,\text{\ \ \ \ \ \ \ \ }0.1$ & $\ \left[
0.2,0.6\right] ,\text{\ \ \ \ \ \ \ \ }0.6$ & $\ \left[ 0.5,0.9\right] ,%
\text{\ \ \ \ \ \ \ }0.4$ \\ \hline
$p_{2}$ & $\ \left[ 0.2,0.6\right] ,\text{\ \ \ \ \ \ \ \ }0.6$ & $\ \left[
0.3,0.6\right] ,\text{\ \ \ \ \ \ \ \ }0.75$ & $\ \left[ 0.6,1\right] ,\text{%
\ \ \ \ \ \ \ \ \ \ \ \ }1$ \\ \hline
$p_{3}$ & $\ \left[ 0.3,0.7\right] ,\text{\ \ \ \ \ \ \ \ }0.2$ & $\ \left[
0.4,0.8\right] ,\text{\ \ \ \ \ \ \ \ }0.9$ & $\ \left[ 0.5,0.9\right] ,%
\text{\ \ \ \ \ \ \ }0.9$ \\ \hline
$p_{4}$ & $\ \left[ 0.5,0.9\right] ,\text{\ \ \ \ \ \ \ \ }0.4$ & $\ \left[
0.2,0.6\right] ,\text{\ \ \ \ \ \ \ \ }0.65$ & $\ \left[ 0.4,0.8\right] ,%
\text{\ \ \ \ \ \ \ }0.4$ \\ \hline
\end{tabular}%
\end{equation*}%
an external cubic soft set (ECSS) in $X$.
\end{example}

\begin{theorem}
Let $\left( \widetilde{F,}\text{ }I\right) =\left\{ \widetilde{F}%
(e_{i})=\left\{ <x,\text{ }A_{e_{i}}(x),\text{ }\lambda _{e_{i}}(x)>:\text{ }%
x\in X\right\} \text{ }e_{i}\in I\right\} $ be a cubic soft set in $X$ which
is not an ECSS. Then, there exists at least one $e_{i}\in I$ for which there
exists some $x\in X$ such that $\lambda _{e_{i}\text{ }}(x)\in
(A_{e_{i}}^{-}(x),A_{e_{i}}^{+}(x)).$
\end{theorem}

\begin{proof}
By definition of an external cubic soft set (ECSS) we know that $\lambda
_{e_{i}\text{ }}(x)\notin (A_{e_{i}}^{-}(x),A_{e_{i}}^{+}(x))$ for all $x\in
X$ corresponding to each $e_{i}\in I$. But given that $\left( \widetilde{F,}%
\text{ }I\right) $ is not an ECSS so for at least one $e_{i}\in I$ there
exists some $x\in X$ such that $\lambda _{e_{i}}(x)\in
(A_{e_{i}}^{-}(x),A_{e_{i}}^{+}(x)).$ Hence the result.
\end{proof}

\begin{theorem}
Let $\left( \widetilde{F,}\text{ }I\right) =\left\{ \widetilde{F}%
(e_{i})=\left\{ <x,\text{ }A_{e_{i}}(x),\text{ }\lambda _{e_{i}}(x)>:\text{ }%
x\in X\right\} \text{ }e_{i}\in I\right\} $ be a cubic soft set in $X.$ If $%
\left( \widetilde{F,}\text{ }I\right) $ is both an ICSS and ECSS in $X$,
then for all $x\in X$ corresponding to each $e_{i}\in I,$ $\lambda
_{e_{i}}(x)\in (U(A_{e_{i}})\cup L(A_{e_{i}})),$ where $U(A_{e_{i}})=\left\{
A_{e_{i}}^{+}(x):x\in X\right\} $ and $L(A_{e_{i}})=\left\{
A_{e_{i}}^{-}(x):x\in X\right\} .$
\end{theorem}

\begin{proof}
Let $\left( \widetilde{F,}\text{ }I\right) =\left\{ \widetilde{F}(e_{i})=%
\mathcal{A}_{i\text{ }}=\left\{ <x,\text{ }A_{e_{i}\text{ }}(x),\text{ }%
\lambda _{e_{i}\text{ }}(x)>:\text{ }x\in X\right\} \text{ }e_{i}\in
I\right\} $ is both an ICSS and ECSS, then by definition of ICSS
corresponding to each $e_{i}\in I$ and for all $x\in X$ we have $\lambda
_{e_{i}\text{ }}(x)\in (A_{e_{i}}^{-}(x),A_{e_{i}}^{+}(x)).$ By definition
of ECSS corresponding to each $e_{i}\in I$ and for all $x\in X$ we have\ $%
\lambda _{e_{i}}(x)\notin (A_{e_{i}}^{-}(x),A_{e_{i}}^{+}(x)).$ Since $%
\left( \widetilde{F,}\text{ }I\right) $ is both an ICSS and ECSS, so the
only possibility is that $\lambda _{e_{i}\text{ }}(x)=A_{e_{i}}^{-}(x)$ or $%
\lambda _{e_{i}\text{ }}(x)=A_{e_{i}}^{+}(x)$ for all $x\in X$ corresponding
to each $e_{i}\in I$. Hence $\lambda _{e_{i}\text{ }}(x)\in
(U(A_{e_{i}})\cup L(A_{e_{i}}))$ for all $e_{i}\in I$ and for all $x\in X$.
\end{proof}

\begin{definition}
Let $\left( \widetilde{F,}\text{ }I\right) =\left\{ \widetilde{F}(e_{i})=%
\mathcal{A}_{i\text{ }}=\left\{ <x,\text{ }A_{e_{i}}(x),\text{ }\lambda
_{e_{i}}(x)>:x\in X\right\} \text{ }e_{i}\in I\right\} $ and $\left( 
\widetilde{G,}\text{ }J\right) =\left\{ \widetilde{G}(e_{i})=\mathcal{B}%
_{i}=\left\{ <x,\text{ }B_{e_{i}}(x),\text{ }\mu _{e_{i}}(x)>:x\in X\right\} 
\text{ }e_{i}\in J\right\} $ be two cubic soft sets in $X,$ $I$ and $J$ are
any subsets of $E$ (set of parameters). Then we have the following

\begin{enumerate}
\item[1] $\left( \widetilde{F,}\text{ }I\right) =\left( \widetilde{G,}\text{ 
}J\right) $ if and only if\ the following\ conditions are satisfied
\end{enumerate}

(a) $I=J$ and

(b) $\widetilde{F}(e_{i})=\widetilde{G}(e_{i})$ for all $e_{i}\in I$ if and
only if $A_{e_{i}}(x)=B_{e_{i}}(x)$ and $\lambda _{e_{i}}(x)=\mu _{e_{i}}(x)$
for all $e_{i}\in I.$

\begin{enumerate}
\item[2] If $\left( \widetilde{F,}\text{ }I\right) $ and $\left( \widetilde{%
G,}\text{ }J\right) $ are two cubic soft sets then we denote and define\
P-order as$\left( \widetilde{F,}\text{ }I\right) \subseteq _{p}\left( 
\widetilde{G,}\text{ }J\right) $ if and only if\ the following\ conditions
are satisfied
\end{enumerate}

(a) $I\subseteq J$ and

(b) $\widetilde{F}(e_{i})\leq _{P}\widetilde{G}(e_{i})$ for all $e_{i}\in I$
if and only if $A_{e_{i}}(x)\preccurlyeq B_{e_{i}\text{ }}(x)$ and $\lambda
_{e_{i}}(x)\leq \mu _{e_{i}}(x)$ for all $x\in X$ corresponding to each $%
e_{i}\in I.$

\begin{enumerate}
\item[3] If $\left( \widetilde{F,}\text{ }I\right) $ and $\left( \widetilde{%
G,}\text{ }J\right) $ are two cubic soft sets then we denote and define\
R-order as $\left( \widetilde{F,}\text{ }I\right) \subseteq _{R}\left( 
\widetilde{G,}\text{ }J\right) $ if and only if the following\ conditions
are satisfied:
\end{enumerate}

(a) $I\subseteq J$ and

(b) $\widetilde{F}(e_{i})\leq _{R}\widetilde{G}(e_{i})$ for all $e_{i}\in I$
if and only if $A_{e_{i}}(x)\preccurlyeq B_{e_{i}}(x)$ and $\lambda
_{e_{i}}(x)\geq \mu _{e_{i}}(x)$ for all $x\in X$ corresponding to each $%
e_{i}\in I$ .
\end{definition}

\begin{example}
(P-order)

Let $\left( \widetilde{F,}\text{ }I\right) $ and $\left( \widetilde{G,}\text{
}J\right) $ be the cubic soft sets in $X$ defined as follows,

$\left( \widetilde{F,}\text{ }I\right) =\left\{ \widetilde{F}(e_{i})=%
\mathcal{A}_{i}=\left\{ <x,\text{ }A_{i}(p),\text{ }\lambda _{i}(p)>:\text{ }%
p\in X\right\} \text{ }e_{i}\in J,i=1,2,3\right\} $%
\begin{equation*}
\begin{tabular}{|l|l|l|l|}
\hline
$p$ & $%
\begin{array}{c}
\widetilde{F}(e_{1})=\mathcal{A}_{1}= \\ 
<A_{e_{1}}(p),\text{ \ \ \ \ \ }\lambda _{e_{1}}(p)>%
\end{array}%
$ & $%
\begin{array}{c}
\widetilde{F}(e_{2})=\mathcal{A}_{2}= \\ 
<A_{e_{2}}(p),\text{ \ \ \ \ \ }\lambda _{e_{2}}(p)>%
\end{array}%
$ & $%
\begin{array}{c}
\widetilde{F}(e_{3})=\mathcal{A}_{3}= \\ 
<A_{e_{3}}(p),\text{ \ \ \ \ \ }\lambda _{e_{3}}(p)>%
\end{array}%
$ \\ \hline
$p_{1}$ & $\ \left[ 0.2,0.5\right] ,\text{ \ \ \ \ \ }0.7$ & $\ \left[
0.1,0.4\right] ,\text{ \ \ \ \ \ \ \ }0.2$ & $\ \left[ 0.4,0.7\right] ,\text{
\ \ \ \ \ \ }0.5$ \\ \hline
$p_{2}$ & $\ \left[ 0.3,0.6\right] ,\text{ \ \ \ \ \ }0.5\text{\ }$ & $\ %
\left[ 0.6,0.9\right] ,\text{ \ \ \ \ \ \ \ }0.4$ & $\ \left[ 0.6,0.9\right]
,\text{ \ \ \ \ \ \ }0.7$ \\ \hline
$p_{3}$ & $\ \left[ 0.4,0.7\right] ,\text{ \ \ \ \ \ }0.5$ & $\ \left[ 0.7,1%
\right] ,\text{ \ \ \ \ \ \ \ \ \ }0.8$ & $\ \left[ 0.3,0.6\right] ,\text{ \
\ \ \ \ \ }0.5$ \\ \hline
$p_{4}$ & $\ \left[ 0.5,0.8\right] ,\text{ \ \ \ \ \ }0.6\text{\ }$ & $\ %
\left[ 0.3,0.6\right] ,\text{ \ \ \ \ \ \ \ }0.5$ & $\ \left[ 0.2,0.5\right]
,\text{ \ \ \ \ \ \ }0.7$ \\ \hline
\end{tabular}%
\end{equation*}%
and $\left( \widetilde{G,}\text{ }J\right) =\left\{ \widetilde{G}(e_{i})=%
\mathcal{B}_{i}=\left\{ <x,\text{ }B_{e_{i}}(p),\ \mu _{e_{i}}(p)>:\text{ }%
p\in X\right\} \text{ }e_{i}\in J,i=1,2,3\right\} $ is%
\begin{equation*}
\begin{tabular}{|l|l|l|l|}
\hline
$p$ & $%
\begin{array}{c}
\widetilde{G}(e_{1})=\mathcal{B}_{1}= \\ 
<B_{e_{1}}(p),\text{ \ \ \ \ }\ \mu _{e_{1}}(p)>%
\end{array}%
$ & $%
\begin{array}{c}
\widetilde{G}(e_{2})=\mathcal{B}_{2}= \\ 
<B_{e_{2}}(p),\text{ \ \ \ \ }\ \mu _{e_{2}}(p)>%
\end{array}%
$ & $%
\begin{array}{c}
\widetilde{G}(e_{3})=\mathcal{B}_{3}= \\ 
<B_{e_{3}}(p),\text{ \ \ \ \ }\ \mu _{e_{3}}(p)>%
\end{array}%
$ \\ \hline
$p_{1}$ & $\ \left[ 0.3,0.6\right] ,\text{ \ \ \ \ \ }0.8$ & $\ \left[
0.2,0.5\right] ,\text{ \ \ \ \ \ \ \ }0.3$ & $\ \left[ 0.5,0.8\right] ,\text{
\ \ \ \ \ \ }0.6$ \\ \hline
$p_{2}$ & $\ \left[ 0.4,0.7\right] ,\text{ \ \ \ \ \ }0.6$ & $\ \left[
0.6,0.9\right] ,\text{ \ \ \ \ \ \ \ }0.5$ & $\ \left[ 0.7,0.9\right] ,\text{
\ \ \ \ \ \ }0.8$ \\ \hline
$p_{3}$ & $\ \left[ 0.5,0.8\right] ,\text{ \ \ \ \ \ }0.5\text{\ }$ & $\ %
\left[ 0.7,1\right] ,\text{ \ \ \ \ \ \ \ \ \ }0.8$ & $\ \left[ 0.4,0.7%
\right] ,\text{ \ \ \ \ \ \ }0.6$ \\ \hline
$p_{4}$ & $\ \left[ 0.6,0.9\right] ,\text{ \ \ \ \ \ }0.7\text{\ }$ & $\ %
\left[ 0.4,0.7\right] ,\text{ \ \ \ \ \ \ \ }0.6$ & $\ \left[ 0.3,0.6\right]
,\text{ \ \ \ \ \ \ }0.7$ \\ \hline
\end{tabular}%
\end{equation*}%
Thus clearly we have that $\left( \widetilde{F,}\text{ }I\right) \subseteq
_{p}\left( \widetilde{G,}\text{ }J\right) .$
\end{example}

\begin{example}
(R-order)

Let $\left( \widetilde{F,}\text{ }I\right) $ and $\left( \widetilde{G,}\text{
}J\right) $ be the cubic soft sets in $X$ defined as follows,

$\left( \widetilde{F,}\text{ }I\right) =$%
\begin{equation*}
\begin{tabular}{|l|l|l|l|}
\hline
$p$ & $%
\begin{array}{c}
\widetilde{F}(e_{1})=\mathcal{A}_{1}= \\ 
<A_{e_{1}}(p),\text{ \ \ \ \ \ }\lambda _{e_{1}}(p)>%
\end{array}%
$ & $%
\begin{array}{c}
\widetilde{F}(e_{2})=\mathcal{A}_{2}= \\ 
<A_{e_{2}}(p),\text{ \ \ \ \ \ }\lambda _{e_{2}}(p)>%
\end{array}%
$ & $%
\begin{array}{c}
\widetilde{F}(e_{3})=\mathcal{A}_{3}= \\ 
<A_{e_{3}}(p),\text{ \ \ \ \ \ }\lambda _{e_{3}}(p)>%
\end{array}%
$ \\ \hline
$p_{1}$ & $\ \left[ 0.3,0.6\right] ,\text{ \ \ \ \ \ }0.7$ & $\ \left[
0.1,0.4\right] ,\text{ \ \ \ \ \ \ \ }0.6$ & $\ \left[ 0.2,0.5\right] ,\text{
\ \ \ \ \ \ }0.3$ \\ \hline
$p_{2}$ & $\ \left[ 0.4,0.7\right] ,\text{ \ \ \ \ \ }0.3$ & $\ \left[
0.2,0.5\right] ,\text{ \ \ \ \ \ \ \ }0.2$ & $\ \left[ 0.6,0.9\right] ,\text{
\ \ \ \ \ \ \ \ \ }1$ \\ \hline
$p_{3}$ & $\ \left[ 0.5,0.8\right] ,\text{ \ \ \ \ \ }0.4$ & $\ \left[
0.6,0.9\right] ,\text{ \ \ \ \ \ \ \ }0.2$ & $\ \left[ 0.5,0.8\right] ,\text{
\ \ \ \ \ \ \ }0.7$ \\ \hline
$p_{4}$ & $\ \left[ 0.6,0.9\right] ,\text{ \ \ \ \ \ }0.5$ & $\ \left[
0.4,0.7\right] ,\text{ \ \ \ \ \ \ \ }0.6$ & $\ \left[ 0.1,0.4\right] ,\text{
\ \ \ \ \ \ \ }0.9$ \\ \hline
\end{tabular}%
\end{equation*}%
and $\left( \widetilde{G,}\text{ }J\right) =\left\{ \widetilde{G}(e_{i})=%
\mathcal{B}_{i}=\left\{ <p,\text{ }B_{e_{i}}(p),\ \mu _{e_{i}}(p)>:\text{ }%
p\in X\right\} \text{ }e_{i}\in J,i=1,2,3\right\} $ is%
\begin{equation*}
\begin{tabular}{|l|l|l|l|}
\hline
$p$ & $%
\begin{array}{c}
\widetilde{G}(e_{1})=\mathcal{B}_{1}= \\ 
<B_{e_{1}}(p),\text{ \ \ \ \ }\ \mu _{e_{1}}(p)>%
\end{array}%
$ & $%
\begin{array}{c}
\widetilde{G}(e_{2})=\mathcal{B}_{2}= \\ 
<B_{e_{2}}(p),\text{ \ \ \ \ }\ \mu _{e_{2}}(p)>%
\end{array}%
$ & $%
\begin{array}{c}
\widetilde{G}(e_{3})=\mathcal{B}_{3}= \\ 
<B_{e_{3}}(p),\text{ \ \ \ \ }\ \mu _{e_{3}}(p)>%
\end{array}%
$ \\ \hline
$p_{1}$ & $\ \left[ 0.4,0.7\right] ,\text{ \ \ \ \ \ }0.6$ & $\ \left[
0.2,0.5\right] ,\text{ \ \ \ \ \ \ \ }0.5$ & $\ \left[ 0.3,0.6\right] ,\text{
\ \ \ \ \ \ }0.2$ \\ \hline
$p_{2}$ & $\ \left[ 0.5,0.8\right] ,\text{ \ \ \ \ \ }0.2$ & $\ \left[
0.3,0.6\right] ,\text{ \ \ \ \ \ \ \ }0.1$ & $\ \left[ 0.6,0.9\right] ,\text{
\ \ \ \ \ \ }0.6$ \\ \hline
$p_{3}$ & $\ \left[ 0.6,0.9\right] ,\text{ \ \ \ \ \ }0.3$ & $\ \left[
0.7,0.9\right] ,\text{ \ \ \ \ \ \ \ }0.2$ & $\ \left[ 0.6,0.9\right] ,\text{
\ \ \ \ \ \ }0.5$ \\ \hline
$p_{4}$ & $\ \left[ 0.7,0.9\right] ,\text{ \ \ \ \ \ }0.4$ & $\ \left[
0.5,0.8\right] ,\text{ \ \ \ \ \ \ \ }0.5$ & $\ \left[ 0.4,0.7\right] ,\text{
\ \ \ \ \ \ }0.9$ \\ \hline
\end{tabular}%
\end{equation*}%
Thus clearly we have that $\left( \widetilde{F,}\text{ }I\right) \subseteq
_{R}\left( \widetilde{G,}\text{ }J\right) .$
\end{example}

\subsubsection{P-union, P-intersection, R-union and R-intersection of any
two Cubic soft sets in $X$}

\begin{definition}
Let $\left( \widetilde{F,}\text{ }I\right) $ and $\left( \widetilde{G,}\text{
}J\right) $ be two cubic soft sets in $X,$ where $I$ and $J$ are any subsets
of parameter's set $E.$ Then, we define \emph{P-union} as $\left( \widetilde{%
F,}\text{ }I\right) \cup _{P}\left( \widetilde{G,}\text{ }J\right) =\left( 
\widetilde{H,}\text{ }C\right) ,$ where $C=I\cup J$ and%
\begin{equation*}
\widetilde{H}(e_{i})=\left\{ 
\begin{array}{c}
\widetilde{F}(e_{i})\text{ \ \ \ \ \ \ \ \ \ \ \ \ \ \ \ \ If }e_{i}\in I-J,
\\ 
\widetilde{G}(e_{i})\text{ \ \ \ \ \ \ \ \ \ \ \ \ \ \ \ \ If }e_{i}\in J-I,
\\ 
\widetilde{F}(e_{i})\vee _{P}\widetilde{G}(e_{i})\text{ \ \ \ If }e_{i}\in
I\cap J\text{,}%
\end{array}%
\right.
\end{equation*}%
where $\widetilde{F}(e_{i})\vee _{P}\widetilde{G}(e_{i})$ is defined as 
\begin{equation*}
\widetilde{F}(e_{i})\vee _{P}\widetilde{G}(e_{i})=\left\{ <x,\text{ }r\max
\{A_{e_{i}}(x),B_{e_{i}\text{ }}(x)\},(\text{ }\lambda _{e_{i}\text{ }}\vee
\mu _{_{e_{i}}\text{ }})(x)>:\text{ }x\in X\right\} ,e_{i}\in I\cap J
\end{equation*}%
\emph{P-intersection} as $\left( \widetilde{F,}\text{ }I\right) \cap
_{P}\left( \widetilde{G,}\text{ }J\right) =\left( \widetilde{H,}\text{ }%
C\right) ,$ where $C=I\cap J$, $H(e_{i})=\widetilde{F}(e_{i})\wedge _{P}%
\widetilde{G}(e_{i})$ and $e_{i}\in I\cap J.$ Here $\widetilde{F}%
(e_{i})\wedge _{P}\widetilde{G}(e_{i})$ is defined as $\widetilde{F}%
(e_{i})\wedge _{P}\widetilde{G}(e_{i})=\widetilde{H}(e_{i})=\left\{ <x,r\min
\{A_{e_{i}}(x),B_{e_{i}}(x)\},(\lambda _{e_{i}}\wedge \mu _{e_{i}})(x)>:x\in
X\text{, }e_{i}\in I\cap J\right\} $.

\emph{R-union} as $\left( \widetilde{F,}\text{ }I\right) \cup _{R}\left( 
\widetilde{G,}\text{ }J\right) =\left( \widetilde{H,}\text{ }C\right) ,$
where $C=I\cup J$ and%
\begin{equation*}
H(e_{i})=\left\{ 
\begin{array}{c}
\widetilde{F}(e_{i})\text{ \ \ \ \ \ \ \ \ \ \ \ \ \ \ \ \ If }e_{i}\in I-J
\\ 
\widetilde{G}(e_{i})\text{ \ \ \ \ \ \ \ \ \ \ \ \ \ \ \ \ If }e_{i}\in J-I
\\ 
\widetilde{F}(e_{i})\vee _{R}\widetilde{G}(e_{i})\text{ \ \ \ If }e_{i}\in
I\cap J\text{\ }%
\end{array}%
\right.
\end{equation*}
\end{definition}

\begin{example}
Here $\widetilde{F}(e_{i})\vee _{R}\widetilde{G}(e_{i})$ is defined as 
\begin{equation*}
\widetilde{F}(e_{i})\vee _{R}\widetilde{G}(e_{i})=\left\{ <x,\text{ }r\max
\{A_{e_{i}}(x),B_{e_{i}}(p)\},(\lambda _{e_{i}}\wedge \mu _{e_{i}})(x)>:%
\text{ }x\in X\right\} \text{, }e_{i}\in I\cap J
\end{equation*}%
\emph{R-intersection} as $\left( \widetilde{F,}\text{ }I\right) \cap
_{R}\left( \widetilde{G,}\text{ }J\right) =\left( \widetilde{H,}\text{ }%
C\right) ,$\ where $C=I\cap J,$ $H(e_{i})=\widetilde{F}(e_{i})\wedge _{R}%
\widetilde{G}(e_{i})$ and $e_{i}\in I\cap J$. Here $\ \widetilde{F}%
(e_{i})\wedge _{R}\widetilde{G}(e_{i})$ is defined as%
\begin{equation*}
\widetilde{F}(e_{i})\wedge _{R}\widetilde{G}(e_{i})=\widetilde{H}%
(e_{i})=\left\{ <x,\text{ }r\min \{A_{e_{i}}(x),B_{e_{i}}(x)\},(\lambda
_{e_{i}}\vee \mu _{e_{i}})(x)>:\text{ }x\in X\text{ }e_{i}\in I\cap J\right\}
\end{equation*}
\end{example}

\begin{enumerate}
\item 
\begin{example}
(P-union) Let $X=\left\{ p_{1},\text{ }p_{2},\text{ }p_{3},\text{ }%
p_{4}\right\} $ be initial universe, $I=\left\{ e_{1},\text{ }e_{2}\right\} $
and $J=\left\{ e_{1},\text{ }e_{2},\text{ }e_{3}\right\} $ are any subsets
of parameter's set $E=\left\{ e_{1},\text{ }e_{2},\text{ }e_{3}\right\} .$
Let $\left( \widetilde{F,}\text{ }I\right) $ be cubic soft set define as: 
\begin{equation*}
\begin{tabular}{|l|l|l|}
\hline
$p$ & $%
\begin{array}{c}
\widetilde{F}(e_{1})=\mathcal{A}_{1}= \\ 
<A_{e_{1}}(p),\text{ \ \ \ \ \ }\lambda _{e_{1}}(p)>%
\end{array}%
$ & $%
\begin{array}{c}
\widetilde{F}(e_{2})=\mathcal{A}_{2}= \\ 
<A_{e_{2}}(p),\text{ \ \ \ \ \ }\lambda _{e_{2}}(p)>%
\end{array}%
$ \\ \hline
$p_{1}$ & $\ \left[ 0.4,0.7\right] ,\text{ \ \ \ \ \ \ }0.3$ & $\ \left[
0.3,0.6\right] ,\text{ \ \ \ \ \ \ }0.7$ \\ \hline
$p_{2}$ & $\ \left[ 0.3,0.6\right] ,\text{ \ \ \ \ \ \ }0.8$ & $\ \left[
0.7,1\right] ,\text{ \ \ \ \ \ \ \ \ }0.8$ \\ \hline
$p_{3}$ & $\ \left[ 0.5,0.8\right] ,\text{ \ \ \ \ \ \ }0.7$ & $\ \left[
0.3,0.6\right] ,\text{ \ \ \ \ \ \ }0.6$ \\ \hline
$p_{4}$ & $\ \left[ 0.6,0.9\right] ,\text{ \ \ \ \ \ \ }0.5$ & $\ \left[
0.1,0.4\right] ,\text{ \ \ \ \ \ \ }0.9$ \\ \hline
\end{tabular}%
\end{equation*}%
Let $\left( \widetilde{G,}\text{ }J\right) =\widetilde{G}(e_{i})=\mathcal{B}%
_{i}=\left\{ <p,\text{ }B_{e_{i}}(p),\ \mu _{e_{i}}(p)>:\text{ }p\in X\text{%
, }e_{i}\in J,i=1,2,3\right\} $ be a cubic soft set over $X$%
\begin{equation*}
\begin{tabular}{|l|l|l|l|}
\hline
$p$ & $%
\begin{array}{c}
\widetilde{G}(e_{1})=\mathcal{B}_{1}= \\ 
<B_{e_{1}}(p),\text{ \ \ \ \ \ }\ \mu _{e_{1}}(p)>%
\end{array}%
$ & $%
\begin{array}{c}
\widetilde{G}(e_{2})=\mathcal{B}_{2}= \\ 
<B_{e_{2}}(p),\text{ \ \ \ \ }\ \mu _{e_{2}}(p)>%
\end{array}%
$ & $%
\begin{array}{c}
\widetilde{G}(e_{3})=\mathcal{B}_{3}= \\ 
<B_{e_{3}}(p),\text{ \ \ \ \ }\ \mu _{e_{3}}(p)>%
\end{array}%
$ \\ \hline
$p_{1}$ & $\ \left[ 0.3,0.6\right] ,\text{ \ \ \ \ \ }0.8$ & $\ \left[
0.2,0.5\right] ,\text{ \ \ \ \ \ \ \ }0.7$ & $\ \left[ 0.6,0.9\right] ,\text{
\ \ \ \ \ \ }0.7$ \\ \hline
$p_{2}$ & $\ \left[ 0.4,0.7\right] ,\text{ \ \ \ \ \ }0.6$ & $\ \left[
0.3,0.6\right] ,\text{ \ \ \ \ \ \ \ }0.8$ & $\ \left[ 0.3,0.6\right] ,\text{
\ \ \ \ \ \ }0.5$ \\ \hline
$p_{3}$ & $\ \left[ 0.5,0.8\right] ,\text{ \ \ \ \ \ }0.4$ & $\ \left[
0.4,0.7\right] ,\text{ \ \ \ \ \ \ \ }0.9$ & $\ \left[ 0.1,0.4\right] ,\text{
\ \ \ \ \ \ }0.9$ \\ \hline
$p_{4}$ & $\ \left[ 0.6,0.9\right] ,\text{ \ \ \ \ \ }0.7$ & $\ \left[
0.5,0.8\right] ,\text{ \ \ \ \ \ \ \ }0.6$ & $\ \left[ 0.4,0.7\right] ,\text{
\ \ \ \ \ \ }0.6$ \\ \hline
\end{tabular}%
\end{equation*}%
Then, $\left( \widetilde{F,}\text{ }I\right) $ $\cup _{P}\left( \widetilde{G,%
}\text{ }J\right) $ is a cubic soft set over $X$ and defined as:%
\begin{eqnarray*}
&&%
\begin{tabular}{|l|l|}
\hline
$p$ & $%
\begin{array}{c}
\widetilde{F}{\small (e}_{1}{\small )\vee }_{P}\widetilde{G}{\small (e}_{1}%
{\small )=} \\ 
<{\small r}\max \left\{ A_{e_{1}}(p){\small ,}B_{e_{1}}(p)\right\} ,{\small %
(\lambda }_{e_{1}}{\small \vee \mu }_{e_{1}}{\small )(p)>}%
\end{array}%
$ \\ \hline
$p_{1}$ & ${\small \ \ \ \ \ \ \ \ \ \ \ \ \ \ \ \ \ }\left[ {\small 0.4,0.7}%
\right] ,\text{ \ \ \ \ \ }{\small 0.8}$ \\ \hline
$p_{2}$ & ${\small \ \ \ \ \ \ \ \ \ \ \ \ \ \ \ \ \ }\left[ {\small 0.4,0.7}%
\right] ,\text{ \ \ \ \ \ }{\small 0.8}$ \\ \hline
$p_{3}$ & ${\small \ \ \ \ \ \ \ \ \ \ \ \ \ \ \ \ \ }\left[ {\small 0.5,0.8}%
\right] ,\text{ \ \ \ \ \ }{\small 0.7}$ \\ \hline
$p_{4}$ & ${\small \ \ \ \ \ \ \ \ \ \ \ \ \ \ \ \ \ }\left[ {\small 0.6,0.9}%
\right] ,\text{ \ \ \ \ \ }{\small 0.7}$ \\ \hline
\end{tabular}
\\
&&%
\begin{tabular}{|l|l|}
\hline
$p_{1}$ & $%
\begin{array}{c}
\widetilde{F}{\small (e}_{2}{\small )\vee }_{P}\widetilde{G}{\small (e}_{2}%
{\small )=} \\ 
<{\small r}\max \left\{ A_{e_{2}}(p){\small ,}B_{e_{2}}(p)\right\} ,{\small %
(\lambda }_{e_{2}}{\small \vee \mu }_{e_{2}}{\small )(p)>}%
\end{array}%
$ \\ \hline
$p_{2}$ & ${\small \ \ \ \ \ \ \ \ \ \ \ \ \ \ \ \ \ }\left[ {\small 0.3,0.6}%
\right] ,\text{ \ \ \ \ \ \ }{\small 0.7}$ \\ \hline
$p_{2}$ & ${\small \ \ \ \ \ \ \ \ \ \ \ \ \ \ \ \ \ }\left[ {\small 0.7,1}%
\right] ,\text{ \ \ \ \ \ \ \ \ }{\small 0.8}$ \\ \hline
$p_{3}$ & ${\small \ \ \ \ \ \ \ \ \ \ \ \ \ \ \ \ \ }\left[ {\small 0.4,0.7}%
\right] ,\text{ \ \ \ \ \ \ }{\small 0.9}$ \\ \hline
$p_{4}$ & ${\small \ \ \ \ \ \ \ \ \ \ \ \ \ \ \ \ \ }\left[ {\small 0.5,0.8}%
\right] ,\text{ \ \ \ \ \ \ }{\small 0.9}$ \\ \hline
\end{tabular}
\\
&&%
\begin{tabular}{|l|l|}
\hline
$p$ & $%
\begin{array}{c}
\widetilde{G}(e_{3})=\mathcal{B}_{3}= \\ 
<B_{e_{3}}(p){\small ,}\text{ \ \ \ \ \ \ \ }\lambda _{e_{3}}(p){\small >}%
\end{array}%
$ \\ \hline
$p_{1}$ & $\ \left[ {\small 0.6,0.9}\right] ,\text{ \ \ \ \ \ \ \ \ \ }%
{\small 0.7}$ \\ \hline
$p_{2}$ & $\ \left[ {\small 0.3,0.6}\right] ,\text{ \ \ \ \ \ \ \ \ \ }%
{\small 0.5}$ \\ \hline
$p_{3}$ & $\ \left[ {\small 0.1,0.4}\right] ,\text{ \ \ \ \ \ \ \ \ \ }%
{\small 0.9}$ \\ \hline
$p_{4}$ & $\ \left[ {\small 0.4,0.7}\right] ,\text{ \ \ \ \ \ \ \ \ \ }%
{\small 0.6}$ \\ \hline
\end{tabular}%
\end{eqnarray*}%
(P-intersection) Let $\left( \widetilde{F,}\text{ }I\right) $ be a cubic
soft set over $U$ and defined as: 
\begin{equation*}
\begin{tabular}{|l|l|l|l|}
\hline
$p$ & $%
\begin{array}{c}
\widetilde{F}(e_{1})=\mathcal{A}_{1}= \\ 
<A_{e_{1}}(p),\text{ \ \ \ \ \ }\lambda _{e_{1}}(p)>%
\end{array}%
$ & $%
\begin{array}{c}
\widetilde{F}(e_{2})=\mathcal{A}_{2}= \\ 
<A_{e_{2}}(p),\text{ \ \ \ \ \ }\lambda _{e_{2}}(p)>%
\end{array}%
$ & $%
\begin{array}{c}
\widetilde{F}(e_{3})=\mathcal{A}_{3}= \\ 
<A_{e_{3}}(p),\text{ \ \ \ \ \ }\lambda _{e_{3}}(p)>%
\end{array}%
$ \\ \hline
$p_{1}$ & $\ \left[ 0.3,0.6\right] ,\text{ \ \ \ \ \ }0.6$ & $\ \left[
0.1,0.4\right] ,\text{ \ \ \ \ \ \ \ }0.5$ & $\ \left[ 0.5,0.8\right] ,\text{
\ \ \ \ \ \ }0.7$ \\ \hline
$p_{2}$ & $\ \left[ 0.4,0.7\right] ,\text{ \ \ \ \ \ }0.7$ & $\ \left[ 0.7,1%
\right] ,\text{ \ \ \ \ \ \ \ \ \ }0.4$ & $\ \left[ 0.6,0.9\right] ,\text{ \
\ \ \ \ \ }0.5$ \\ \hline
$p_{3}$ & $\ \left[ 0.6,0.9\right] ,\text{ \ \ \ \ \ }0.8$ & $\ \left[
0.4,0.7\right] ,\text{ \ \ \ \ \ \ \ }0.6$ & $\ \left[ 0.4,0.7\right] ,\text{
\ \ \ \ \ \ }0.6$ \\ \hline
$p_{4}$ & $\ \left[ 0.2,0.5\right] ,\text{ \ \ \ \ \ }0.3$ & $\ \left[
0.5,0.8\right] \text{ \ \ \ \ \ \ \ \ }0.7$ & $\ \left[ 0.3,0.6\right] ,%
\text{ \ \ \ \ \ \ }0.8$ \\ \hline
\end{tabular}%
\end{equation*}%
Let $\left( \widetilde{G,}\text{ }J\right) =\widetilde{G}(e_{i})=\mathcal{B}%
_{i}=\left\{ <p,\text{ }B_{e_{i}}(p),\ \mu _{e_{i}}(p)>:\text{ }p\in X\text{%
, }e_{i}\in J,i=1,2,3\right\} $ be a cubic soft set over $X$%
\begin{equation*}
\begin{tabular}{|l|l|l|l|}
\hline
$p$ & $%
\begin{array}{c}
\widetilde{G}(e_{1})=\mathcal{B}_{1}= \\ 
<B_{e_{1}}(p),\text{ \ \ \ \ }\ \mu _{e_{1}}(p)>%
\end{array}%
$ & $%
\begin{array}{c}
\widetilde{G}(e_{2})=\mathcal{B}_{2}= \\ 
<B_{e_{2}}(p),\text{ \ \ \ \ }\ \mu _{e_{2}}(p)>%
\end{array}%
$ & $%
\begin{array}{c}
\widetilde{G}(e_{3})=\mathcal{B}_{3}= \\ 
<B_{e_{3}}(p),\text{ \ \ \ \ }\ \mu _{e_{3}}(p)>%
\end{array}%
$ \\ \hline
$p_{1}$ & $\ \left[ 0.4,0.5\right] ,\text{ \ \ \ \ \ }0.7$ & $\ \left[
0.4,0.6\right] ,\text{ \ \ \ \ \ \ \ }0.6$ & $\ \left[ 0.6,0.9\right] ,\text{
\ \ \ \ \ \ }0.1$ \\ \hline
$p_{2}$ & $\ \left[ 0.5,0.6\right] ,\text{ \ \ \ \ \ }0.5\text{\ }$ & $\ %
\left[ 0.5,0.7\right] ,\text{ \ \ \ \ \ \ \ }0.7$ & $\ \left[ 0.7,1\right] ,%
\text{ \ \ \ \ \ \ \ \ }0.4$ \\ \hline
$p_{3}$ & $\ \left[ 0.7,0.8\right] ,\text{ \ \ \ \ \ }0.7$ & $\ \left[ 0.7,1%
\right] ,\text{ \ \ \ \ \ \ \ \ \ }0.5$ & $\ \left[ 0.5,0.8\right] ,\text{ \
\ \ \ \ \ }0.3$ \\ \hline
$p_{4}$ & $\ \left[ 0.3,0.4\right] ,\text{ \ \ \ \ \ }0.6$ & $\ \left[
0.4,0.7\right] ,\text{ \ \ \ \ \ \ \ }0.8$ & $\ \left[ 0.4,0.7\right] ,\text{
\ \ \ \ \ \ }0.9$ \\ \hline
\end{tabular}%
\end{equation*}%
Then, P-intersection is denoted by $\left( \widetilde{F,}\text{ }I\right)
\cap _{P}\left( \widetilde{G,}\text{ }J\right) $ and defined as:%
\begin{eqnarray*}
&&%
\begin{tabular}{|l|l|}
\hline
$p$ & $%
\begin{array}{c}
\widetilde{F}{\small (e}_{1}{\small )\wedge }_{P}\widetilde{G}{\small (e}_{1}%
{\small )=} \\ 
<{\small r}\min \left\{ A_{e_{1}}(p),B_{e_{1}}(p)\right\} ,(\lambda
_{e_{1}}\wedge \mu _{e_{1}})(p)>%
\end{array}%
$ \\ \hline
$p_{1}$ & $\ \ \ \ \ \ \ \ \ \ \ \ \ \ \left[ 0.3,0.5\right] ,\text{ \ \ \ \
\ \ \ \ }0.6$ \\ \hline
$p_{2}$ & $\ \ \ \ \ \ \ \ \ \ \ \ \ \ \left[ 0.4,0.6\right] ,\text{ \ \ \ \
\ \ \ \ }0.5$ \\ \hline
$p_{3}$ & $\ \ \ \ \ \ \ \ \ \ \ \ \ \ \left[ 0.6,0.8\right] ,\text{ \ \ \ \
\ \ \ \ }0.7$ \\ \hline
$p_{4}$ & $\ \ \ \ \ \ \ \ \ \ \ \ \ \ \left[ 0.2,0.4\right] ,\text{ \ \ \ \
\ \ \ \ }0.3$ \\ \hline
\end{tabular}
\\
&&%
\begin{tabular}{|l|l|}
\hline
$p$ & $%
\begin{array}{c}
\widetilde{F}{\small (e}_{2}{\small )\wedge }_{P}\widetilde{G}{\small (e}_{2}%
{\small )=} \\ 
<{\small r}\min \left\{ A_{e_{2}}(p),B_{e_{2}}(p)\right\} ,(\lambda
_{e_{2}}\wedge \mu _{e_{2}})(p)>%
\end{array}%
$ \\ \hline
$p_{1}$ & $\ \ \ \ \ \ \ \ \ \ \ \ \ \ \left[ 0.1,0.4\right] ,\text{ \ \ \ \
\ \ \ \ }0.5$ \\ \hline
$p_{2}$ & $\ \ \ \ \ \ \ \ \ \ \ \ \ \ \left[ 0.5,0.7\right] ,\text{ \ \ \ \
\ \ \ \ }0.4$ \\ \hline
$p_{3}$ & $\ \ \ \ \ \ \ \ \ \ \ \ \ \ \left[ 0.4,0.7\right] ,\text{ \ \ \ \
\ \ \ \ }0.5$ \\ \hline
$p_{4}$ & $\ \ \ \ \ \ \ \ \ \ \ \ \ \ \left[ 0.4,0.7\right] ,\text{ \ \ \ \
\ \ \ \ }0.7$ \\ \hline
\end{tabular}
\\
&&%
\begin{tabular}{|l|l|}
\hline
$p$ & $%
\begin{array}{c}
\widetilde{F}{\small (e}_{3}{\small )\wedge }_{P}\widetilde{G}{\small (e}_{3}%
{\small )=} \\ 
<{\small r}\min \left\{ A_{e_{2}}(p),B_{e_{2}}(p)\right\} ,(\lambda
_{e_{2}}\wedge \mu _{e_{2}})(p)>%
\end{array}%
$ \\ \hline
$p_{1}$ & $\ \ \ \ \ \ \ \ \ \ \ \ \ \left[ 0.5,0.8\right] ,\text{ \ \ \ \ \
\ \ \ \ }0.1$ \\ \hline
$p_{2}$ & $\ \ \ \ \ \ \ \ \ \ \ \ \ \left[ 0.6,0.9\right] ,\text{ \ \ \ \ \
\ \ \ \ }0.4$ \\ \hline
$p_{3}$ & $\ \ \ \ \ \ \ \ \ \ \ \ \ \left[ 0.4,0.7\right] ,\text{ \ \ \ \ \
\ \ \ \ }0.3$ \\ \hline
$p_{4}$ & $\ \ \ \ \ \ \ \ \ \ \ \ \ \left[ 0.3,0.6\right] ,\text{ \ \ \ \ \
\ \ \ \ }0.8$ \\ \hline
\end{tabular}%
\end{eqnarray*}%
(R-union) Let $\left( \widetilde{F,}\text{ }I\right) $ be a cubic soft set
over $X$ and defined as: 
\begin{equation*}
\begin{tabular}{|l|l|l|l|}
\hline
$p$ & $%
\begin{array}{c}
\widetilde{F}(e_{1})=\mathcal{A}_{1}= \\ 
<A_{e_{1}}(p),\text{ \ \ \ \ \ }\lambda _{e_{1}}(p)>%
\end{array}%
$ & $%
\begin{array}{c}
\widetilde{F}(e_{2})=\mathcal{A}_{2}= \\ 
<A_{e_{2}}(p),\text{ \ \ \ \ \ }\lambda _{e_{2}}(p)>%
\end{array}%
$ & $%
\begin{array}{c}
\widetilde{F}(e_{3})=\mathcal{A}_{3}= \\ 
<A_{e_{3}}(p),\text{ \ \ \ \ \ }\lambda _{e_{3}}(p)>%
\end{array}%
$ \\ \hline
$p_{1}$ & $\ \left[ 0.3,0.5\right] ,\text{ \ \ \ \ \ }0.6$ & $\ \left[
0.1,0.4\right] ,\text{ \ \ \ \ \ \ \ }0.5$ & $\ \left[ 0.5,0.8\right] ,\text{
\ \ \ \ \ \ }0.7$ \\ \hline
$p_{2}$ & $\ \left[ 0.4,0.6\right] ,\text{ \ \ \ \ \ }0.5$ & $\ \left[
0.5,0.7\right] ,\text{ \ \ \ \ \ \ \ }0.8$ & $\ \left[ 0.6,0.9\right] ,\text{
\ \ \ \ \ \ }0.5$ \\ \hline
$p_{3}$ & $\ \left[ 0.6,0.8\right] ,\text{ \ \ \ \ \ }0.7$ & $\ \left[
0.4,0.7\right] ,\text{ \ \ \ \ \ \ \ }0.5$ & $\ \left[ 0.4,0.7\right] ,\text{
\ \ \ \ \ \ }0.6$ \\ \hline
$p_{4}$ & $\ \left[ 0.3,0.5\right] ,\text{ \ \ \ \ \ }0.6$ & $\ \left[
0.5,0.8\right] ,\text{ \ \ \ \ \ \ \ }0.7$ & $\ \left[ 0.5,0.8\right] ,\text{
\ \ \ \ \ \ }0.7$ \\ \hline
\end{tabular}%
\end{equation*}%
Let$\left( \widetilde{G,}\text{ }J\right) =\widetilde{G}(e_{i})=\mathcal{B}%
_{i}=\left\{ <p,B_{e_{i}}(p),\mu _{i}(p)>:p\in X,e_{i}\in J,i=1,2,3\right\} $
be a cubic soft and defined as:%
\begin{equation*}
\begin{tabular}{|l|l|l|l|}
\hline
$p$ & $%
\begin{array}{c}
\widetilde{G}(e_{1})=\mathcal{B}_{1}= \\ 
<B_{e_{1}}(p),\text{ \ \ \ \ }\ \mu _{e_{1}}(p)>%
\end{array}%
$ & $%
\begin{array}{c}
\widetilde{G}(e_{2})=\mathcal{B}_{2}= \\ 
<B_{e_{2}}(p),\text{ \ \ \ \ }\ \mu _{e_{2}}(p)>%
\end{array}%
$ & $%
\begin{array}{c}
\widetilde{G}(e_{3})=\mathcal{B}_{3}= \\ 
<B_{e_{3}}(p),\text{ \ \ \ \ }\ \mu _{e_{3}}(p)>%
\end{array}%
$ \\ \hline
$p_{1}$ & $\ \left[ 0.4,0.5\right] ,\text{ \ \ \ \ \ }0.7$ & $\ \left[
0.4,0.6\right] ,\text{ \ \ \ \ \ \ \ }0.6$ & $\ \left[ 0.6,0.9\right] ,\text{
\ \ \ \ \ \ }0.1$ \\ \hline
$p_{2}$ & $\ \left[ 0.5,0.6\right] ,\text{ \ \ \ \ \ }0.5$ & $\ \left[
0.5,0.7\right] ,\text{ \ \ \ \ \ \ \ }0.7$ & $\ \left[ 0.7,1\right] ,\text{
\ \ \ \ \ \ \ \ }0.4$ \\ \hline
$p_{3}$ & $\ \left[ 0.7,0.8\right] ,\text{ \ \ \ \ \ }0.7$ & $\ \left[ 0.7,1%
\right] ,\text{ \ \ \ \ \ \ \ \ \ }0.5$ & $\ \left[ 0.5,0.8\right] ,\text{ \
\ \ \ \ \ }0.3$ \\ \hline
$p_{4}$ & $\ \left[ 0.3,0.4\right] ,\text{ \ \ \ \ \ }0.6$ & $\ \left[
0.4,0.7\right] ,\text{ \ \ \ \ \ \ \ }0.8$ & $\ \left[ 0.4,0.7\right] ,\text{
\ \ \ \ \ \ }0.9.$ \\ \hline
\end{tabular}%
\end{equation*}%
Then, their R-union is denoted by $\left( \widetilde{F,}\text{ }I\right)
\cup _{R}\left( \widetilde{G,}\text{ }J\right) $ and defined as%
\begin{eqnarray*}
&&%
\begin{tabular}{|l|l|}
\hline
$p$ & $%
\begin{array}{c}
\widetilde{F}(e_{1})\vee _{R}\widetilde{G}(e_{1})= \\ 
<r\max \left\{ A_{e_{1}}(p),B_{e_{1}}(p)\right\} ,(\lambda _{e_{1}}\wedge
\mu _{e_{1}})(p)>%
\end{array}%
$ \\ \hline
$p_{1}$ & $\ \ \ \ \ \ \ \ \ \ \ \ \ \ \left[ 0.4,0.5\right] ,\text{ \ \ \ \
\ \ \ \ \ }0.6$ \\ \hline
$p_{2}$ & $\ \ \ \ \ \ \ \ \ \ \ \ \ \ \left[ 0.5,0.6\right] ,\text{ \ \ \ \
\ \ \ \ \ }0.5$ \\ \hline
$p_{3}$ & $\ \ \ \ \ \ \ \ \ \ \ \ \ \ \left[ 0.7,0.8\right] ,\text{ \ \ \ \
\ \ \ \ \ }0.7$ \\ \hline
$p_{4}$ & $\ \ \ \ \ \ \ \ \ \ \ \ \ \ \left[ 0.3,0.5\right] ,\text{ \ \ \ \
\ \ \ \ \ }0.6$ \\ \hline
\end{tabular}
\\
&&%
\begin{tabular}{ll}
$p$ & $%
\begin{array}{c}
\widetilde{F}(e_{2})\vee _{R}\widetilde{G}(e_{2})= \\ 
<r\max \left\{ A_{e_{2}}(p),B_{e_{2}}(p)\right\} ,(\lambda _{e_{2}}\wedge
\mu _{e_{2}})(p)>%
\end{array}%
$ \\ 
$p_{1}$ & $\ \ \ \ \ \ \ \ \ \ \ \ \ \ \ \left[ 0.4,0.6\right] ,\text{ \ \ \
\ \ \ \ \ }0.5$ \\ 
$p_{2}$ & $\ \ \ \ \ \ \ \ \ \ \ \ \ \ \ \left[ 0.5,0.7\right] ,\text{ \ \ \
\ \ \ \ \ }0.7$ \\ 
$p_{3}$ & $\ \ \ \ \ \ \ \ \ \ \ \ \ \ \ \left[ 0.7,1\right] ,\text{ \ \ \ \
\ \ \ \ \ \ }0.5$ \\ 
$p_{4}$ & $\ \ \ \ \ \ \ \ \ \ \ \ \ \ \ \left[ 0.5,0.8\right] ,\text{ \ \ \
\ \ \ \ \ }0.7$%
\end{tabular}
\\
&&%
\begin{tabular}{|l|l|}
\hline
$p$ & $%
\begin{array}{c}
\widetilde{F}(e_{3})\vee _{R}\widetilde{G}(e_{3})= \\ 
<r\max \left\{ A_{e_{3}}(p),B_{e_{3}}(p)\right\} ,(\lambda _{e_{3}}\wedge
\mu _{e_{3}})(p)>%
\end{array}%
$ \\ \hline
$p_{1}$ & $\ \ \ \ \ \ \ \ \ \ \ \ \ \ \ \left[ 0.6,0.9\right] ,\text{ \ \ \
\ \ \ \ }0.1$ \\ \hline
$p_{2}$ & $\ \ \ \ \ \ \ \ \ \ \ \ \ \ \ \left[ 0.7,1\right] ,\text{ \ \ \ \
\ \ \ \ \ }0.4$ \\ \hline
$p_{3}$ & $\ \ \ \ \ \ \ \ \ \ \ \ \ \ \ \left[ 0.5,0.8\right] ,\text{ \ \ \
\ \ \ \ }0.3$ \\ \hline
$p_{4}$ & $\ \ \ \ \ \ \ \ \ \ \ \ \ \ \ \left[ 0.5,0.8\right] ,\text{ \ \ \
\ \ \ \ }0.7$ \\ \hline
\end{tabular}%
\end{eqnarray*}%
(R-intersection) Let $\left( \widetilde{F,}\text{ }I\right) $ be a cubic
soft set over $X$ and defined as: 
\begin{equation*}
\begin{tabular}{|l|l|l|l|}
\hline
$p$ & $%
\begin{array}{c}
\widetilde{F}(e_{1})=\mathcal{A}_{1}= \\ 
<A_{e_{1}}(p),\text{ \ \ \ \ \ }\lambda _{e_{1}}(p)>%
\end{array}%
$ & $%
\begin{array}{c}
\widetilde{F}(e_{2})=\mathcal{A}_{2}= \\ 
<A_{e_{2}}(p),\text{ \ \ \ \ \ }\lambda _{e_{2}}(p)>%
\end{array}%
$ & $%
\begin{array}{c}
\widetilde{F}(e_{3})=\mathcal{A}_{3}= \\ 
<A_{e_{3}}(p),\text{ \ \ \ \ \ }\lambda _{e_{3}}(p)>%
\end{array}%
$ \\ \hline
$p_{1}$ & $\ \left[ 0.3,0.6\right] ,\text{ \ \ \ \ \ }0.2$ & $\ \left[ 0.7,1%
\right] ,\text{ \ \ \ \ \ \ \ \ \ }0.6$ & $\ \left[ 0.4,0.7\right] ,\text{ \
\ \ \ \ \ }0.9$ \\ \hline
$p_{2}$ & $\ \left[ 0.4,0.7\right] ,\text{ \ \ \ \ \ }0.4$ & $\ \left[
0.5,0.8\right] ,\text{ \ \ \ \ \ \ }0.7$ & $\ \left[ 0.4,0.7\right] ,\text{
\ \ \ \ \ \ }0.6$ \\ \hline
$p_{3}$ & $\ \left[ 0.5,0.8\right] ,\text{ \ \ \ \ \ }0.7$ & $\ \left[
0.4,0.7\right] ,\text{ \ \ \ \ \ \ }0.8$ & $\ \left[ 0.5,0.8\right] ,\text{
\ \ \ \ \ \ }0.65$ \\ \hline
$p_{4}$ & $\ \left[ 0.6,0.9\right] ,\text{ \ \ \ \ \ }0.5$ & $\ \left[
0.3,0.6\right] ,\text{ \ \ \ \ \ \ }0.9$ & $\ \left[ 0.3,0.6\right] ,\text{
\ \ \ \ \ \ }0.75$ \\ \hline
\end{tabular}%
\end{equation*}%
Let $\left( \widetilde{G,}\text{ }J\right) =\left\{ \widetilde{G}(e_{i})=%
\mathcal{B}_{i\text{ }}=\left\{ <p,\text{ }B_{i\text{ }}(p),\text{ }\mu _{i%
\text{ }}(p)>:\text{ }p\in X\right\} e_{i}\in J,i=1,2,3\right\} $ be a cubic
sot set defined as: 
\begin{equation*}
\begin{tabular}{|l|l|l|l|}
\hline
$p$ & $%
\begin{array}{c}
\widetilde{G}(e_{1})=\mathcal{B}_{1}= \\ 
<B_{e_{1}}(p),\text{ \ \ \ \ }\ \mu _{e_{1}}(p)>%
\end{array}%
$ & $%
\begin{array}{c}
\widetilde{G}(e_{2})=\mathcal{B}_{2}= \\ 
<B_{e_{2}}(p),\text{ \ \ \ \ }\ \mu _{e_{2}}(p)>%
\end{array}%
$ & $%
\begin{array}{c}
\widetilde{G}(e_{3})=\mathcal{B}_{3}= \\ 
<B_{e_{3}}(p),\text{ \ \ \ \ }\ \mu _{e_{3}}(p)>%
\end{array}%
$ \\ \hline
$p_{1}$ & $\ \left[ 0.4,0.7\right] ,\text{ \ \ \ \ \ }0.3$ & $\ \left[
0.6,0.9\right] ,\text{ \ \ \ \ \ \ \ }0.3$ & $\ \left[ 0.5,0.8\right] ,\text{
\ \ \ \ \ \ }0.8$ \\ \hline
$p_{2}$ & $\ \left[ 0.3,0.6\right] ,\text{ \ \ \ \ \ }0.6$ & $\ \left[
0.4,0.7\right] ,\text{ \ \ \ \ \ \ \ }0.2$ & $\ \left[ 0.5,0.8\right] ,\text{
\ \ \ \ \ \ }0.6$ \\ \hline
$p_{3}$ & $\ \left[ 0.6,0.9\right] ,\text{ \ \ \ \ \ }0.7$ & $\ \left[
0.3,0.6\right] ,\text{ \ \ \ \ \ \ \ }0.4$ & $\ \left[ 0.3,0.6\right] ,\text{
\ \ \ \ \ \ }0.5$ \\ \hline
$p_{4}$ & $\ \left[ 0.5,0.8\right] ,\text{ \ \ \ \ \ }0.9$ & $\ \left[
0.4,0.7\right] ,\text{ \ \ \ \ \ \ \ }0.5$ & $\ \left[ 0.4,0.7\right] ,\text{
\ \ \ \ \ \ }0.7.$ \\ \hline
\end{tabular}%
\end{equation*}%
Then, the R-intersection is is denoted by $\left( \widetilde{F,}\text{ }%
I\right) \cap _{R}\left( \widetilde{G,}\text{ }J\right) $ and defined as
below:%
\begin{eqnarray*}
&&%
\begin{tabular}{|l|l|}
\hline
$p$ & $%
\begin{array}{c}
\widetilde{F}(e_{1})\wedge _{R}\widetilde{G}(e_{1})= \\ 
<r\min \left\{ A_{e_{1}}(p),B_{e_{1}}(p)\right\} ,(\lambda _{e_{1}}\vee \mu
_{e_{1}})(p)>%
\end{array}%
$ \\ \hline
$p_{1}$ & $\ \ \ \ \ \ \ \ \ \ \ \ \ \ \left[ 0.3,0.6\right] ,\text{ \ \ \ \
\ \ \ \ }0.3$ \\ \hline
$p_{2}$ & $\ \ \ \ \ \ \ \ \ \ \ \ \ \ \left[ 0.3,0.6\right] ,\text{ \ \ \ \
\ \ \ \ }0.6$ \\ \hline
$p_{3}$ & $\ \ \ \ \ \ \ \ \ \ \ \ \ \ \left[ 0.5,0.8\right] ,\text{ \ \ \ \
\ \ \ \ }0.7$ \\ \hline
$p_{4}$ & $\ \ \ \ \ \ \ \ \ \ \ \ \ \ \left[ 0.5,0.8\right] ,\text{ \ \ \ \
\ \ \ \ }0.9$ \\ \hline
\end{tabular}
\\
&&%
\begin{tabular}{ll}
$p$ & $%
\begin{array}{c}
\widetilde{F}(e_{2})\wedge _{R}\widetilde{G}(e_{2})= \\ 
<r\min \left\{ A_{e_{2}}(p),B_{e_{2}}(p)\right\} ,(\lambda _{e_{2}}\vee \mu
_{e_{2}})(p)>%
\end{array}%
$ \\ 
$p_{1}$ & $\ \ \ \ \ \ \ \ \ \ \ \ \ \ \ \left[ 0.6,0.9\right] ,\text{ \ \ \
\ \ \ }0.6$ \\ 
$p_{2}$ & $\ \ \ \ \ \ \ \ \ \ \ \ \ \ \ \left[ 0.4,0.7\right] ,\text{ \ \ \
\ \ \ }0.7$ \\ 
$p_{3}$ & \ $\ \ \ \ \ \ \ \ \ \ \ \ \ \ \left[ 0.3,0.6\right] ,\text{ \ \ \
\ \ \ }0.8$ \\ 
$p_{4}$ & $\ \ \ \ \ \ \ \ \ \ \ \ \ \ \ \left[ 0.3,0.6\right] ,\text{ \ \ \
\ \ \ }0.9$%
\end{tabular}
\\
&&%
\begin{tabular}{|l|l|}
\hline
$p$ & $%
\begin{array}{c}
\widetilde{F}(e_{3})\wedge _{R}\widetilde{G}(e_{3})= \\ 
<r\min A_{e_{3}}(p),B_{e_{3}}(p),(\lambda _{e_{3}}\vee \mu _{e_{3}})(p)>%
\end{array}%
$ \\ \hline
$p_{1}$ & $\ \ \ \ \ \ \ \ \ \ \ \ \ \ \left[ 0.4,0.7\right] ,\text{ \ \ \ \ 
}0.9$ \\ \hline
$p_{2}$ & $\ \ \ \ \ \ \ \ \ \ \ \ \ \ \left[ 0.4,0.7\right] ,\text{ \ \ \ \ 
}0.6$ \\ \hline
$p_{3}$ & $\ \ \ \ \ \ \ \ \ \ \ \ \ \ \left[ 0.3,0.6\right] ,\text{ \ \ \ \ 
}0.65$ \\ \hline
$p_{4}$ & $\ \ \ \ \ \ \ \ \ \ \ \ \ \ \left[ 0.3,0.6\right] ,\text{ \ \ \ \ 
}0.75$ \\ \hline
\end{tabular}%
\end{eqnarray*}
\end{example}
\end{enumerate}

\subsubsection{P-OR, R-OR, P-AND and R-AND of cubic soft sets}

\begin{definition}
Let $\left( \widetilde{F,}\text{ }I\right) =\left\{ \widetilde{F}(e_{i})=%
\mathcal{A}_{i}=\left\{ <x,\text{ }A_{e_{i}}(x),\text{ }\lambda _{e_{i}}(x)>:%
\text{ }x\in X\right\} \text{ }e_{i}\in I\right\} $ and $\left( \widetilde{G,%
}\text{ }J\right) =\left\{ \widetilde{G}(e_{i})=\mathcal{B}_{i\text{ }%
}=\left\{ <x,\text{ }B_{i\text{ }}(x),\text{ }\mu _{e_{i}\text{ }}(x)>:\text{
}x\in X\right\} \text{ }e_{i}\in J\right\} $ be cubic soft sets in $X.$Then,

\begin{enumerate}
\item P-OR is denoted by $\left( \widetilde{F,}\text{ }I\right) \vee
_{P}\left( \widetilde{G,}\text{ }J\right) $ and defined as $\left( 
\widetilde{F,}\text{ }I\right) \vee _{P}\left( \widetilde{G,}\text{ }%
J\right) =\left( \widetilde{H,}\text{ }I\times J\right) $ where $\widetilde{H%
}(\mathcal{\alpha }_{i},\mathcal{\beta }_{i})=\widetilde{F}(\mathcal{\alpha }%
_{i})\cup _{P}\widetilde{G}(\mathcal{\beta }_{i})$ for all $(\mathcal{\alpha 
}_{i},\mathcal{\beta }_{i})\in I\times J.$

\item R-OR is denoted by $\left( \widetilde{F,}\text{ }I\right) \vee
_{R}\left( \widetilde{G,}\text{ }J\right) $ and defined as $\left( 
\widetilde{F,}\text{ }I\right) \vee _{R}\left( \widetilde{G,}\text{ }%
J\right) =\left( \widetilde{H,}\text{ }I\times J\right) $ where $\widetilde{H%
}(\mathcal{\alpha }_{i}$,$\mathcal{\beta }_{i})=\widetilde{F}(\mathcal{%
\alpha }_{i})\cup _{R}\widetilde{G}(\mathcal{\beta }_{i})$ for all $(%
\mathcal{\alpha }_{i},\mathcal{\beta }_{i})\in I\times J.$

\item P-AND is denoted by $\left( \widetilde{F,}\text{ }I\right) \wedge
_{P}\left( \widetilde{G,}\text{ }J\right) $ and defined as $\left( 
\widetilde{F,}\text{ }I\right) \wedge _{P}\left( \widetilde{G,}\text{ }%
J\right) =\left( \widetilde{H,}\text{ }I\times J\right) $ where $\widetilde{H%
}(\mathcal{\alpha }_{i},\mathcal{\beta }_{i})=\widetilde{F}(\mathcal{\alpha }%
_{i})\cap _{P}\widetilde{G}(\mathcal{\beta }_{i})$ for all $(\mathcal{\alpha 
}_{i},\mathcal{\beta }_{i})\in I\times J.$

\item R-AND is denoted by $\left( \widetilde{F,}\text{ }I\right) \wedge
_{R}\left( \widetilde{G,}\text{ }J\right) $ and defined as $\left( 
\widetilde{F,}\text{ }I\right) \wedge _{R}\left( \widetilde{G,}\text{ }%
J\right) =\left( \widetilde{H,}\text{ }I\times J\right) $ where $\widetilde{H%
}(\mathcal{\alpha }_{i},\mathcal{\beta }_{i})=\widetilde{F}(\mathcal{\alpha }%
_{i})\cap _{R}\widetilde{G}(\mathcal{\beta }_{i})$ for all $(\mathcal{\alpha 
}_{i},\mathcal{\beta }_{i})\in I\times J.$
\end{enumerate}
\end{definition}

\begin{example}
Let $X=\left\{ p_{1},\text{ }p_{2},\text{ }p_{3},\text{ }p_{4}\right\} $ be
initial universe, $I=\left\{ e_{1},\text{ }e_{2}\right\} $ and $J=\left\{
e_{1},\text{ }e_{2},\text{ }e_{3}\right\} $ are any subsets of parameter's
set $E=\left\{ e_{1},\text{ }e_{2},\text{ }e_{3}\right\} .$ Let $\left( 
\widetilde{F,}\text{ }I\right) $ and $\left( \widetilde{G,}\text{ }J\right) $
be two cubic soft over $X$ and defined as below, respectively.%
\begin{equation*}
\begin{tabular}{|l|l|l|}
\hline
$p$ & $%
\begin{array}{c}
\widetilde{F}(e_{1})=\mathcal{A}_{1}= \\ 
<A_{e_{1}}(p),\text{ }\lambda _{e_{1}}(p)>%
\end{array}%
$ & $%
\begin{array}{c}
\widetilde{F}(e_{2})=\mathcal{A}_{2}= \\ 
<A_{e_{2}}(p),\lambda _{e_{2}}(p)>%
\end{array}%
$ \\ \hline
$p_{1}$ & $\ \left[ 0.3,0.6\right] ,\text{ \ \ \ \ \ }0.2\text{\ }$ & $\left[
0.7,1\right] ,\text{ \ \ \ \ \ \ \ \ \ }0.6$ \\ \hline
$p_{2}$ & $\ \left[ 0.4,0.7\right] ,\text{ \ \ \ \ \ }0.4\text{\ }$ & $\left[
0.5,0.8\right] ,\text{ \ \ \ \ \ \ }0.7$ \\ \hline
$p_{3}$ & $\ \left[ 0.5,0.8\right] ,\text{ \ \ \ \ \ }0.7$ & $\left[ 0.4,0.7%
\right] ,\text{ \ \ \ \ \ \ }0.8$ \\ \hline
$p_{4}$ & $\ \left[ 0.6,0.9\right] ,\text{ \ \ \ \ \ }0.5$ & $\left[ 0.3,0.6%
\right] ,\text{ \ \ \ \ \ \ }0.9$ \\ \hline
\end{tabular}%
\end{equation*}%
and 
\begin{equation*}
\begin{tabular}{|l|l|l|}
\hline
$p$ & $%
\begin{array}{c}
\widetilde{G}(e_{1})=\mathcal{B}_{1}= \\ 
<B_{e_{1}}(p),\text{ \ \ \ \ }\ \mu _{e_{1}}(p)>%
\end{array}%
$ & $%
\begin{array}{c}
\widetilde{G}(e_{2})=\mathcal{B}_{2}= \\ 
<B_{e_{2}}(p),\text{ \ \ \ \ }\ \mu _{e_{2}}(p)>%
\end{array}%
$ \\ \hline
$p_{1}$ & $\ \left[ 0.4,0.7\right] ,\text{ \ \ \ \ \ }0.3\text{\ \ }$ & $\ %
\left[ 0.6,0.9\right] ,\text{ \ \ \ \ \ \ \ }0.3\text{\ }$ \\ \hline
$p_{2}$ & $\ \left[ 0.3,0.6\right] ,\text{ \ \ \ \ \ }0.6\text{\ \ }$ & $\ %
\left[ 0.4,0.7\right] ,\text{ \ \ \ \ \ \ \ }0.2\text{\ }$ \\ \hline
$p_{3}$ & $\ \left[ 0.6,0.9\right] ,\text{ \ \ \ \ \ }0.7\text{\ \ }$ & $\ %
\left[ 0.3,0.6\right] ,\text{ \ \ \ \ \ \ \ }0.4\text{\ }$ \\ \hline
$p_{4}$ & $\ \left[ 0.5,0.8\right] ,\text{ \ \ \ \ \ }0.9\text{\ }$ & $\ %
\left[ 0.4,0.7\right] ,\text{ \ \ \ \ \ \ \ }0.5$. \\ \hline
\end{tabular}%
\end{equation*}%
Then P-OR is denoted as $\left( \widetilde{H,}\text{ }I\times J\right)
=\left( \widetilde{F,}\text{ }I\right) \vee _{P}\left( \widetilde{G,}\text{ }%
J\right) $, where $I\times J=\{(e_{1},e_{1}),$ $(e_{1},e_{2}),$ $%
(e_{2},e_{1}),$ $(e_{2},e_{2})\}$, is defined%
\begin{equation*}
\begin{tabular}{|l|l|l|l|l|}
\hline
$p$ & $\underset{\widetilde{F}(e_{1})\cup _{P}\widetilde{G}(e_{1})}{%
\widetilde{H}(e_{1},e_{1})=}$ & $\underset{\widetilde{F}(e_{1})\cup _{P}%
\widetilde{G}(e_{2})}{\widetilde{H}(e_{1},e_{2})=}$ & $\underset{\widetilde{F%
}(e_{2})\cup _{P}\widetilde{G}(e_{1})}{\widetilde{H}(e_{2},e_{1})=}$ & $%
\underset{\widetilde{F}(e_{2})\cup _{P}\widetilde{G}(e_{2})}{\widetilde{H}%
(e_{2},e_{2})=}$ \\ \hline
$p_{1}$ & $\left[ 0.4,0.7\right] ,0.3$ & $\left[ 0.6,0.9\right] ,0.3$ & $%
\left[ 0.7,1\right] ,$ $\ \ 0.6$ & $\left[ 0.7,1\right] ,$ $\ \ 0.6$ \\ 
\hline
$p_{2}$ & $\left[ 0.4,0.7\right] ,0.6$ & $\left[ 0.4,0.7\right] ,0.4$ & $%
\left[ 0.5,0.8\right] ,0.7$ & $\left[ 0.5,0.8\right] ,0.7$ \\ \hline
$p_{3}$ & $\left[ 0.6,0.9\right] ,0.7$ & $\left[ 0.5,0.8\right] ,0.7$ & $%
\left[ 0.6,0.9\right] ,0.8$ & $\left[ 0.4,0.7\right] ,0.8$ \\ \hline
$p_{4}$ & $\left[ 0.6,0.9\right] ,0.9$ & $\left[ 0.6,0.9\right] ,0.5$ & $%
\left[ 0.5,0.8\right] ,0.9$ & $\left[ 0.4,0.7\right] ,0.9$ \\ \hline
\end{tabular}%
\end{equation*}%
R-OR is denoted by $\left( \widetilde{H,}\text{ }I\times J\right) =\left( 
\widetilde{F,}\text{ }I\right) \vee _{R}\left( \widetilde{G,}\text{ }%
J\right) $, where $I\times J=\{(e_{1},e_{1}),$ $(e_{1},e_{2}),$ $%
(e_{2},e_{1}),$ $(e_{2},e_{2})\}$, is defined%
\begin{equation*}
\begin{tabular}{|l|l|l|l|l|}
\hline
$p$ & $\underset{\widetilde{F}(e_{1})\cup _{R}\widetilde{G}(e_{1})}{%
\widetilde{H}(e_{1},e_{1})=}$ & $\underset{\widetilde{F}(e_{1})\cup _{R}%
\widetilde{G}(e_{2})}{\widetilde{H}(e_{1},e_{2})=}$ & $\underset{\widetilde{F%
}(e_{2})\cup _{R}\widetilde{G}(e_{1})}{\widetilde{H}(e_{2},e_{1})=}$ & $%
\underset{\widetilde{F}(e_{2})\cup _{R}\widetilde{G}(e_{2})}{\widetilde{H}%
(e_{2},e_{2})=}$ \\ \hline
$p_{1}$ & $\left[ 0.4,0.7\right] ,0.2$ & $\left[ 0.6,0.9\right] ,0.2$ & $%
\left[ 0.7,1\right] ,$ $\ \ 0.3$ & $\left[ 0.7,1\right] ,$ $\ \ 0.3$ \\ 
\hline
$p_{2}$ & $\left[ 0.4,0.7\right] ,0.4$ & $\left[ 0.4,0.7\right] ,0.2$ & $%
\left[ 0.5,0.8\right] ,0.6$ & $\left[ 0.5,0.8\right] ,0.2$ \\ \hline
$p_{3}$ & $\left[ 0.6,0.9\right] ,0.7$ & $\left[ 0.5,0.8\right] ,0.4$ & $%
\left[ 0.6,0.9\right] ,0.7$ & $\left[ 0.4,0.7\right] ,0.4$ \\ \hline
$p_{4}$ & $\left[ 0.6,0.9\right] ,0.5$ & $\left[ 0.6,0.9\right] ,0.5$ & $%
\left[ 0.5,0.8\right] ,0.9$ & $\left[ 0.4,0.7\right] ,0.5$ \\ \hline
\end{tabular}%
\end{equation*}%
P-AND is denoted by $\left( \widetilde{H,}\text{ }I\times J\right) =\left( 
\widetilde{F,}\text{ }I\right) \wedge _{P}\left( \widetilde{G,}\text{ }%
J\right) $, where $I\times J=\{(e_{1},e_{1}),$ $(e_{1},e_{2}),$ $%
(e_{2},e_{1}),$ $(e_{2},e_{2})\}$, is defined%
\begin{equation*}
\begin{tabular}{|l|l|l|l|l|}
\hline
$p$ & $\underset{\widetilde{F}(e_{1})\cap _{P}\widetilde{G}(e_{1})}{%
\widetilde{H}(e_{1},e_{1})=}$ & $\underset{\widetilde{F}(e_{1})\cap _{P}%
\widetilde{G}(e_{2})}{\widetilde{H}(e_{1},e_{2})=}$ & $\underset{\widetilde{F%
}(e_{2})\cap _{P}\widetilde{G}(e_{1})}{\widetilde{H}(e_{2},e_{1})=}$ & $%
\underset{\widetilde{F}(e_{2})\cap _{P}\widetilde{G}(e_{2})}{\widetilde{H}%
(e_{2},e_{2})=}$ \\ \hline
$p_{1}$ & $\left[ 0.3,0.6\right] ,0.2$ & $\left[ 0.3,0.6\right] ,0.2$ & $%
\left[ 0.4,0.7\right] ,0.3$ & $\left[ 0.6,0.9\right] ,0.3$ \\ \hline
$p_{2}$ & $\left[ 0.3,0.6\right] ,0.4$ & $\left[ 0.4,0.7\right] ,0.2$ & $%
\left[ 0.3,0.6\right] ,0.6$ & $\left[ 0.4,0.7\right] ,0.2$ \\ \hline
$p_{3}$ & $\left[ 0.5,0.8\right] ,0.7$ & $\left[ 0.3,0.6\right] ,0.4$ & $%
\left[ 0.4,0.7\right] ,0.7$ & $\left[ 0.3,0.6\right] ,0.4$ \\ \hline
$p_{4}$ & $\left[ 0.5,0.8\right] ,0.5$ & $\left[ 0.4,0.7\right] ,0.5$ & $%
\left[ 0.3,0.6\right] ,0.9$ & $\left[ 0.3,0.6\right] ,0.5$ \\ \hline
\end{tabular}%
\end{equation*}%
R-AND is denoted by $\left( \widetilde{H,}\text{ }I\times J\right) =\left( 
\widetilde{F,}\text{ }I\right) \wedge _{P}\left( \widetilde{G,}\text{ }%
J\right) $, where $I\times J=\{(e_{1},e_{1}),$ $(e_{1},e_{2}),$ $%
(e_{2},e_{1}),$ $(e_{2},e_{2})\}$, is defiend 
\begin{equation*}
\begin{tabular}{|l|l|l|l|l|}
\hline
$p$ & $\underset{\widetilde{F}(e_{1})\cap _{R}\widetilde{G}(e_{1})}{%
\widetilde{H}(e_{1},e_{1})=}$ & $\underset{\widetilde{F}(e_{1})\cap _{R}%
\widetilde{G}(e_{2})}{\widetilde{H}(e_{1},e_{2})=}$ & $\underset{\widetilde{F%
}(e_{2})\cap _{R}\widetilde{G}(e_{1})}{\widetilde{H}(e_{2},e_{1})=}$ & $%
\underset{\widetilde{F}(e_{2})\cap _{R}\widetilde{G}(e_{2})}{\widetilde{H}%
(e_{2},e_{2})=}$ \\ \hline
$p_{1}$ & $\left[ 0.3,0.6\right] ,0.3$ & $\left[ 0.3,0.6\right] ,0.3$ & $%
\left[ 0.4,0.7\right] ,0.6$ & $\left[ 0.6,0.9\right] ,0.6$ \\ \hline
$p_{2}$ & $\left[ 0.3,0.6\right] ,0.6$ & $\left[ 0.4,0.7\right] ,0.4$ & $%
\left[ 0.3,0.6\right] ,0.7$ & $\left[ 0.4,0.7\right] ,0.7$ \\ \hline
$p_{3}$ & $\left[ 0.5,0.8\right] ,0.7$ & $\left[ 0.3,0.6\right] ,0.7$ & $%
\left[ 0.4,0.7\right] ,0.8$ & $\left[ 0.3,0.6\right] ,0.8$ \\ \hline
$p_{4}$ & $\left[ 0.5,0.8\right] ,0.9$ & $\left[ 0.4,0.7\right] ,0.5$ & $%
\left[ 0.3,0.6\right] ,0.9$ & $\left[ 0.3,0.6\right] ,0.9$ \\ \hline
\end{tabular}%
\end{equation*}
\end{example}

\begin{definition}
The complement of a cubic soft set%
\begin{equation*}
\left( \widetilde{F,}\text{ }I\right) =\left\{ \widetilde{F}(e_{i})=\left\{
<x,\text{ }A_{e_{i}}(x),\text{ }\lambda _{e_{i}}(x)>:\text{ }x\in X\right\} 
\text{ }e_{i}\in I\right\}
\end{equation*}%
is denoted by $\left( \widetilde{F,}\text{ }I\right) ^{c}$ and defined as $%
\left( \widetilde{F,}\text{ }I\right) ^{c}=$ $(\widetilde{F}^{c},\lnot I),$
where $\widetilde{F}^{c}:\lnot I\longrightarrow CP(X)$ and 
\begin{eqnarray*}
\widetilde{F}^{c}(e_{i}) &=&(\widetilde{F}(\lnot e_{i}))^{c}\text{ for all }%
e_{i}\in \lnot I \\
&=&(\widetilde{F}(e_{i}))^{c}\text{ (as }\lnot (\lnot e_{i})=e_{i})
\end{eqnarray*}%
$\left( \widetilde{F,}\text{ }I\right) ^{c}=\{((\widetilde{F}%
(e_{i}))^{c}=\{<x,A_{e_{i}}^{c}(x),\lambda _{e_{i}}^{c}(x)>:x\in X\}$ $%
e_{i}\in I\}.$
\end{definition}

\begin{example}
Let $X=\left\{ p_{1},\text{ }p_{2},\text{ }p_{3},\text{ }p_{4}\right\} $ be
initial universe and $E=\left\{ e_{1},\text{ }e_{2},\text{ }e_{3}\right\} $
parameter's set. Let $\left( \widetilde{F,}\text{ }I\right) $ be a cubic
soft set over $X$ and defined as:%
\begin{equation*}
\begin{tabular}{|l|l|l|l|}
\hline
$p$ & $%
\begin{array}{c}
\widetilde{F}(e_{1})=\mathcal{A}_{1} \\ 
=<A_{e_{1}}(p),\text{ \ \ \ \ \ }\lambda _{e_{1}}(p)>%
\end{array}%
$ & $%
\begin{array}{c}
\widetilde{F}(e_{2})=\mathcal{A}_{2} \\ 
=<A_{e_{2}}(p),\text{ \ \ \ \ \ \ }\lambda _{e_{2}}(p)>%
\end{array}%
$ & $%
\begin{array}{c}
\widetilde{F}(e_{3})=\mathcal{A}_{3} \\ 
=<A_{e_{3}}(p),\text{ \ \ \ \ \ }\lambda _{e_{3}}(p)>%
\end{array}%
$ \\ \hline
$p_{1}$ & $\ \left[ 0.3,0.5\right] ,\text{ \ \ \ \ \ }0.6$ & $\ \left[
0.1,0.4\right] ,\text{ \ \ \ \ \ \ }0.5\text{\ }$ & $\ \left[ 0.5,0.8\right]
,\text{ \ \ \ \ \ \ }0.7$ \\ \hline
$p_{2}$ & $\ \left[ 0.4,0.6\right] ,\text{ \ \ \ \ \ }0.5$ & $\ \left[
0.5,0.7\right] ,\text{ \ \ \ \ \ \ }0.4$ & $\ \left[ 0.6,0.9\right] ,\text{
\ \ \ \ \ \ }0.4$ \\ \hline
$p_{3}$ & $\ \left[ 0.6,0.8\right] ,\text{ \ \ \ \ \ }0.7$ & $\ \left[
0.4,0.7\right] ,\text{ \ \ \ \ \ \ }0.5$ & $\ \left[ 0.4,0.7\right] ,\text{
\ \ \ \ \ \ }0.3$ \\ \hline
$p_{4}$ & $\ \left[ 0.2,0.4\right] ,\text{ \ \ \ \ \ }0.3$ & $\ \left[
0.4,0.7\right] ,\text{ \ \ \ \ \ \ }0.7$ & $\ \left[ 0.3,0.6\right] ,\text{
\ \ \ \ \ \ }0.8.$ \\ \hline
\end{tabular}%
\end{equation*}%
Then, $\left( \widetilde{F,}\text{ }I\right) ^{c}==\{((\widetilde{F}%
(e_{i}))^{c}=\{<x,A_{e_{i}}^{c}(x),\lambda _{e_{i}}^{c}(x)>:x\in X\}$ $%
e_{i}\in I\}$ is defined as:%
\begin{eqnarray*}
&&%
\begin{tabular}{|l|l|}
\hline
$p$ & $%
\begin{array}{c}
((\widetilde{F}(e_{1}))^{c}=\mathcal{A}_{1}^{c}= \\ 
\text{ \ \ \ \ \ \ \ }\ \ <A_{e_{1}}^{c}(p),\text{\ \ \ \ \ \ \ \ \ \ \ \ \
\ }\lambda _{1}^{c}(p)> \\ 
=<\left[ 1-A_{e_{1}}^{+}(p),1-A_{e_{1}}^{-}(p)\right] ,\text{ \ }1-\lambda
_{e_{1}}(p)>%
\end{array}%
$ \\ \hline
$p_{1}$ & $\ \ \ \ \ \ \ \ \ \ \ \ \ \left[ 0.5,0.7\right] ,\text{ \ \ \ \ \
\ \ \ \ \ \ \ \ \ \ \ }0.4$ \\ \hline
$p_{2}$ & $\ \ \ \ \ \ \ \ \ \ \ \ \ \left[ 0.4,0.6\right] ,\text{ \ \ \ \ \
\ \ \ \ \ \ \ \ \ \ \ }0.5$ \\ \hline
$p_{3}$ & $\ \ \ \ \ \ \ \ \ \ \ \ \ \left[ 0.2,0.4\right] ,\text{ \ \ \ \ \
\ \ \ \ \ \ \ \ \ \ \ }0.3$ \\ \hline
$p_{4}$ & $\ \ \ \ \ \ \ \ \ \ \ \ \ \left[ 0.6,0.8\right] ,\text{ \ \ \ \ \
\ \ \ \ \ \ \ \ \ \ \ }0.7$ \\ \hline
\end{tabular}
\\
&&%
\begin{tabular}{|l|l|}
\hline
$p$ & $%
\begin{array}{c}
((\widetilde{F}(e_{2}))^{c}=\mathcal{A}_{2}^{c}= \\ 
\text{ \ \ \ \ \ \ \ }<A_{e_{2}}^{c}(p),\text{\ \ \ \ \ \ \ \ \ \ \ \ \ \ \
\ \ }\lambda _{2}^{c}(p)> \\ 
=<\left[ 1-A_{e_{2}}^{+}(p),1-A_{e_{2}}^{-}(p)\right] ,\text{ \ }1-\lambda
_{e_{2}}(p)>%
\end{array}%
$ \\ \hline
$p_{1}$ & $\ \ \ \ \ \ \ \ \ \ \ \ \ \ \left[ 0.6,0.9\right] ,\text{ \ \ \ \
\ \ \ \ \ \ \ \ \ \ }0.5$ \\ \hline
$p_{2}$ & $\ \ \ \ \ \ \ \ \ \ \ \ \ \ \left[ 0.3,0.5\right] ,\text{ \ \ \ \
\ \ \ \ \ \ \ \ \ \ }0.6$ \\ \hline
$p_{3}$ & $\ \ \ \ \ \ \ \ \ \ \ \ \ \ \left[ 0.3,0.6\right] ,\text{ \ \ \ \
\ \ \ \ \ \ \ \ \ \ }0.5$ \\ \hline
$p_{4}$ & $\ \ \ \ \ \ \ \ \ \ \ \ \ \ \left[ 0.3,0.6\right] ,\text{ \ \ \ \
\ \ \ \ \ \ \ \ \ \ }0.3$ \\ \hline
\end{tabular}
\\
&&%
\begin{tabular}{|l|l|}
\hline
$p$ & $%
\begin{array}{c}
((\widetilde{F}(e_{3}))^{c}=\mathcal{A}_{3}^{c}= \\ 
\text{ \ \ \ \ \ \ \ \ }<A_{e_{3}}^{c}(p),\text{\ \ \ \ \ \ \ \ \ \ \ \ }%
\lambda _{e_{3}}^{c}(p)> \\ 
=<\left[ 1-A_{e_{3}}^{+}(p),1-A_{e_{3}}^{-}(p)\right] ,1-\lambda _{e_{3}}(p)>%
\end{array}%
$ \\ \hline
$p_{1}$ & $\ \ \ \ \ \ \ \ \ \ \ \ \ \ \left[ 0.2,0.5\right] ,\text{ \ \ \ \
\ \ \ \ \ \ \ \ \ }0.3$ \\ \hline
$p_{2}$ & $\ \ \ \ \ \ \ \ \ \ \ \ \ \ \ \left[ 0.1,0.4\right] ,\text{ \ \ \
\ \ \ \ \ \ \ \ \ }0.6$ \\ \hline
$p_{3}$ & $\ \ \ \ \ \ \ \ \ \ \ \ \ \ \left[ 0.3,0.6\right] ,\text{ \ \ \ \
\ \ \ \ \ \ \ \ \ }0.7$ \\ \hline
$p_{4}$ & $\ \ \ \ \ \ \ \ \ \ \ \ \ \ \left[ 0.4,0.7\right] ,\text{ \ \ \ \
\ \ \ \ \ \ \ \ \ }0.2$ \\ \hline
\end{tabular}%
\end{eqnarray*}
\end{example}

\begin{proposition}
Let $X$ be initial universe and $I$, $J$, $L$ and $S$ subset of $E$. Then,
for any cubic soft sets $\left( \widetilde{F,}\text{ }I\right) $, $\left( 
\widetilde{G,}\text{ }J\right) $, $\left( \widetilde{E,}\text{ }L\right) $
and $\left( \widetilde{T,}\text{ }S\right) $ the following properties hold.

\begin{enumerate}
\item If $\left( \widetilde{F,}\text{ }I\right) \subseteq _{P}\left( 
\widetilde{G,}\text{ }J\right) $ and $\left( \widetilde{G,}\text{ }J\right)
\subseteq _{P}\left( \widetilde{E,}\text{ }L\right) ,$ then$\left( 
\widetilde{F,}\text{ }I\right) \subseteq _{P}(\widetilde{E,}$ $L).$

\item If $\left( \widetilde{F,}\text{ }I\right) \subseteq _{P}\left( 
\widetilde{G,}\text{ }J\right) ,$ then $\left( \widetilde{G,}\text{ }%
J\right) ^{c}\subseteq _{P}\left( \widetilde{F,}\text{ }I\right) ^{c}$ if $%
I=J.$

\item If $\left( \widetilde{F,}\text{ }I\right) \subseteq _{P}\left( 
\widetilde{G,}\text{ }J\right) $ and $\left( \widetilde{F,}\text{ }I\right)
\subseteq _{P}\left( \widetilde{E,}\text{ }L\right) ,$ then $\left( 
\widetilde{F,}\text{ }I\right) \subseteq _{P}\left( \widetilde{G,}\text{ }%
J\right) \cap _{P}\left( \widetilde{E,}\text{ }L\right) .$

\item If $\left( \widetilde{F,}\text{ }I\right) \subseteq _{P}\left( 
\widetilde{G,}\text{ }J\right) $ and $\left( \widetilde{E,}\text{ }L\right)
\subseteq _{P}\left( \widetilde{G,}\text{ }J\right) ,$ then $\left( 
\widetilde{F,}\text{ }I\right) \cup _{P}\left( \widetilde{E,}\text{ }%
L\right) \subseteq _{P}\left( \widetilde{G,}\text{ }J\right) .$

\item If $\left( \widetilde{F,}\text{ }I\right) \subseteq _{P}\left( 
\widetilde{G,}\text{ }J\right) $ and $\left( \widetilde{E,}\text{ }L\right)
\subseteq _{P}\left( \widetilde{T,}\text{ }S\right) ,$ then (a) $\left( 
\widetilde{F,}\text{ }I\right) \cup _{P}\left( \widetilde{E,}\text{ }%
L\right) \subseteq _{P}\left( \widetilde{G,}\text{ }J\right) \cup _{P}\left( 
\widetilde{T,}\text{ }S\right) $ and (b) $\left( \widetilde{F,}\text{ }%
I\right) \cap _{P}\left( \widetilde{E,}\text{ }L\right) \subseteq _{P}\left( 
\widetilde{G,}\text{ }J\right) \cap _{P}\left( \widetilde{T,}\text{ }%
S\right) .$

\item If $\left( \widetilde{F,}\text{ }I\right) \subseteq _{R}\left( 
\widetilde{G,}\text{ }J\right) $ and $\left( \widetilde{G,}\text{ }J\right)
\subseteq _{R}\left( \widetilde{E,}\text{ }L\right) ,$ then $\left( 
\widetilde{F,}\text{ }I\right) \subseteq _{R}\left( \widetilde{E,}\text{ }%
L\right) .$

\item If $\left( \widetilde{F,}\text{ }I\right) \subseteq _{R}\left( 
\widetilde{G,}\text{ }J\right) ,$ then $\left( \widetilde{G,}\text{ }%
J\right) ^{c}\subseteq _{R}\left( \widetilde{F,}\text{ }I\right) ^{c}$ if $%
I=J.$

\item If $\left( \widetilde{F,}\text{ }I\right) \subseteq _{R}\left( 
\widetilde{G,}\text{ }J\right) $ and $\left( \widetilde{F,}\text{ }I\right)
\subseteq _{R}\left( \widetilde{E,}\text{ }L\right) ,$ then $\left( 
\widetilde{F,}\text{ }I\right) \subseteq _{R}\left( \widetilde{G,}\text{ }%
J\right) \cap _{R}\left( \widetilde{E,}\text{ }L\right) .$

\item If $\left( \widetilde{F,}\text{ }I\right) \subseteq _{R}\left( 
\widetilde{G,}\text{ }J\right) $ and $\left( \widetilde{E,}\text{ }L\right)
\subseteq _{R}\left( \widetilde{G,}\text{ }J\right) ,$ then $\left( 
\widetilde{F,}\text{ }I\right) \cup _{R}\left( \widetilde{E,}\text{ }%
L\right) \subseteq _{R}\left( \widetilde{G,}\text{ }J\right) .$

\item If $\left( \widetilde{F,}\text{ }I\right) \subseteq _{R}\left( 
\widetilde{G,}\text{ }J\right) $ and $\left( \widetilde{E,}\text{ }L\right)
\subseteq _{R}\left( \widetilde{T,}\text{ }S\right) ,$ then (a)$\left( 
\widetilde{F,}\text{ }I\right) \cup _{R}\left( \widetilde{E,}\text{ }%
L\right) \subseteq _{R}\left( \widetilde{G,}\text{ }J\right) \cup _{R}\left( 
\widetilde{T,}\text{ }S\right) $ and

(b) $\left( \widetilde{F,}\text{ }I\right) \cap _{R}\left( \widetilde{E,}%
\text{ }L\right) \subseteq _{R}\left( \widetilde{G,}\text{ }J\right) \cap
_{R}\left( \widetilde{T,}\text{ }S\right) .$
\end{enumerate}
\end{proposition}

\begin{proof}
\begin{proof}
Proof is straightforward.
\end{proof}
\end{proof}

\begin{theorem}
Let $\left( \widetilde{F,}\text{ }I\right) $ be a cubic soft set over $X.$
\end{theorem}

\begin{enumerate}
\item If $\left( \widetilde{F,}\text{ }I\right) $ is an internal cubic soft
set, then $\left( \widetilde{F,}\text{ }I\right) ^{c}$ is also an internal
cubic soft set (ICSS).

\item If $\left( \widetilde{F,}\text{ }I\right) $ is an external cubic soft
set, then $\left( \widetilde{F,}\text{ }I\right) ^{c}$ is also an external
cubic soft set (ECSS).
\end{enumerate}

\begin{proof}
(1) Given$\left( \widetilde{F,}\text{ }I\right) =\left\{ \widetilde{F}%
(e_{i})=\left\{ <x,\text{ }A_{e_{i}}(x),\text{ }\lambda _{e_{i}}(x)>:\text{ }%
x\in X\right\} \text{ }e_{i}\in I,\right\} $ is an ICSS this implies $%
A_{e_{i}}^{-}(x)\leq \lambda _{e_{i}}(x)\leq A_{e_{i}}^{+}(x)$ for all $%
e_{i}\in I$ and for all $x\in X,$ this implies $1-A_{e_{i}}^{+}(x)\leq
1-\lambda _{e_{i}}(x)\leq 1-A_{e_{i}}^{-}(x)$ for all $e_{i}\in I$ and for
all $x\in X.$ Also we have

Hence $\left( \widetilde{F,}\text{ }I\right) ^{c}$ is an ICSS.

(2{\large ) }Given $\left( \widetilde{F,}\text{ }I\right) =\left\{ 
\widetilde{F}(e_{i})=\left\{ <x,\text{ }A_{e_{i}}(x),\text{ }\lambda
_{e_{i}}(x)>:\text{ }x\in X\right\} \text{ }e_{i}\in I,\right\} $ is an ECSS
this implies $\lambda _{e_{i}}(x)\notin (A_{e_{i}}^{-}(x),A_{e_{i}}^{+}(x))$
for all $e_{i}\in I$ and for all $x\in X.$ Since $\lambda _{e_{i}}(x)\notin
((A_{e_{i}}^{-}(x),A_{e_{i}}^{+}(x))$ and $0\leq A_{e_{i}}^{-}(x)\leq
A_{e_{i}}^{+}(x)\leq 1.$ So we have $\lambda _{e_{i}}(x)\leq
A_{e_{i}}^{-}(x) $ or $A_{e_{i}}^{+}(x)\leq \lambda _{e_{i}}(x)$ this
implies $1-\lambda _{e_{i}}(x)\geq 1-A_{e_{i}}^{-}(x)$ or $%
1-A_{e_{i}}^{+}(x)\geq 1-\lambda _{e_{i}}(x).$ Thus $1-\lambda
_{e_{i}}(x)\notin (1-A_{e_{i}}^{-}(x),1-A_{e_{i}}^{+}(x))$ for all $e_{i}\in
I$ and for all $x\in X.$ Hence

$\left( \widetilde{F,}\text{ }I\right) ^{c}$ is an ECSS.
\end{proof}

\begin{theorem}
Let $\left( \widetilde{F,}\text{ }I\right) =\left\{ \widetilde{F}%
(e_{i})=\left\{ <x,\text{ }A_{e_{i}}(x),\text{ }\lambda _{e_{i}}(x)>:\text{ }%
x\in X\right\} \text{ }e_{i}\in I\right\} $ and $\left( \widetilde{G,}\text{ 
}J\right) =\left\{ \widetilde{G}(e_{i})=\left\{ <x,\text{ }B_{e_{i}}(x),%
\text{ }\mu _{e_{i}}(x)>:\text{ }x\in X\right\} \text{ }e_{i}\in J\right\} $
be internal cubic soft sets. Then,

$\left( 1\right) $ $\left( \widetilde{F,}\text{ }I\right) \cup _{P}\left( 
\widetilde{G,}\text{ }J\right) $ is an ICSS.

$\left( 2\right) $ $\left( \widetilde{F,}\text{ }I\right) \cap _{P}\left( 
\widetilde{G,}\text{ }J\right) $ is an ICSS.
\end{theorem}

\begin{proof}
(1) Since $\left( \widetilde{F,}\text{ }I\right) $ and $\left( \widetilde{G,}%
\text{ }J\right) $ are internal cubic soft sets. So for $\left( \widetilde{F,%
}\text{ }I\right) $ we have $A_{e_{i}}^{-}(x)\leq \lambda _{e_{i}}(x)\leq
A_{e_{i}}^{+}(x)$ for all $e_{i}\in I$ and for all $x\in X.$ Also for $%
\left( \widetilde{G,}\text{ }J\right) $ we have $B_{e_{i}}^{-}(x)\leq \mu
_{e_{i}}(x)\leq B_{e_{i}}^{+}(x)$ for all $e_{i}\in J$ and for all $x\in X.$
Then we have $\max \{A_{e_{i}}^{-}(x),B_{e_{i}}^{-}(x)\}\leq (\lambda
_{e_{i}}\vee \mu _{e_{i}})(x)\leq \max \{A_{e_{i}}^{+}(x),B_{e_{i}}^{+}(x)\}$
for all $e_{i}\in I\cup J$ and for all $x\in X$. Now by definition of
P-union of $\left( \widetilde{F,}\text{ }I\right) $ and $\left( \widetilde{G,%
}\text{ }J\right) ,$ we have $\left( \widetilde{F,}\text{ }I\right) \cup
_{P}\left( \widetilde{G,}\text{ }J\right) =\left( \widetilde{H,}\text{ }%
C\right) $ where $C=I\cup J$ and%
\begin{equation*}
H(e_{i})=\left\{ 
\begin{array}{c}
\begin{array}{c}
\widetilde{F}(e_{i})\text{ If }e_{i}\in I-J \\ 
\widetilde{G}(e_{i})\text{ If }e_{i}\in J-I%
\end{array}%
\text{ \ \ \ \ \ \ \ \ \ \ \ \ } \\ 
\widetilde{F}(e_{i})\vee _{P}\widetilde{G}(e_{i})\text{ If }e_{i}\in I\cap J%
\end{array}%
\right.
\end{equation*}%
If $e_{i}\in I\cap J,$ then $\widetilde{F}(e_{i})\vee _{P}\widetilde{G}%
(e_{i})$ is defined as%
\begin{equation*}
\widetilde{F}(e_{i})\vee _{P}\widetilde{G}(e_{i})=\widetilde{H}%
(e_{i})=\left\{ 
\begin{array}{c}
<x,\text{ }r\max \{A_{e_{i}}(x),B_{e_{i}}(x)\},(\lambda _{e_{i}}\vee \mu
_{e_{i}})(x)>: \\ 
\text{ }x\in X,\text{ }e_{i}\in I\cap J.%
\end{array}%
\right\}
\end{equation*}%
Thus $\left( \widetilde{F,}\text{ }I\right) \cup _{P}\left( \widetilde{G,}%
\text{ }J\right) $ is an ICSS if $e_{i}\in I\cap J.$ If $e_{i}\in I-J$ or $%
e_{i}\in J-I$ , then the result is trivial. Hence, $\left( \widetilde{F,}%
\text{ }I\right) \cup _{P}\left( \widetilde{G,}\text{ }J\right) $ is an ICSS
in all cases.

(2) Since $\left( \widetilde{F,}\text{ }I\right) \cap _{P}\left( \widetilde{%
G,}\text{ }J\right) =\left( \widetilde{H,}\text{ }C\right) ,$ where $C=I\cap
J$ and $\widetilde{H}(e_{i})=\widetilde{F}(e_{i})\wedge _{P}\widetilde{G}%
(e_{i}).$ If $e_{i}\in I\cap J,$ then $\widetilde{F}(e_{i})\wedge _{P}%
\widetilde{G}(e_{i})$ is defined as 
\begin{equation*}
\widetilde{F}(e_{i})\wedge _{P}\widetilde{G}(e_{i})=\widetilde{H}%
(e_{i})=\left\{ 
\begin{array}{c}
<x,\text{ }r\min \{A_{e_{i}}(x),B_{e_{i}}(x)\},(\lambda _{e_{i}}\wedge \mu
_{e_{i}})(x)>: \\ 
\text{ }x\in X,\text{ }e_{i}\in I\cap J.%
\end{array}%
\right\}
\end{equation*}%
Also given that $\left( \widetilde{F,}\text{ }I\right) $ and $\left( 
\widetilde{G,}\text{ }J\right) $ are internal cubic soft sets. So for $%
\left( \widetilde{F,}\text{ }I\right) $ we have $A_{e_{i}}^{-}(x)\leq
\lambda _{e_{i}}(x)\leq A_{e_{i}}^{+}(x)$ for all $e_{i}\in I$ and for all $%
x\in X.$ And for $\left( \widetilde{G,}\text{ }J\right) $ we have $%
B_{e_{i}}^{-}(x)\leq \mu _{e_{i}}(x)\leq B_{e_{i}}^{+}(x)$ for all $e_{i}\in
J$ and for all $x\in X.$ This implies $\min
\{A_{e_{i}}^{-}(x),B_{e_{i}}^{-}(x)\}\leq (\lambda _{e_{i}}\wedge \mu
_{e_{i}})(x)\leq \min \{A_{e_{i}}^{+}(x),B_{e_{i}}^{+}(x)\}$ for all $%
e_{i}\in I\cap J.$ Hence $\left( \widetilde{H,}\text{ }C\right) =\left( 
\widetilde{F,}\text{ }I\right) \cap _{P}\left( \widetilde{G,}\text{ }%
J\right) $ is an internal cubic soft set (ICSS).
\end{proof}

\begin{theorem}
Let $\left( \widetilde{F,}\text{ }I\right) $ and $\left( \widetilde{G,}\text{
}J\right) $ be ICSSs over $X$ such that $\max
\{A_{e_{i}}^{-}(x),B_{e_{i}}^{-}(x)\}\leq (\lambda _{e_{i}}\wedge \mu
_{e_{i}})(x)$ for all $e_{i}\in I\cap J$ and for all $x\in X.$ Then, the
R-union of $\left( \widetilde{F,}\text{ }I\right) $ and $\left( \widetilde{G,%
}\text{ }J\right) $ is also an ICSS.
\end{theorem}

\begin{proof}
Since $\left( \widetilde{F,}\text{ }I\right) $ and $\left( \widetilde{G,}%
\text{ }J\right) $ are internal cubic soft sets in $X$. So for $\left( 
\widetilde{F,}\text{ }I\right) $ we have $A_{e_{i}}^{-}(x)\leq \lambda
_{e_{i}}(x)\leq A_{e_{i}}^{+}(x)$ for all $e_{i}\in I$ and for all $x\in X.$
Also for $\left( \widetilde{G,}\text{ }J\right) $ we have $%
B_{e_{i}}^{-}(x)\leq \mu _{e_{i}}(x)\leq B_{e_{i}}^{+}(x)$ for all $e_{i}\in
J$ and for all $x\in X.$ So we have $(\lambda _{e_{i}}\wedge \mu
_{e_{i}})(x)\leq \max \{A_{e_{i}}^{+}(x),B_{e_{i}}^{+}(x)\}.$ Also given
that $\max \{A_{e_{i}}^{-}(x),B_{e_{i}}^{-}(x)\}\leq (\lambda _{e_{i}}\wedge
\mu _{e_{i}})(x)$ for all $e_{i}\in I\cap J$ and for all $x\in X.$ Now $%
\left( \widetilde{F,}\text{ }I\right) \cup _{R}\left( \widetilde{G,}\text{ }%
J\right) =\left( \widetilde{H,}\text{ }C\right) $, where $C=I\cup J$ and%
\begin{equation*}
H(e_{i})=\left\{ 
\begin{array}{c}
\begin{array}{c}
\widetilde{F}(e_{i})\text{ If }e_{i}\in I-J \\ 
\widetilde{G}(e_{i})\text{ If }e_{i}\in J-I%
\end{array}%
\text{ \ \ \ \ \ \ \ \ \ \ \ \ } \\ 
\widetilde{F}(e_{i})\vee _{R}\widetilde{G}(e_{i})\text{ If }e_{i}\in I\cap J,%
\end{array}%
\right.
\end{equation*}%
where $\widetilde{F}(e_{i})\vee _{R}\widetilde{G}(e_{i})$ is defined as 
\begin{equation*}
\widetilde{F}(e_{i})\vee _{R}\widetilde{G}(e_{i})=\widetilde{H}%
(e_{i})=\left\{ 
\begin{array}{c}
<x,\text{ }r\max \{A_{e_{i}}(x),B_{e_{i}}(p)\},(\lambda _{e_{i}}\wedge \mu
_{e_{i}})(x)>: \\ 
\text{ }x\in X\text{, }e_{i}\in I\cap J.%
\end{array}%
\right\}
\end{equation*}
Since $\left( \widetilde{F,}\text{ }I\right) $ and $\left( \widetilde{G,}%
\text{ }J\right) $ are ICSSs so from above given condition and definition of
an ICSS we can write $\max \{A_{e_{i}}^{-}(x),B_{e_{i}}^{-}(x)\}\leq
(\lambda _{e_{i}}\wedge \mu _{e_{i}})(x)\leq \max
\{A_{e_{i}}^{+}(x),B_{e_{i}}^{+}(x)\}$ for all $e_{i}\in I\cap J,$ and for
all $x\in X.$ If $e_{i}\in I-J$ or $e_{i}\in J-I$ then the result is
trivial. Thus $\left( \widetilde{F,}\text{ }I\right) $ $\cup _{R}\left( 
\widetilde{G,}\text{ }J\right) =\left( \widetilde{H,}\text{ }C\right) $ is
an ICSS if $\max \{A_{e_{i}}^{-}(x),B_{e_{i}}^{-}(x)\}\leq (\lambda
_{e_{i}}\wedge \mu _{e_{i}})(x)$ for all $e_{i}\in I\cup J$ and for all $%
x\in X.$
\end{proof}

\begin{theorem}
Let $\left( \widetilde{F,}\text{ }I\right) =\left\{ \widetilde{F}%
(e_{i})=\left\{ <x,\text{ }A_{e_{i}}(x),\text{ }\lambda _{e_{i}}(x)>:\text{ }%
x\in X\right\} \text{ }e_{i}\in I\right\} $ and $\left( \widetilde{G,}\text{ 
}J\right) =\left\{ \widetilde{G}(e_{i})=\left\{ <x,\text{ }B_{e_{i}}(x),%
\text{ }\mu _{e_{i}}(x)>:\text{ }x\in X\right\} \text{ }e_{i}\in J\right\} $
be ICSSs in $X$ satisfying the following inequality

$\min \{A_{e_{i}}^{+}(x),B_{e_{i}}^{+}(x)\}\geq (\lambda _{e_{i}}\vee \mu
_{e_{i}})(x)$ for all $e_{i}\in I\cap J$ and for all $x\in X.$ Then $\left( 
\widetilde{F,}\text{ }I\right) $ $\cap _{R}\left( \widetilde{G,}\text{ }%
J\right) $ is an ICSS.
\end{theorem}

\begin{proof}
Let $\left( \widetilde{F,}\text{ }I\right) =\left\{ \widetilde{F}%
(e_{i})=\left\{ <x,\text{ }A_{e_{i}}(x),\text{ }\lambda _{e_{i}}(x)>:\text{ }%
x\in X\right\} \text{ }e_{i}\in I\right\} $ and $\left( \widetilde{G,}\text{ 
}J\right) =\left\{ \widetilde{G}(e_{i})=\left\{ <x,\text{ }B_{e_{i}}(x),%
\text{ }\mu _{e_{i}}(x)>:\text{ }x\in X\right\} \text{ }e_{i}\in J\right\} .$
Then, by definition of an ICSS. $A_{e_{i}}^{-}(x)\leq \lambda
_{e_{i}}(x)\leq A_{e_{i}}^{+}(x)$ for all $e_{i}\in I$ and for all $x\in X$
and $B_{e_{i}}^{-}(x)\leq \mu _{e_{i}}(x)\leq B_{e_{i}}^{+}(x)$ for all $%
e_{i}\in J$ and for all $x\in X,$ this implis that $\min
\{A_{e_{i}}^{-}(x),B_{e_{i}}^{-}(x)\}\leq (\lambda _{e_{i}}\vee \mu
_{e_{i}})(x).$ Also since $\left( \widetilde{F,}\text{ }I\right) \cap
_{R}\left( \widetilde{G,}\text{ }J\right) =\left( \widetilde{H,}\text{ }%
C\right) $ where $C=I\cap J$ and \ $H(e_{i})=\widetilde{F}(e_{i})\wedge _{R}%
\widetilde{G}(e_{i})$ If $e_{i}\in I\cap J$ then $\widetilde{F}(e_{i})\wedge
_{R}\widetilde{G}(e_{i})$ is defined as 
\begin{equation*}
\widetilde{F}(e_{i})\wedge _{R}\widetilde{G}(e_{i})=\left\{ \widetilde{H}%
(e_{i})=\left\{ 
\begin{array}{c}
<x,\text{ }r\min \{A_{e_{i}}(x),B_{e_{i}}(x)\},(\ \lambda _{e_{i}}\vee \mu
_{e_{i}})(x)>: \\ 
\text{ }x\in X,\text{ }e_{i}\in I\cap J%
\end{array}%
\right\} \right\}
\end{equation*}%
Given condition $\min \{A_{e_{i}}^{+}(x),B_{e_{i}}^{+}(x)\}\geq (\lambda
_{e_{i}}\vee \mu _{e_{i}})(x)$ for all $e_{i}\in I\cap J$ and for all $x\in
X.$ Thus from given condition and definition of ICSSs $%
\{A_{e_{i}}^{-}(x),B_{e_{i}}^{-}(x)\}\leq (\lambda _{e_{i}}\vee \mu
_{e_{i}})(x)\leq \min \{A_{e_{i}}^{+}(x),B_{e_{i}}^{+}(x)\}.$ Hence $\left( 
\widetilde{F,}\text{ }I\right) $ $\cap _{R}\left( \widetilde{G,}\text{ }%
J\right) $ is an ICSS.
\end{proof}

\begin{definition}
Given two cubic soft sets

$\left( \widetilde{F,}\text{ }I\right) =\left\{ \widetilde{F}(e_{i})=\left\{
<x,\text{ }A_{e_{i}}(x),\text{ }\lambda _{e_{i}}(x)>:\text{ }x\in X\right\} 
\text{ }e_{i}\in I,\right\} $

and

$\left( \widetilde{G,}\text{ }J\right) =\left\{ \widetilde{G}(e_{i})=\left\{
<x,\text{ }B_{e_{i}}(x),\text{ }\mu _{e_{i}}(x)>:\text{ }x\in X\right\} 
\text{ }e_{i}\in J\right\} ,$ if we interchange $\lambda $ and $\mu $, then
the new cubic soft sets are denoted and defined as

$\left( \widetilde{F,}\text{ }I\right) ^{\ast }=\left\{ \widetilde{F}%
(e_{i})=\left\{ <x,\text{ }A_{e_{i}}(x),\text{ }\mu _{e_{i}}(x)>:\text{ }%
x\in X\right\} \text{ }e_{i}\in I,\right\} $ and

$\left( \widetilde{G,}\text{ }J\right) ^{\ast }=\left\{ \widetilde{G}%
(e_{i})=\left\{ <x,\text{ }B_{e_{i}}(x),\text{ }\lambda _{e_{i}}(x)>:\text{ }%
x\in X\right\} \text{ }e_{i}\in J\right\} $ respectively.
\end{definition}

\begin{theorem}
\label{P1}For two ECSSs $\left( \widetilde{F,}\text{ }I\right) =\left\{ 
\widetilde{F}(e_{i})=\left\{ <x,\text{ }A_{e_{i}}(x),\text{ }\lambda
_{e_{i}}(x)>:\text{ }x\in X\right\} \text{ }e_{i}\in I,\right\} $ and $%
\left( \widetilde{G,}\text{ }J\right) =\left\{ \widetilde{G}(e_{i})=\left\{
<x,\text{ }B_{e_{i}}(x),\text{ }\mu _{e_{i}}(x)>:\text{ }x\in X\right\} 
\text{ }e_{i}\in J\right\} $ in $X,$ if $\left( \widetilde{F,}\text{ }%
I\right) ^{\ast }$ and $\left( \widetilde{G,}\text{ }J\right) ^{\ast }$ are
ICSSs in $X$ then $\left( \widetilde{F,}\text{ }I\right) \cup _{P}\left( 
\widetilde{G,}\text{ }J\right) $ is an ICSS in $X$.
\end{theorem}

\begin{proof}
Since $\left( \widetilde{F,}\text{ }I\right) =\left\{ \widetilde{F}%
(e_{i})=\left\{ <x,\text{ }A_{e_{i}}(x),\text{ }\lambda _{e_{i}}(x)>:\text{ }%
x\in X\right\} \text{ }e_{i}\in I\right\} $ and $\left( \widetilde{G,}\text{ 
}J\right) =\left\{ \widetilde{G}(e_{i})=\left\{ <x,\text{ }B_{e_{i}}(x),%
\text{ }\mu _{e_{i}}(x)>:\text{ }x\in X\right\} \text{ }e_{i}\in J\right\} $
are ECSSs. Then for $\left( \widetilde{F,}\text{ }I\right) ,$ we have $%
\lambda _{e_{i}}(x)\notin (A_{e_{i}}^{-}(x),A_{e_{i}}^{+}(x))$ for all $%
e_{i}\in I$ and for all $x\in X.$ And for $\left( \widetilde{G,}\text{ }%
J\right) ,$ we have$\ \mu _{e_{i}}(x)\notin
(B_{e_{i}}^{-}(x),B_{e_{i}}^{+}(x))$ for all $e_{i}\in J$ and for all $x\in
X.$ Also given that

$\left( \widetilde{F,}\text{ }I\right) ^{\ast }=\left\{ \widetilde{F}%
(e_{i})=\left\{ <x,\text{ }A_{e_{i}\text{ }}(x),\text{ }\mu _{e_{i}\text{ }%
}(x)>:\text{ }x\in X\right\} \text{ }e_{i}\in I,\right\} $ and

$\left( \widetilde{G,}\text{ }J\right) ^{\ast }=\left\{ \widetilde{G}%
(e_{i})=\left\{ <x,\text{ }B_{e_{i}\text{ }}(x),\text{ }\lambda _{e_{i}\text{
}}(x)>:\text{ }x\in X\right\} \text{ }e_{i}\in J\right\} $ are ICSSs so this
implies $A_{e_{i}}^{-}(x)\leq \mu _{e_{i}}(x)\leq A_{e_{i}}^{+}(x)$ for all $%
e_{i}\in I$ and for all $x\in X$ and $B_{e_{i}}^{-}(x)\leq \lambda
_{e_{i}}(x)\leq B_{e_{i}}^{+}(x)$ for all $e_{i}\in J$ and for all $x\in X.$
Since $\left( \widetilde{F,}\text{ }I\right) $ and $\left( \widetilde{G,}%
\text{ }J\right) $ are ECSSs, and $\left( \widetilde{F,}\text{ }I\right)
^{\ast }$ and $\left( \widetilde{G,}\text{ }J\right) ^{\ast }$ are ICSSs.
Thus by definition of ECSSs and ICSSs, all the possibilities are as under

(i)(a) $\lambda _{e_{i}\text{ }}(x)\leq A_{e_{i}}^{-}(x)\leq \mu
_{e_{i}}(x)\leq A_{e_{i}}^{+}(x)$ for all $e_{i}\in I$ and for all $x\in X$.

(i)(b) $\mu _{e_{i}}(x)\leq B_{e_{i}}^{-}(x)\leq \lambda _{e_{i}\text{ }%
}(x)\leq B_{e_{i}}^{+}(x)$ for all $e_{i}\in J$ and for all $x\in X.$

(ii)(a) $A_{e_{i}}^{-}(x)\leq \mu _{e_{i}}(x)\leq A_{e_{i}}^{+}(x)\leq
\lambda _{e_{i}\text{ }}(x)$ for all $e_{i}\in I$ and for all $x\in X$.

(ii)(b) $B_{e_{i}}^{-}(x)\leq \lambda _{e_{i}\text{ }}(x)\leq
B_{e_{i}}^{+}(x)\leq \mu _{e_{i}}(x)$ for all $e_{i}\in J$ and for all $x\in
X.$

(iii)(a) $\lambda _{e_{i}\text{ }}(x)\leq A_{e_{i}}^{-}(x)\leq \mu
_{e_{i}}(x)\leq A_{e_{i}}^{+}(x)$ for all $e_{i}\in I$ and for all $x\in X$.

(iii)(b) $B_{e_{i}}^{-}(x)\leq \lambda _{e_{i}\text{ }}(x)\leq
B_{e_{i}}^{+}(x)\leq \mu _{e_{i}}(x)$ for all $e_{i}\in J$ and for all $x\in
X.$

(iv)(a) $A_{e_{i}}^{-}(x)\leq \mu _{e_{i}}(x)\leq A_{e_{i}}^{+}(x)$ $\leq
\lambda _{e_{i}\text{ }}(x)$ for all $e_{i}\in I$ and for all $x\in X$.

(iv)(b) $\mu _{e_{i}}(x)\leq B_{e_{i}}^{-}(x)\leq \lambda _{e_{i}\text{ }%
}(x)\leq B_{e_{i}}^{+}(x)$ for all $e_{i}\in J$ and for all $x\in X.$ Since
P-union of $\left( \widetilde{F,}\text{ }I\right) $ and $\left( \widetilde{G,%
}\text{ }J\right) $ is denoted and defined as$\left( \widetilde{F,}\text{ }%
I\right) \cup _{P}\left( \widetilde{G,}\text{ }J\right) =\left( \widetilde{H,%
}\text{ }C\right) $ where $C=I\cup J$ and%
\begin{equation*}
\widetilde{H}(e_{i})=\left\{ 
\begin{array}{c}
\begin{array}{c}
\widetilde{F}(e_{i})\text{ If }e_{i}\in I-J \\ 
\widetilde{G}(e_{i})\text{ If }e_{i}\in J-I%
\end{array}%
\text{ \ \ \ \ \ \ \ \ \ \ \ \ } \\ 
\widetilde{F}(e_{i})\vee _{P}\widetilde{G}(e_{i})\text{ If }e_{i}\in I\cap J,%
\end{array}%
\right.
\end{equation*}%
where $\widetilde{F}(e_{i})\vee _{P}\widetilde{G}(e_{i})$ is defined as%
\begin{equation*}
\widetilde{F}(e_{i})\vee _{P}\widetilde{G}(e_{i})=\widetilde{H}%
(e_{i})=\left\{ 
\begin{array}{c}
<x,\text{ }r\max \{A_{e_{i}\text{ }}(x),B_{e_{i}\text{ }}(x)\},(\text{ }%
\lambda _{e_{i}\text{ }}\vee \mu _{e_{i}\text{ }})(x)>: \\ 
\text{ }x\in X,\text{ }e_{i}\in I\cap J.%
\end{array}%
\right\}
\end{equation*}%
Case 1: If $\widetilde{H}(e_{i})=\widetilde{F}(e_{i})$ that is if $e_{i}\in
I-J,$ then from (i)(a) and (ii)(a), we have $\lambda
_{e_{i}}(x)=A_{e_{i}}^{-}(x)$ and $\lambda _{e_{i}}(x)=A_{e_{i}}^{+}(x)$ for
all $e_{i}\in I$ and for all $x\in X.$ Thus $A_{e_{i}}^{-}(x)\leq \lambda
_{e_{i}}(x)\leq A_{e_{i}}^{+}(x)$ for all $e_{i}\in I-J$ and for all $x\in
X. $

Case 2: If $\widetilde{H}(e_{i})=\widetilde{G}(e_{i})$ that is if $e_{i}\in
J-I,$ then from (i)(b) and (ii)(b), we have $\mu
_{e_{i}}(x)=B_{e_{i}}^{-}(x) $ and $\mu _{e_{i}}(x)=B_{e_{i}}^{+}(x)$ for
all $e_{i}\in J$ and for all $x\in X.$ Thus $B_{e_{i}}^{-}(x)\leq \mu
_{e_{i}}(x)\leq B_{e_{i}}^{+}(x)$ for all $e_{i}\in J-I$ and for all $x\in
X. $

Case 3: If $\widetilde{H}(e_{i})=\widetilde{F}(e_{i})\vee _{P}\widetilde{G}%
(e_{i})$ that is if $e_{i}\in I\cap J,$ then from (i)(a) and (i)(b), we have 
$A_{e_{i}}^{-}(x)\leq \lambda _{e_{i}}(x)\leq A_{e_{i}}^{+}(x)$ for all $%
e_{i}\in I$ and for all $x\in X$ and $B_{e_{i}}^{-}(x)\leq \mu
_{e_{i}}(x)\leq B_{e_{i}}^{+}(x)$ for all $e_{i}\in J$ and for all $x\in X.$

Hence if $e_{i}\in I\cap J,$ then $\max
\{A_{e_{i}}^{-}(x),B_{e_{i}}^{-}(x)\}\leq (\lambda _{e_{i}}\vee \mu
_{e_{i}})(x)\leq \max \{A_{e_{i}}^{+}(x),B_{e_{i}}^{+}(x)\}.$ Thus, in all
the three cases $\left( \widetilde{F,}\text{ }I\right) \cup _{P}\left( 
\widetilde{G,}\text{ }J\right) $ is an ICSS in $X.$
\end{proof}

\begin{example}
Let$\ \left( \widetilde{F,}\text{ }I\right) =\left\{ \widetilde{F}%
(e_{i})=\left\{ <x,\text{ }A_{e_{i}}(x),\text{ }\lambda _{e_{i}}(x)>:\text{ }%
x\in X\right\} \text{ }e_{i}\in I\right\} $ be cubic soft set defined by 
\begin{equation*}
\begin{tabular}{|l|l|l|l|}
\hline
$p$ & $%
\begin{array}{c}
\widetilde{F}(e_{1})=\mathcal{A}_{1}= \\ 
<A_{e_{1}}(p),\text{ \ \ \ \ \ }\lambda _{e_{1}}(p)>%
\end{array}%
$ & $%
\begin{array}{c}
\widetilde{F}(e_{2})=\mathcal{A}_{2}= \\ 
<A_{e_{2}}(p),\text{ \ \ \ \ \ }\lambda _{e_{2}}(p)>%
\end{array}%
$ & $%
\begin{array}{c}
\widetilde{F}(e_{3})=\mathcal{A}_{3}= \\ 
<A_{e_{3}}(p),\text{ \ \ \ \ \ }\lambda _{e_{3}}(p)>%
\end{array}%
$ \\ \hline
$p_{1}$ & $\ \left[ 0.2,0.4\right] ,$ \ \ \ \ \ $0.5$ & $\ \left[ 0.7,0.9%
\right] ,\text{ \ \ \ \ \ \ }0.6$ & $\ \left[ 0.4,0.6\right] ,\text{ \ \ \ \
\ \ }0.3$ \\ \hline
$p_{2}$ & $\ \left[ 0.3,0.5\right] ,\text{ \ \ \ \ \ }0.6$ & $\ \left[
0.2,0.4\right] ,\text{ \ \ \ \ \ \ }0.5$ & $\ \left[ 0.5,0.7\right] ,\text{
\ \ \ \ \ \ }0.4$ \\ \hline
$p_{3}$ & $\ \left[ 0.4,0.6\right] ,\text{ \ \ \ \ \ }0.7$ & $\ \left[
0.3,0.5\right] ,\text{ \ \ \ \ \ \ }0.6$ & $\ \left[ 0.4,0.6\right] ,\text{
\ \ \ \ \ \ }0.3$ \\ \hline
\end{tabular}%
\end{equation*}%
and $\left( \widetilde{G,}\text{ }J\right) =\left\{ \widetilde{G}(e_{i})=%
\mathcal{B}_{i}=\left\{ <p,\text{ }B_{e_{i}}(p),\ \mu _{e_{i}}(p)>:\text{ }%
p\in X\right\} \text{ }e_{i}\in J,i=1,2,3\right\} $ cubic soft set defined by%
\begin{equation*}
\begin{tabular}{|l|l|l|l|}
\hline
$p$ & $%
\begin{array}{c}
\widetilde{G}(e_{1})=\mathcal{B}_{1}= \\ 
<B_{e_{1}}(p),\text{ \ \ \ \ }\ \mu _{e_{1}}(p)>%
\end{array}%
$ & $%
\begin{array}{c}
\widetilde{G}(e_{2})=\mathcal{B}_{2}= \\ 
<B_{e_{2}}(p),\text{\ \ \ \ \ }\ \mu _{e_{2}}(p)>%
\end{array}%
$ & $%
\begin{array}{c}
\widetilde{G}(e_{3})=\mathcal{B}_{3}= \\ 
<B_{e_{3}}(p),\text{ \ \ \ \ }\ \mu _{e_{3}}(p)>%
\end{array}%
$ \\ \hline
$p_{1}$ & $\ \left[ 0.4,0.6\right] ,\text{ \ \ \ \ \ }0.3$ & $\ \left[
0.5,0.7\right] ,\text{ \ \ \ \ \ }0.8$ & $\ \left[ 0.2,0.4\right] ,\text{ \
\ \ \ \ \ }0.5$ \\ \hline
$p_{2}$ & $\ \left[ 0.5,0.7\right] ,\text{ \ \ \ \ \ }0.4$ & $\ \left[
0.4,0.6\right] ,\text{ \ \ \ \ \ }0.3$ & $\ \left[ 0.3,0.5\right] ,\text{ \
\ \ \ \ \ }0.6$ \\ \hline
$p_{3}$ & $\ \left[ 0.6,0.8\right] ,\text{ \ \ \ \ \ }0.5$ & $\ \left[
0.5,0.7\right] ,\text{ \ \ \ \ \ }0.4$ & $\ \left[ 0.2,0.4\right] ,\text{ \
\ \ \ \ \ }0.5$ \\ \hline
\end{tabular}%
\end{equation*}%
Above $\left( \widetilde{F,}\text{ }I\right) $ and $\left( \widetilde{G,}%
\text{ }J\right) $ are ECSSs and constructed so that their corresponding $%
\left( \widetilde{F,}\text{ }I\right) ^{\ast }$ and $\left( \widetilde{G,}%
\text{ }J\right) ^{\ast }$are ICSSs. Now consider $\left( \widetilde{F,}%
\text{ }I\right) $ $\cup _{P}\left( \widetilde{G,}\text{ }J\right) $ is
defined by%
\begin{eqnarray*}
&&%
\begin{tabular}{|l|l|}
\hline
$p$ & $%
\begin{array}{c}
\widetilde{F}{\small (e}_{1}{\small )\vee }_{P}\widetilde{G}{\small (e}_{1}%
{\small )=} \\ 
<{\small r}\max \left\{ A_{e_{1}}(p){\small ,}B_{e_{1}}(p)\right\} ,{\small %
(\lambda }_{e_{1}}{\small \vee \mu }_{e_{1}}{\small )p>}%
\end{array}%
$ \\ \hline
$p_{1}$ & ${\small \ \ \ \ \ \ \ \ \ \ \ \ \ \ \ \ \ \ \ }\left[ {\small %
0.4,0.6}\right] ,\text{ \ \ \ \ \ \ }{\small 0.5}$ \\ \hline
$p_{2}$ & ${\small \ \ \ \ \ \ \ \ \ \ \ \ \ \ \ \ \ \ \ }\left[ {\small %
0.5,0.7}\right] ,\text{ \ \ \ \ \ \ }{\small 0.6}$ \\ \hline
$p_{3}$ & ${\small \ \ \ \ \ \ \ \ \ \ \ \ \ \ \ \ \ \ \ }\left[ {\small %
0.6,0.8}\right] ,\text{ \ \ \ \ \ \ }{\small 0.7}$ \\ \hline
\end{tabular}
\\
&&%
\begin{tabular}{|l|l|}
\hline
$p$ & $%
\begin{array}{c}
\widetilde{F}{\small (e}_{2}{\small )\vee }_{P}\widetilde{G}{\small (e}_{2}%
{\small )=} \\ 
<{\small r}\max \left\{ A_{e_{2}}(p){\small ,}B_{e_{2}}(p)\right\} ,{\small %
(\lambda }_{e_{2}}{\small \vee \mu }_{e_{2}}{\small )p>}%
\end{array}%
$ \\ \hline
$p_{1}$ & ${\small \ \ \ \ \ \ \ \ \ \ \ \ \ \ \ \ \ \ }\left[ {\small %
0.7,0.9}\right] ,\text{ \ \ \ \ \ \ }{\small 0.8}$ \\ \hline
$p_{2}$ & ${\small \ \ \ \ \ \ \ \ \ \ \ \ \ \ \ \ \ \ }\left[ {\small %
0.4,0.6}\right] ,\text{ \ \ \ \ \ \ }{\small 0.5}$ \\ \hline
$p_{3}$ & ${\small \ \ \ \ \ \ \ \ \ \ \ \ \ \ \ \ \ \ }\left[ {\small %
0.5,0.7}\right] ,\text{ \ \ \ \ \ \ }{\small 0.6}$ \\ \hline
\end{tabular}
\\
&&%
\begin{tabular}{|l|l|}
\hline
$p$ & $%
\begin{array}{c}
\widetilde{F}{\small (e}_{3}{\small )\vee }_{P}\widetilde{G}{\small (e}_{3}%
{\small )=} \\ 
<{\small r}\max \left\{ A_{e_{3}}(p){\small ,}B_{e_{3}}(p)\right\} \text{ }%
{\small (\lambda }_{e_{3}}{\small \vee \mu }_{e_{3}}{\small )p>}%
\end{array}%
$ \\ \hline
$p_{1}$ & $\ \ \ \ \ \ \ \ \ \ \ \ \ \ \ \ \left[ {\small 0.4,0.6}\right] ,%
\text{ \ \ \ \ \ }{\small 0.5}$ \\ \hline
$p_{2}$ & $\ \ \ \ \ \ \ \ \ \ \ \ \ \ \ \ \left[ {\small 0.5,0.7}\right] ,%
\text{ \ \ \ \ \ }{\small 0.6}$ \\ \hline
$p_{3}$ & $\ \ \ \ \ \ \ \ \ \ \ \ \ \ \ \ \left[ {\small 0.4,0.6}\right] ,%
\text{ \ \ \ \ \ }{\small 0.5}$ \\ \hline
\end{tabular}%
\end{eqnarray*}%
Which is an ICSS.
\end{example}

\begin{theorem}
For two ECSSs $\left( \widetilde{F,}\text{ }I\right) =\left\{ \widetilde{F}%
(e_{i})=\left\{ <x,\text{ }A_{e_{i}}(x),\text{ }\lambda _{e_{i}}(x)>:\text{ }%
x\in X\right\} \text{ }e_{i}\in I,\right\} $ and $\left( \widetilde{G,}\text{
}J\right) =\left\{ \widetilde{G}(e_{i})=\left\{ <x,\text{ }B_{e_{i}}(x),%
\text{ }\mu _{e_{i}}(x)>:\text{ }x\in X\right\} e_{i}\in J\right\} $ in $X,$
if

$\left( \widetilde{F,}\text{ }I\right) ^{\ast }=\left\{ \widetilde{F}%
(e_{i})=\left\{ <x,\text{ }A_{i\text{ }}(x),\text{ }\mu _{e_{i}\text{ }}(x)>:%
\text{ }x\in X\right\} \text{ }e_{i}\in I,\right\} $ and

$\left( \widetilde{G,}\text{ }J\right) ^{\ast }=\left\{ \widetilde{G}%
(e_{i})=\left\{ <x,\text{ }B_{i\text{ }}(x),\text{ }\lambda _{e_{i}\text{ }%
}(x)>:\text{ }x\in X\right\} e_{i}\in J\right\} $ are ICSSs in $X,$ then $%
\left( \widetilde{F,}\text{ }I\right) $ $\cap _{P}\left( \widetilde{G,}\text{
}J\right) $ is an ICSS in $X$.
\end{theorem}

\begin{proof}
By similar way to Theorem \ref{P1}, we can obtain the result.
\end{proof}

\begin{theorem}
Let $\left( \widetilde{F,}\text{ }I\right) =\left\{ \widetilde{F}%
(e_{i})=\left\{ <x,\text{ }A_{e_{i}}(x),\text{ }\lambda _{e_{i}}(x)>:\text{ }%
x\in X\right\} \text{ }e_{i}\in I\right\} $ and $\left( \widetilde{G,}\text{ 
}J\right) =\left\{ \widetilde{G}(e_{i})=\left\{ <x,\text{ }B_{e_{i}}(x),%
\text{ }\mu _{e_{i}}(x)>:\text{ }x\in X\right\} \text{ }e_{i}\in J\right\} $
be ECSSs in $X$ such that

$\left( \widetilde{F,}\text{ }I\right) ^{\ast }=\left\{ \widetilde{F}%
(e_{i})=\left\{ <x,\text{ }A_{e_{i}}(x),\text{ }\mu _{e_{i}}(x)>:\text{ }%
x\in X\right\} \text{ }e_{i}\in I\right\} $ and

$\left( \widetilde{G,}\text{ }J\right) ^{\ast }=\left\{ \widetilde{G}%
(e_{i})=\left\{ <x,\text{ }B_{e_{i}}(x),\text{ }\lambda _{e_{i}}(x)>:\text{ }%
x\in X\right\} \text{ }e_{i}\in J\right\} $ be ECSSs. Then, the P-union of $%
\left( \widetilde{F,}\text{ }I\right) $ and $\left( \widetilde{G,}\text{ }%
J\right) $ is an ECSS in $X$.
\end{theorem}

\begin{proof}
Since $\left( \widetilde{F,}\text{ }I\right) $,$\left( \widetilde{G,}\text{ }%
J\right) ,$ $\left( \widetilde{F,}\text{ }I\right) ^{\ast }$ and $\left( 
\widetilde{G,}\text{ }J\right) ^{\ast }$ are ECSSs so by using the
definition of an external cubic soft set for $\left( \widetilde{F,}\text{ }%
I\right) $,$\left( \widetilde{G,}\text{ }J\right) ,$ $\left( \widetilde{F,}%
\text{ }I\right) ^{\ast }$ and $\left( \widetilde{G,}\text{ }J\right) ^{\ast
}$, we have $\lambda _{e_{i}}(x)\notin (A_{e_{i}}^{-}(x),A_{e_{i}}^{+}(x))$
for all $e_{i}\in I$ and for all $x\in X$, $\mu _{e_{i}}(x)\notin
(B_{e_{i}}^{-}(x),B_{e_{i}}^{+}(x))$ for all $e_{i}\in J$ and for all $x\in
X,$ $\mu _{e_{i}}(x)\notin (A_{e_{i}}^{-}(x),A_{e_{i}}^{+}(x))$ for all $%
e_{i}\in I$ and for all $x\in X$ and $\lambda _{e_{i}}(x)\notin
(B_{e_{i}}^{-}(x),B_{e_{i}}^{+}(x))$ for all $e_{i}\in J$ and for all $x\in
X $ respectively. Thus we have $(\lambda _{e_{i}}\vee \mu _{e_{i}})(x)\notin
(\max \{A_{e_{i}}^{-}(x),B_{e_{i}}^{-}(x)\},\max
\{A_{e_{i}}^{+}(x),B_{e_{i}}^{+}(x)\})$ for all $e_{i}\in I\cap J$ and for
all $x\in X.$

Also since $\left( \widetilde{F,}\text{ }I\right) \cup _{P}\left( \widetilde{%
G,}\text{ }J\right) =\left( \widetilde{H,}\text{ }C\right) $ where $C=I\cup
J $ and%
\begin{equation*}
\widetilde{H}(e_{i})=\left\{ 
\begin{array}{c}
\begin{array}{c}
\widetilde{F}(e_{i})\text{ If }e_{i}\in I-J \\ 
\widetilde{G}(e_{i})\text{ If }e_{i}\in J-I%
\end{array}%
\text{ \ \ \ \ \ \ \ \ \ \ \ \ } \\ 
\widetilde{F}(e_{i})\vee _{P}\widetilde{G}(e_{i})\text{ If }e_{i}\in I\cap J,%
\end{array}%
\right.
\end{equation*}%
where $\widetilde{F}(e_{i})\vee _{P}\widetilde{G}(e_{i})$ is defined as%
\begin{equation*}
\widetilde{F}(e_{i})\vee _{P}\widetilde{G}(e_{i})=\widetilde{H}%
(e_{i})=\left\{ 
\begin{array}{c}
<x,\text{ }r\max \{A_{e_{i}}(x),B_{e_{i}}(x)\},(\lambda _{e_{i}}\vee \mu
_{e_{i}})(x)>: \\ 
\text{ }x\in X,\text{ }e_{i}\in I\cap J.%
\end{array}%
\right\}
\end{equation*}%
So, by definitions of an external cubic soft set $\left( \widetilde{F,}\text{
}I\right) \cup _{P}\left( \widetilde{G,}\text{ }J\right) $ is an ECSS in $X$.
\end{proof}

\begin{theorem}
Let $\left( \widetilde{F,}\text{ }I\right) =\left\{ \widetilde{F}%
(e_{i})=\left\{ <x,\text{ }A_{e_{i}}(x),\text{ }\lambda _{e_{i}}(x)>:\text{ }%
x\in X\right\} \text{ }e_{i}\in I\right\} $ and $\left( \widetilde{G,}\text{ 
}J\right) =\left\{ \widetilde{G}(e_{i})=\left\{ <x,\text{ }B_{e_{i}}(x),%
\text{ }\mu _{e_{i}}(x)>:\text{ }x\in X\right\} \text{ }e_{i}\in J\right\} $
be ECSSs in $X$ such that 
\begin{equation*}
(\lambda _{e_{i}}\wedge \mu _{e_{i}})(x)\in \left[ 
\begin{array}{c}
\min \{\max \{A_{e_{i}}^{+}(x),B_{e_{i}}^{-}(x)\},\max
\{A_{e_{i}}^{-}(x),B_{e_{i}}^{+}(x)\}\}, \\ 
\max \{\min \{A_{e_{i}}^{+}(x),B_{e_{i}}^{-}(x)\},\min
\{A_{e_{i}}^{-}(x),B_{e_{i}\text{ }}^{+}(x)\}\}%
\end{array}%
\right)
\end{equation*}%
for all $e_{i}\in I,$ for all $e_{i}\in J$ and for all $x\in X.$ Then $%
\left( \widetilde{F,}\text{ }I\right) \cap _{P}\left( \widetilde{G,}\text{ }%
J\right) $ is also an ECSS.
\end{theorem}

\begin{proof}
Consider $\left( \widetilde{F,}\text{ }I\right) \cap _{P}\left( \widetilde{G,%
}\text{ }J\right) =\left( \widetilde{H,}\text{ }C\right) $, $C=I\cap J$

where $\widetilde{H}(e_{i})=\widetilde{F}(e_{i})\wedge _{P}\widetilde{G}%
(e_{i})$ is defined as%
\begin{equation*}
\widetilde{F}(e_{i})\wedge _{P}\widetilde{G}(e_{i})=\widetilde{H}%
(e_{i})=\left\{ 
\begin{array}{c}
<x,\text{ }r\min \{A_{e_{i}}(x),B_{e_{i}}(x)\},(\lambda _{e_{i}}\wedge \mu
_{e_{i}})(x)>: \\ 
\text{ }x\in X,\text{ }e_{i}\in I\cap J%
\end{array}%
\right\}
\end{equation*}%
For each $e_{i}\in I\cap J,$ take ${\Large \mathcal{\alpha }}_{e_{i}}=\min
\{\max \{A_{e_{i}}^{+}(x),B_{e_{i}}^{-}(x)\}$,$\max
\{A_{e_{i}}^{-}(x),B_{e_{i}}^{+}(x)\}\}$ and

${\Large \mathcal{\beta }}_{e_{i}}=\max \{\min
\{A_{e_{i}}^{+}(x),B_{e_{i}}^{-}(x)\}$,$\min \{A_{i\text{ }%
}^{-}(x),B_{e_{i}}^{+}(x)\}\}.$ Then ${\Large \mathcal{\alpha }}_{e_{i}}$ is
one of $A_{e_{i}}^{-}(x)$, $B_{e_{i}}^{-}(x)$, $A_{e_{i}}^{+}(x)$ and $%
B_{e_{i}}^{+}(x).$ We consider ${\Large \mathcal{\alpha }}%
_{e_{i}}=A_{e_{i}}^{-}(x)$ or $A_{e_{i}}^{+}(x)$ only, as the remaining
cases are similar to this one. If ${\Large \mathcal{\alpha }}%
_{e_{i}}=A_{e_{i}}^{-}(x),$ then $B_{e_{i}}^{-}(x)\leq B_{e_{i}}^{+}(x)\leq
A_{e_{i}}^{-}(x)\leq A_{e_{i}}^{+}(x)$ and so ${\Large \mathcal{\beta }}%
_{e_{i}}=B_{e_{i}}^{+}(x)$ thus $B_{e_{i}}^{-}(x)=(A_{e_{i}}\cap
B_{e_{i}})^{-}(x)\leq (A_{e_{i}}\cap B_{e_{i}})^{+}(x)=B_{e_{i}}^{+}(x)=%
{\Large \mathcal{\beta }}_{e_{i}}<(\lambda _{e_{i}}\wedge \mu _{e_{i}})(x).$
Hence $(\lambda _{e_{i}}\wedge \mu _{e_{i}})(x)\notin ((A_{e_{i}}\cap
B_{e_{i}})^{-}(x)$,$(A_{e_{i}}\cap B_{e_{i}})^{+}(x)).$ If ${\Large \mathcal{%
\alpha }}_{e_{i}}=A_{e_{i}}^{+}(x),$ then $B_{e_{i}}^{-}(x)\leq
A_{e_{i}}^{+}(x)\leq B_{e_{i}}^{+}(x)$ and so ${\Large \mathcal{\beta }}%
_{e_{i}}=\max \{A_{e_{i}}^{-}(x),B_{e_{i}}^{-}(x)\}.$ Assume that ${\Large 
\mathcal{\beta }}_{e_{i}}=A_{e_{i}}^{-}(x),$ then $B_{e_{i}}^{-}(x)\leq
A_{e_{i}}^{-}(x)<(\lambda _{e_{i}}\wedge \mu _{e_{i}})(x)\leq
A_{e_{i}}^{+}(x)\leq B_{e_{i}}^{+}(x).$ So from this, we can write $%
B_{e_{i}}^{-}(x)\leq A_{e_{i}}^{-}(x)<(\lambda _{e_{i}}\wedge \mu
_{e_{i}})(x)<A_{e_{i}}^{+}(x)\leq B_{e_{i}}^{+}(x)$ or $B_{e_{i}}^{-}(x)\leq
A_{e_{i}}^{-}(x)<(\lambda _{e_{i}}\wedge \mu
_{e_{i}})(x)=A_{e_{i}}^{+}(x)\leq B_{e_{i}}^{+}(x).$

For the case $B_{e_{i}}^{-}(x)\leq A_{e_{i}}^{-}(x)<(\lambda _{e_{i}}\wedge
\mu _{e_{i}})(x)<A_{e_{i}}^{+}(x)\leq B_{e_{i}}^{+}(x)$ it is contradiction
to the fact that $\left( \widetilde{F,}\text{ }I\right) $ and $\left( 
\widetilde{G,}\text{ }J\right) $ are ECSSs. For the case $%
B_{e_{i}}^{-}(x)\leq A_{e_{i}}^{-}(x)<(\lambda _{e_{i}}\wedge \mu
_{e_{i}})(x)=A_{e_{i}}^{+}(x)\leq B_{e_{i}}^{+}(x),$ we have $(\lambda
_{e_{i}}\wedge \mu _{e_{i}})(x)\notin ((A_{e_{i}}\cap B_{e_{i}})^{-}(x)$, $%
(A_{e_{i}}\cap B_{e_{i}})^{+}(x))$ because, $(\lambda _{e_{i}}\wedge \mu
_{e_{i}})(x)=A_{e_{i}}^{+}(x)=(A_{e_{i}}\cap B_{e_{i}})^{+}(x).$ Again
assume that ${\Large \mathcal{\beta }}_{e_{i}}=B_{e_{i}}^{-}(x),$ then $%
A_{e_{i}}^{-}(x)\leq B_{e_{i}}^{-}(x)<(\lambda _{e_{i}}\wedge \mu
_{e_{i}})(x)\leq A_{e_{i}}^{+}(x)\leq B_{e_{i}}^{+}(x).$ From this we can
write $A_{e_{i}}^{-}(x)\leq B_{e_{i}}^{-}(x)<(\lambda _{e_{i}}\wedge \mu
_{e_{i}})(x)<A_{e_{i}}^{+}(x)\leq B_{e_{i}}^{+}(x)$ or $A_{e_{i}}^{-}(x)\leq
B_{e_{i}}^{-}(x)<(\lambda _{e_{i}}\wedge \mu
_{e_{i}})(x)=A_{e_{i}}^{+}(x)\leq B_{e_{i}}^{+}(x).$

For the case $A_{e_{i}}^{-}(x)\leq B_{e_{i}}^{-}(x)<(\lambda _{e_{i}}\wedge
\mu _{e_{i}})(x)<A_{e_{i}}^{+}(x)\leq B_{e_{i}}^{+}(x)$ it is contradiction
to the fact that $\left( \widetilde{F,}\text{ }I\right) $ and $\left( 
\widetilde{G,}\text{ }J\right) $ are ECSSs. And if we take the case $%
A_{e_{i}}^{-}(x)\leq B_{e_{i}}^{-}(x)<(\lambda _{e_{i}}\wedge \mu
_{e_{i}})(x)=A_{e_{i}}^{+}(x)\leq B_{e_{i}}^{+}(x)$ we get $(\lambda
_{e_{i}}\wedge \mu _{e_{i}})(x)\notin ((A_{e_{i}}\cap B_{e_{i}})^{-}(x)$, $%
(A_{e_{i}}\cap B_{e_{i}})^{+}(x))$ because, $(\lambda _{e_{i}}\wedge \mu
_{e_{i}})(x)=A_{e_{i}}^{+}(x)=(A_{e_{i}}\cap B_{e_{i}})^{+}(x).$ Hence in
all the cases, $\left( \widetilde{F,}\text{ }I\right) \cap _{P}\left( 
\widetilde{G,}\text{ }J\right) $ is an ECSS in $X$.
\end{proof}

\begin{theorem}
Let $\left( \widetilde{F,}\text{ }I\right) =\left\{ \widetilde{F}%
(e_{i})=\left\{ <x,\text{ }A_{e_{i}}(x),\text{ }\lambda _{e_{i}}(x)>:\text{ }%
x\in X\right\} \text{ }e_{i}\in I,\right\} $ and $\left( \widetilde{G,}\text{
}J\right) =\left\{ \widetilde{G}(e_{i})=\left\{ <x,\text{ }B_{e_{i}}(x),%
\text{ }\mu _{e_{i}}(x)>:\text{ }x\in X\right\} \text{ }e_{i}\in J\right\} $
be cubic soft sets in $X$ such that 
\begin{eqnarray*}
&&\min \{\max \{A_{e_{i}}^{+}(x),B_{e_{i}}^{-}(x)\},\max
\{A_{e_{i}}^{-}(x),B_{e_{i}}^{+}(x)\}\} \\
&=&(\lambda _{e_{i}}\wedge \mu _{e_{i}})(x) \\
&=&\max \{\min \{A_{e_{i}}^{+}(x),B_{e_{i}}^{-}(x)\},\min
\{A_{e_{i}}^{-}(x),B_{e_{i}\text{ }}^{+}(x)\}\}
\end{eqnarray*}%
for all $e_{i}\in I,$ for all $e_{i}\in J$ and for all $x\in X$. Then, the $%
\left( \widetilde{F,}\text{ }I\right) \cap _{P}\left( \widetilde{G,}\text{ }%
J\right) $ is both an ECSS and an ICSS in $X$.
\end{theorem}

\begin{proof}
Consider $\left( \widetilde{F,}\text{ }I\right) \cap _{P}\left( \widetilde{G,%
}\text{ }J\right) =\left( \widetilde{H,}\text{ }C\right) ,$ where $C=I\cap
J. $ $\widetilde{H}(e_{i})=\widetilde{F}(e_{i})\wedge _{P}\widetilde{G}%
(e_{i})$ is defined as $\widetilde{F}(e_{i})\wedge _{P}\widetilde{G}(e_{i})=$

$\left\{ \widetilde{H}(e_{i})=\left\{ <x,\text{ }r\min
\{A_{e_{i}}(x),B_{e_{i}}(x)\},(\lambda _{e_{i}}\wedge \mu _{e_{i}})(x)>:%
\text{ }x\in X\right\} \text{ }e_{i}\in I\cap J\right\} .$ For each $%
e_{i}\in I\cap J$ take

${\Large \mathcal{\alpha }}_{e_{i}}=\min \{\max
\{A_{e_{i}}^{+}(x),B_{e_{i}}^{-}(x)\}$,$\max
\{A_{e_{i}}^{-}(x),B_{e_{i}}^{+}(x)\}\}$ and

${\Large \mathcal{\beta }}_{e_{i}}=\max \{\min
\{A_{e_{i}}^{+}(x),B_{e_{i}}^{-}(x)\}$,$\min \{A_{e_{i}\text{ }%
}^{-}(x),B_{e_{i}}^{+}(x)\}\}.$

Then ${\Large \mathcal{\alpha }}_{e_{i}}$ is one of $A_{e_{i}}^{-}(x)$, $%
B_{e_{i}}^{-}(x)$, $A_{e_{i}\text{ }}^{+}(x)$ and $B_{e_{i}}^{+}(x).$ We
consider ${\Large \mathcal{\alpha }}_{e_{i}}=A_{e_{i}}^{-}(x)$ or $%
A_{e_{i}}^{+}(x)$ only, as the remaining cases are similar to this one. If $%
{\Large \mathcal{\alpha }}_{e_{i}}=A_{e_{i}}^{-}(x)$ then $%
B_{e_{i}}^{-}(x)\leq B_{e_{i}}^{+}(x)\leq A_{e_{i}}^{-}(x)\leq
A_{e_{i}}^{+}(x)$ and so ${\Large \mathcal{\beta }}_{e_{i}}=B_{e_{i}}^{+}(x)$
this implies that $A_{e_{i}}^{-}(x)={\Large \mathcal{\alpha }}%
_{e_{i}}=(\lambda _{e_{i}}\wedge \mu _{e_{i}})(x)={\Large \mathcal{\beta }}%
_{e_{i}}=B_{e_{i}}^{+}(x).$ Thus $B_{e_{i}}^{-}(x)\leq
B_{e_{i}}^{+}(x)=(\lambda _{e_{i}}\wedge \mu
_{e_{i}})(x)=A_{e_{i}}^{-}(x)\leq A_{e_{i}}^{+}(x),$ which implies that $%
(\lambda _{e_{i}}\wedge \mu _{e_{i}})(x)=B_{e_{i}}^{+}(x)=(A_{e_{i}}\cap
B_{e_{i}})^{+}(x).$ Hence $(\lambda _{e_{i}}\wedge \mu _{e_{i}})(x)\notin
((A_{e_{i}}\cap B_{e_{i}})^{-}(x)$,$(A_{e_{i}}\cap B_{e_{i}})^{+}(x))$ and $%
(A_{i}\cap B_{i})^{-}(x)\leq (\lambda _{e_{i}}\wedge \mu _{e_{i}})(x)\leq
(A_{e_{i}}\cap B_{e_{i}})^{+}(x).$ If ${\Large \mathcal{\alpha }}%
_{e_{i}}=A_{e_{i}}^{+}(x),$ then $B_{e_{i}}^{-}(x)\leq A_{e_{i}}^{+}(x)\leq
B_{e_{i}}^{+}(x)$ and so $(\lambda _{e_{i}}\wedge \mu
_{e_{i}})(x)=A_{e_{i}}^{+}(x)=(A_{e_{i}}\cap B_{e_{i}})^{+}(x).$ Hence $%
(\lambda _{e_{i}}\wedge \mu _{e_{i}})(x)\notin ((A_{e_{i}}\cap
B_{e_{i}})^{-}(x)$, $(A_{e_{i}}\cap B_{e_{i}})^{+}(x))$ and $(A_{e_{i}}\cap
B_{e_{i}})^{-}(x)\leq (\lambda _{e_{i}}\wedge \mu _{e_{i}})(x)\leq
(A_{e_{i}}\cap B_{e_{i}})^{+}(x).$ Consequently, we note that $\left( 
\widetilde{F,}\text{ }I\right) \cap _{P}\left( \widetilde{G,}\text{ }%
J\right) $ is both an ECSS and an ICSS in $X$.
\end{proof}

\begin{theorem}
Let $\left( \widetilde{F,}\text{ }I\right) =\left\{ \widetilde{F}%
(e_{i})=\left\{ <x,\text{ }A_{e_{i}}(x),\text{ }\lambda _{e_{i}}(x)>:\text{ }%
x\in X\right\} \text{ }e_{i}\in I,\right\} $ and $\left( \widetilde{G,}\text{
}J\right) =\left\{ \widetilde{G}(e_{i})=\left\{ <x,\text{ }B_{e_{i}}(x),%
\text{ }\mu _{e_{i}}(x)>:\text{ }x\in X\right\} \text{ }e_{i}\in J\right\} $
are ECSSs in $X$ such that 
\begin{equation*}
(\lambda _{e_{i}}\vee \mu _{e_{i}})(x)\in \left( 
\begin{array}{c}
\min \{\max \{A_{e_{i}}^{+}(x),B_{e_{i}}^{-}(x)\},\max
\{A_{e_{i}}^{-}(x),B_{e_{i}}^{+}(x)\}\}, \\ 
\max \{\min \{A_{e_{i}}^{+}(x),B_{e_{i}}^{-}(x)\},\min
\{A_{e_{i}}^{-}(x),B_{e_{i}}^{+}(x)\}\}%
\end{array}%
\right]
\end{equation*}%
for all $e_{i}\in I$ for all $e_{i}\in J$ and for all $x\in X$ then the $%
\left( \widetilde{F,}\text{ }I\right) \cup _{P}\left( \widetilde{G,}\text{ }%
J\right) $ is an ECSS in X.
\end{theorem}

\begin{proof}
Since $\left( \widetilde{F,}\text{ }I\right) \cup _{P}\left( \widetilde{G,}%
\text{ }J\right) $ is defined as $\left( \widetilde{F,}\text{ }I\right) \cup
_{P}\left( \widetilde{G,}\text{ }J\right) =\left( \widetilde{H,}\text{ }%
C\right) $ where $C=I\cup J$ and%
\begin{equation*}
\widetilde{H}(e_{i})=\left\{ 
\begin{array}{c}
\begin{array}{c}
\widetilde{F}(e_{i})\text{ If }e_{i}\in I-J \\ 
\widetilde{G}(e_{i})\text{ If }e_{i}\in J-I%
\end{array}%
\text{ \ \ \ \ \ \ \ \ \ \ \ \ } \\ 
\widetilde{F}(e_{i})\vee _{P}\widetilde{G}(e_{i})\text{ If }e_{i}\in I\cap J,%
\end{array}%
\right.
\end{equation*}%
where $\widetilde{F}(e_{i})\vee _{P}\widetilde{G}(e_{i})$ is defined as%
\begin{equation*}
\widetilde{F}(e_{i})\vee _{P}\widetilde{G}(e_{i})=\widetilde{H}%
(e_{i})=\left\{ 
\begin{array}{c}
<x,\text{ }r\max \{A_{e_{i}}(x),B_{e_{i}}(x)\},(\lambda _{e_{i}}\vee \mu
_{e_{i}})(x)>: \\ 
\text{ }x\in X,\text{ }e_{i}\in I\cap J%
\end{array}%
\right\}
\end{equation*}%
If $e_{i}\in I\cap J,$ ${\Large \mathcal{\alpha }}_{e_{i}}=\min \{\max
\{A_{e_{i}}^{+}(x),B_{e_{i}}^{-}(x)\}$,$\max
\{A_{e_{i}}^{-}(x),B_{e_{i}}^{+}(x)\}\}$ and

${\Large \mathcal{\beta }}_{e_{i}}=\max \{\min
\{A_{e_{i}}^{+}(x),B_{e_{i}}^{-}(x)\}$,$\min \{A_{i\text{ }%
}^{-}(x),B_{e_{i}}^{+}(x)\}\}.$ Then ${\Large \mathcal{\alpha }}_{e_{i}}$ is
one of $A_{e_{i}}^{-}(x)$, $B_{e_{i}}^{-}(x)$, $A_{i\text{ }}^{+}(x)$ and $%
B_{e_{i}}^{+}(x).$ We consider ${\Large \mathcal{\alpha }}%
_{e_{i}}=A_{e_{i}}^{-}(x)$ or $A_{e_{i}}^{+}(x)$ only, as the remaining
cases are similar to this one. If ${\Large \mathcal{\alpha }}%
_{e_{i}}=A_{e_{i}}^{-}(x),$ then $B_{e_{i}}^{-}(x)\leq B_{e_{i}}^{+}(x)\leq
A_{e_{i}}^{-}(x)\leq A_{e_{i}}^{+}(x)$ and so ${\Large \mathcal{\beta }}%
_{e_{i}}=B_{e_{i}}^{+}(x).$ Thus $(A_{e_{i}}\cup
B_{e_{i}})^{-}(x)=A_{e_{i}}^{-}(x)={\Large \mathcal{\alpha }}%
_{e_{i}}>(\lambda _{e_{i}}\vee \mu _{e_{i}})(x).$ Hence $(\lambda
_{e_{i}}\vee \mu _{e_{i}})(x)\notin ((A_{e_{i}}\cup B_{e_{i}})^{-}(x)$,$%
(A_{e_{i}}\cup B_{e_{i}})^{+}(x)).$ If ${\Large \mathcal{\alpha }}%
_{e_{i}}=A_{e_{i}}^{+}(x),$ then $B_{e_{i}}^{-}(x)\leq A_{e_{i}}^{+}(x)\leq
B_{e_{i}}^{+}(x)$ and so ${\Large \mathcal{\beta }}_{e_{i}}=\max
\{A_{e_{i}}^{-}(x),B_{e_{i}}^{-}(x)\}.$ Assume ${\Large \mathcal{\beta }}%
_{e_{i}}=A_{e_{i}}^{-}(x),$ then $B_{e_{i}}^{-}(x)\leq A_{e_{i}}^{-}(x)\leq
(\lambda _{e_{i}}\vee \mu _{e_{i}})(x)<A_{e_{i}}^{+}(x)\leq
B_{e_{i}}^{+}(x), $ so from this we can write $B_{e_{i}}^{-}(x)\leq
A_{e_{i}}^{-}(x)<(\lambda _{e_{i}}\vee \mu _{e_{i}})(x)<A_{e_{i}}^{+}(x)\leq
B_{e_{i}}^{+}(x)$ or $B_{e_{i}}^{-}(x)\leq A_{e_{i}}^{-}(x)=(\lambda
_{e_{i}}\vee \mu _{e_{i}})(x)\leq A_{e_{i}}^{+}(x)\leq B_{e_{i}}^{+}(x).$

For the case $B_{e_{i}}^{-}(x)\leq A_{e_{i}}^{-}(x)<(\lambda _{e_{i}}\vee
\mu _{e_{i}})(x)<A_{e_{i}}^{+}(x)\leq B_{e_{i}}^{+}(x),$ it is contradiction
to the fact that $\left( \widetilde{F,}\text{ }I\right) $ and $\left( 
\widetilde{G,}\text{ }J\right) $ are ECSSs. For the case $%
B_{e_{i}}^{-}(x)\leq A_{e_{i}}^{-}(x)=(\lambda _{e_{i}}\vee \mu
_{e_{i}})(x)\leq A_{e_{i}}^{+}(x)\leq B_{e_{i}}^{+}(x),$ we have $(\lambda
_{e_{i}}\vee \mu _{e_{i}})(x)\notin ((A_{e_{i}}\cup B_{e_{i}})^{-}(x)$,$%
(A_{e_{i}}\cup B_{e_{i}})^{+}(x))$ because, $(A_{e_{i}}\cup
B_{e_{i}})^{-}(x)=A_{e_{i}}^{-}(x)=(\lambda _{e_{i}}\vee \mu _{e_{i}})(x).$
Again assume that

${\Large \mathcal{\beta }}_{e_{i}}=B_{e_{i}}^{-}(x),$ then $%
A_{e_{i}}^{-}(x)\leq B_{e_{i}}^{-}(x)\leq (\lambda _{e_{i}}\vee \mu
_{e_{i}})(x)\leq A_{e_{i}}^{+}(x)\leq B_{e_{i}}^{+}(x),$ so from this we can
write $A_{e_{i}}^{-}(x)\leq B_{e_{i}}^{-}(x)<(\lambda _{e_{i}}\vee \mu
_{e_{i}})(x)<A_{e_{i}}^{+}(x)\leq B_{e_{i}}^{+}(x)$ or $A_{e_{i}}^{-}(x)\leq
B_{e_{i}}^{-}(x)=(\lambda _{e_{i}}\vee \mu _{e_{i}})(x)<A_{e_{i}}^{+}(x)\leq
B_{e_{i}}^{+}(x).$ For the case $A_{e_{i}}^{-}(x)\leq
B_{e_{i}}^{-}(x)<(\lambda _{e_{i}}\vee \mu _{e_{i}})(x)<A_{e_{i}}^{+}(x)\leq
B_{e_{i}}^{+}(x),$ it is contradiction to the fact $\left( \widetilde{F,}%
\text{ }I\right) $ and $\left( \widetilde{G,}\text{ }J\right) $ are ECSSs.
And if we take the case $A_{e_{i}}^{-}(x)\leq B_{e_{i}}^{-}(x)=(\lambda
_{e_{i}}\vee \mu _{e_{i}})(x)\leq A_{e_{i}}^{+}(x)\leq B_{e_{i}}^{+}(x),$ we
get $(\lambda _{e_{i}}\vee \mu _{e_{i}})(x)\notin ((A_{e_{i}}\cup
B_{e_{i}})^{-}(x)$, $(A_{e_{i}}\cup B_{e_{i}})^{+}(x))$ because, $%
(A_{e_{i}}\cup B_{e_{i}})^{-}(x)=B_{e_{i}}^{-}(x)=(\lambda _{e_{i}}\vee \mu
_{e_{i}})(x).$ If $e_{i}\in I-J$ or $e_{i}\in J-I,$ then we have the
trivialresult. Hence $\left( \widetilde{F,}\text{ }I\right) \cup _{P}\left( 
\widetilde{G,}\text{ }J\right) $ is an ECSS in $X$.
\end{proof}

\begin{theorem}
Let $\left( \widetilde{F,}\text{ }I\right) =\left\{ \widetilde{F}%
(e_{i})=\left\{ <x,\text{ }A_{e_{i}}(x),\text{ }\lambda _{e_{i}}(x)>:\text{ }%
x\in X\right\} \text{ }e_{i}\in I,\right\} $ and $\left( \widetilde{G,}\text{
}J\right) =\left\{ \widetilde{G}(e_{i})=\left\{ <x,\text{ }B_{e_{i}}(x),%
\text{ }\mu _{e_{i}}(x)>:\text{ }x\in X\right\} e_{i}\in J\right\} $ are
ECSSs in $X$ such that 
\begin{equation*}
(\lambda _{e_{i}}\wedge \mu _{e_{i}})(x)\in \left( 
\begin{array}{c}
\min \{\max \{A_{e_{i}}^{+}(x),B_{e_{i}}^{-}(x)\},\max
\{A_{e_{i}}^{-}(x),B_{e_{i}}^{+}(x)\}\}, \\ 
\max \{\min \{A_{e_{i}}^{+}(x),B_{e_{i}}^{-}(x)\},\min
\{A_{e_{i}}^{-}(x),B_{i\text{ }}^{+}(x)\}\}%
\end{array}%
\right]
\end{equation*}%
for all $e_{i}\in I$ for all $e_{i}\in J$ and for all $x\in X$ then $\left( 
\widetilde{F,}\text{ }I\right) \cup _{R}\left( \widetilde{G,}\text{ }%
J\right) $ is also an ECSS in $X$.
\end{theorem}

\begin{proof}
Since $\left( \widetilde{F,}\text{ }I\right) \cup _{R}\left( \widetilde{G,}%
\text{ }J\right) $ is defined as $\left( \widetilde{F,}\text{ }I\right) \cup
_{R}\left( \widetilde{G,}\text{ }J\right) =\left( \widetilde{H,}\text{ }%
C\right) ,$ where $C=I\cup J$ and%
\begin{equation*}
\widetilde{H}(e_{i})=\left\{ 
\begin{array}{c}
\begin{array}{c}
\widetilde{F}(e_{i})\text{ If }e_{i}\in I-J \\ 
\widetilde{G}(e_{i})\text{ If }e_{i}\in J-I%
\end{array}%
\text{ \ \ \ \ \ \ \ \ \ \ \ \ } \\ 
\widetilde{F}(e_{i})\vee _{R}\widetilde{G}(e_{i})\text{ If }e_{i}\in I\cap J,%
\end{array}%
\right.
\end{equation*}%
where $\widetilde{F}(e_{i})\vee _{R}\widetilde{G}(e_{i})$ is defined as%
\begin{equation*}
\widetilde{F}(e_{i})\vee _{R}\widetilde{G}(e_{i})=\widetilde{H}%
(e_{i})=\left\{ 
\begin{array}{c}
<x,\text{ }r\max \{A_{e_{i}\text{ }}(x),B_{e_{i}\text{ }}(x)\},(\lambda
_{e_{i}\text{ }}\wedge \mu _{e_{i}\text{ }})(x)>: \\ 
\text{ }x\in X,\text{ }e_{i}\in I\cap J.%
\end{array}%
\right\}
\end{equation*}%
If $e_{i}\in I\cap J,$ take ${\Large \mathcal{\alpha }}_{e_{i}}=\min \{\max
\{A_{e_{i}}^{+}(x),B_{e_{i}}^{-}(x)\}$,$\max
\{A_{e_{i}}^{-}(x),B_{e_{i}}^{+}(x)\}\}$ and

${\Large \mathcal{\beta }}_{e_{i}}=\max \{\min
\{A_{e_{i}}^{+}(x),B_{e_{i}}^{-}(x)\}$,$\min \{A_{e_{i}\text{ }%
}^{-}(x),B_{e_{i}}^{+}(x)\}\}.$ Then ${\Large \mathcal{\alpha }}_{e_{i}}$ is
one of $A_{e_{i}}^{-}(x)$, $B_{e_{i}}^{-}(x)$, $A_{i\text{ }}^{+}(x)$ and $%
B_{e_{i}}^{+}(x).$ We consider ${\Large \mathcal{\alpha }}%
_{e_{i}}=B_{e_{i}}^{-}(x)$ or $B_{e_{i}}^{+}(x)$ only, as the remaining
cases are similar to this one. If ${\Large \mathcal{\alpha }}%
_{e_{i}}=B_{e_{i}}^{-}(x)$ then $A_{e_{i}}^{-}(x)\leq A_{e_{i}}^{+}(x)\leq
B_{e_{i}}^{-}(x)\leq B_{e_{i}}^{+}(x)$ and so ${\Large \mathcal{\beta }}%
_{e_{i}}=A_{e_{i}}^{+}(x).$ Thus $(A_{e_{i}}\cup
B_{e_{i}})^{-}(x)=B_{e_{i}}^{-}(x)={\Large \mathcal{\alpha }}%
_{e_{i}}>(\lambda _{e_{i}}\wedge \mu _{e_{i}})(x).$ Hence $(\lambda
_{e_{i}}\wedge \mu _{e_{i}})(x)\notin ((A_{e_{i}}\cup B_{e_{i}})^{-}(x)$,$%
(A_{e_{i}}\cup B_{e_{i}})^{+}(x)).$

If ${\Large \mathcal{\alpha }}_{e_{i}}=B_{e_{i}}^{+}(x),$ then $%
A_{e_{i}}^{-}(x)\leq B_{e_{i}}^{+}(x)\leq A_{e_{i}}^{+}(x)$ and so ${\Large 
\mathcal{\beta }}_{e_{i}}=\max \{A_{e_{i}}^{-}(x),B_{e_{i}}^{-}(x)\}.$
Assume ${\Large \mathcal{\beta }}_{e_{i}}=A_{e_{i}}^{-}(x),$ then we have $%
B_{e_{i}}^{-}(x)\leq A_{e_{i}}^{-}(x)\leq (\lambda _{e_{i}}\wedge \mu
_{e_{i}})(x)<B_{e_{i}}^{+}(x)\leq A_{e_{i}}^{+}(x).$ So from this we can
write $B_{e_{i}}^{-}(x)\leq A_{e_{i}}^{-}(x)<(\lambda _{e_{i}}\wedge \mu
_{e_{i}})(x)<B_{e_{i}}^{+}(x)\leq A_{e_{i}}^{+}(x)$ or $B_{e_{i}}^{-}(x)\leq
A_{e_{i}}^{-}(x)=(\lambda _{e_{i}}\wedge \mu _{e_{i}})(x)\leq
B_{e_{i}}^{+}(x)\leq A_{e_{i}}^{+}(x).$ For the case $B_{e_{i}}^{-}(x)\leq
A_{e_{i}}^{-}(x)<(\lambda _{e_{i}}\wedge \mu
_{e_{i}})(x)<B_{e_{i}}^{+}(x)\leq A_{e_{i}}^{+}(x),$ it is contradiction to
the fact that $\left( \widetilde{F,}\text{ }I\right) $ and $\left( 
\widetilde{G,}\text{ }J\right) $ are ECSSs. For the case $%
B_{e_{i}}^{-}(x)<A_{e_{i}}^{-}(x)=(\lambda _{e_{i}}\wedge \mu
_{e_{i}})(x)\leq B_{e_{i}}^{+}(x)\leq A_{e_{i}}^{+}(x),$ we have $(\lambda
_{e_{i}}\wedge \mu _{e_{i}})(x)\notin ((A_{e_{i}}\cup B_{e_{i}})^{-}(x)$,$%
(A_{e_{i}}\cup B_{e_{i}})^{+}(x))$ because $(A_{i}\cup
B_{i})^{-}(x)=A_{e_{i}}^{-}(x)=(\lambda _{e_{i}}\wedge \mu _{e_{i}})(x).$
Again assume that ${\Large \mathcal{\beta }}_{e_{i}}=B_{e_{i}}^{-}(x),$ then
we have $A_{e_{i}}^{-}(x)\leq B_{e_{i}}^{-}(x)\leq (\lambda _{e_{i}}\wedge
\mu _{e_{i}})(x)\leq B_{e_{i}}^{+}(x)\leq A_{e_{i}}^{+}(x).$ From this we
can write $A_{e_{i}}^{-}(x)\leq B_{e_{i}}^{-}(x)<(\lambda _{e_{i}}\wedge \mu
_{e_{i}})(x)<B_{e_{i}}^{+}(x)\leq A_{e_{i}}^{+}(x)$ or $A_{e_{i}}^{-}(x)\leq
B_{e_{i}}^{-}(x)=(\lambda _{e_{i}}\wedge \mu
_{e_{i}})(x)<B_{e_{i}}^{+}(x)\leq A_{e_{i}}^{+}(x).$ For the case $%
A_{e_{i}}^{-}(x)\leq B_{e_{i}}^{-}(x)<(\lambda _{e_{i}}\wedge \mu
_{e_{i}})(x)<B_{e_{i}}^{+}(x)\leq A_{e_{i}}^{+}(x),$ it is contradiction to
the fact $\left( \widetilde{F,}\text{ }I\right) $ and $\left( \widetilde{G,}%
\text{ }J\right) $ are ECSSs. And if we take the case $A_{e_{i}}^{-}(x)\leq
B_{e_{i}}^{-}(x)=(\lambda _{e_{i}}\wedge \mu _{e_{i}})(x)\leq
B_{e_{i}}^{+}(x)\leq A_{e_{i}}^{+}(x),$ we get $(\lambda _{e_{i}}\wedge \mu
_{e_{i}})(x)\notin ((A_{e_{i}}\cup B_{e_{i}})^{-}(x)$,$(A_{e_{i}}\cup
B_{e_{i}})^{+}(x))$ because, $(A_{e_{i}}\cup
B_{e_{i}})^{-}(x)=B_{e_{i}}^{-}(x)=(\lambda _{e_{i}}\wedge \mu _{e_{i}})(x).$
If $e_{i}\in I-J$ or $e_{i}\in J-I$, then the result is trivial. Hence $%
\left( \widetilde{F,}\text{ }I\right) \cup _{R}\left( \widetilde{G,}\text{ }%
J\right) $ is an ECSS in $X$.
\end{proof}

\begin{theorem}
Let $\left( \widetilde{F,}\text{ }I\right) =\left\{ \widetilde{F}%
(e_{i})=\left\{ <x,\text{ }A_{e_{i}}(x),\text{ }\lambda _{e_{i}}(x)>:\text{ }%
x\in X\right\} \text{ }e_{i}\in I,\right\} $ and $\left( \widetilde{G,}\text{
}J\right) =\left\{ \widetilde{G}(e_{i})=\left\{ <x,\text{ }B_{e_{i}}(x),%
\text{ }\mu _{e_{i}}(x)>:\text{ }x\in X\right\} \text{ }e_{i}\in J\right\} $
are ECSSs in $X$ such that 
\begin{equation*}
(\lambda _{e_{i}}\vee \mu _{e_{i}})(x)\in \left[ 
\begin{array}{c}
\min \{\max \{A_{e_{i}}^{+}(x),B_{e_{i}}^{-}(x)\},\max
\{A_{e_{i}}^{-}(x),B_{e_{i}}^{+}(x)\}\}, \\ 
\max \{\min \{A_{e_{i}}^{+}(x),B_{e_{i}}^{-}(x)\},\min
\{A_{e_{i}}^{-}(x),B_{i\text{ }}^{+}(x)\}\}%
\end{array}%
\right)
\end{equation*}%
for all $e_{i}\in I$ for all $e_{i}\in J$ and for all $x\in X$ then the $%
\left( \widetilde{F,}\text{ }I\right) \cap _{R}\left( \widetilde{G,}\text{ }%
J\right) $ is an ECSS in $X$.
\end{theorem}

\begin{proof}
Consider $\left( \widetilde{F,}\text{ }I\right) \cap _{R}\left( \widetilde{G,%
}\text{ }J\right) =\left( \widetilde{H,}\text{ }C\right) $, where $C=I\cap
J. $ $\widetilde{H}(e_{i})=\widetilde{F}(e_{i})\wedge _{R}\widetilde{G}%
(e_{i})$ is defined as%
\begin{equation*}
\widetilde{F}(e_{i})\wedge _{R}\widetilde{G}(e_{i})=\widetilde{H}%
(e_{i})=\left\{ 
\begin{array}{c}
<x,\text{ }r\min \{A_{e_{i}}(x),B_{e_{i}}(x)\},(\lambda _{e_{i}}\vee \mu
_{e_{i}})(x)>: \\ 
\text{ }x\in X,\text{ }e_{i}\in I\cap J%
\end{array}%
\right\}
\end{equation*}
For each $e_{i}\in I\cap J,$ take ${\Large \mathcal{\alpha }}_{e_{i}}=\min
\{\max \{A_{e_{i}}^{+}(x),B_{e_{i}}^{-}(x)\}$,$\max
\{A_{e_{i}}^{-}(x),B_{e_{i}}^{+}(x)\}\}$ and ${\Large \mathcal{\beta }}%
_{e_{i}}=\max \{\min \{A_{e_{i}}^{+}(x),B_{e_{i}}^{-}(x)\}$,$\min \{A_{i%
\text{ }}^{-}(x),B_{e_{i}}^{+}(x)\}\}.$ Then ${\Large \mathcal{\alpha }}%
_{e_{i}}$ is one of $A_{e_{i}}^{-}(x)$, $B_{e_{i}}^{-}(x)$, $A_{i\text{ }%
}^{+}(x)$ and $B_{e_{i}}^{+}(x).$ We consider ${\Large \mathcal{\alpha }}%
_{e_{i}}=B_{e_{i}}^{-}(x)$ or $B_{e_{i}}^{+}(x)$ only, as the remaining
cases are similar to this one. If ${\Large \mathcal{\alpha }}%
_{e_{i}}=B_{e_{i}}^{-}(x),$ then $A_{e_{i}}^{-}(x)\leq A_{e_{i}}^{+}(x)\leq
B_{e_{i}}^{-}(x)\leq B_{e_{i}}^{+}(x)$ and so ${\Large \mathcal{\beta }}%
_{e_{i}}=A_{e_{i}}^{+}(x).$ Then by given inequality we have $(A_{e_{i}}\cap
B_{e_{i}})^{+}(x)=A_{e_{i}}^{+}(x)={\Large \mathcal{\beta }}%
_{e_{i}}<(\lambda _{e_{i}}\vee \mu _{e_{i}})(x).$ Thus we have $(\lambda
_{e_{i}}\vee \mu _{e_{i}})(x)\notin ((A_{e_{i}}\cap B_{e_{i}})^{-}(x)$,$%
(A_{e_{i}}\cap B_{e_{i}})^{+}(x)).$ If ${\Large \mathcal{\alpha }}%
_{e_{i}}=B_{e_{i}}^{+}(x),$ then $A_{e_{i}}^{-}(x)\leq B_{e_{i}}^{+}(x)\leq
A_{e_{i}}^{+}(x)$ and so ${\Large \mathcal{\beta }}_{e_{i}}=\max
\{A_{e_{i}}^{-}(x),B_{e_{i}}^{-}(x)\}.$ Assume ${\Large \mathcal{\beta }}%
_{e_{i}}=A_{e_{i}}^{-}(x),$ then we have $B_{e_{i}}^{-}(x)\leq
A_{e_{i}}^{-}(x)<(\lambda _{e_{i}}\vee \mu _{e_{i}})(x)\leq
B_{e_{i}}^{+}(x)\leq A_{e_{i}}^{+}(x).$ So from this we can write $%
B_{e_{i}}^{-}(x)\leq A_{e_{i}}^{-}(x)<(\lambda _{e_{i}}\vee \mu
_{e_{i}})(x)<B_{e_{i}}^{+}(x)\leq A_{e_{i}}^{+}(x)$ or $B_{e_{i}}^{-}(x)\leq
A_{e_{i}}^{-}(x)<(\lambda _{e_{i}}\vee \mu _{e_{i}})(x)=B_{e_{i}}^{+}(x)\leq
A_{e_{i}}^{+}(x).$ For the case $B_{e_{i}}^{-}(x)\leq
A_{e_{i}}^{-}(x)<(\lambda _{e_{i}}\vee \mu _{e_{i}})(x)<B_{e_{i}}^{+}(x)\leq
A_{e_{i}}^{+}(x),$ it is contradiction to the fact that $\left( \widetilde{F,%
}\text{ }I\right) $ and $\left( \widetilde{G,}\text{ }J\right) $ are ECSSs.
For the case $B_{e_{i}}^{-}(x)\leq A_{e_{i}}^{-}(x)<(\lambda _{e_{i}}\vee
\mu _{e_{i}})(x)=B_{e_{i}}^{+}(x)\leq A_{e_{i}}^{+}(x),$ we have $(\lambda
_{e_{i}}\vee \mu _{e_{i}})(x)\notin ((A_{e_{i}}\cap B_{e_{i}})^{-}(x)$, $%
(A_{e_{i}}\cap B_{e_{i}})^{+}(x))$ because, $(A_{e_{i}}\cap
B_{e_{i}})^{+}(x)=B_{e_{i}}^{+}(x)=(\lambda _{e_{i}}\vee \mu _{e_{i}})(x).$
Again assume that

${\Large \mathcal{\beta }}_{e_{i}}=B_{e_{i}}^{-}(x),$ then we have $%
A_{e_{i}}^{-}(x)\leq B_{e_{i}}^{-}(x)<(\lambda _{e_{i}}\vee \mu
_{e_{i}})(x)\leq B_{e_{i}}^{+}(x)\leq A_{e_{i}}^{+}(x).$ From this we can
write $A_{e_{i}}^{-}(x)\leq B_{e_{i}}^{-}(x)<(\lambda _{e_{i}}\vee \mu
_{e_{i}})(x)<B_{e_{i}}^{+}(x)\leq A_{e_{i}}^{+}(x)$ or $A_{e_{i}}^{-}(x)\leq
B_{e_{i}}^{-}(x)<(\lambda _{e_{i}}\vee \mu _{e_{i}})(x)=B_{e_{i}}^{+}(x)\leq
A_{e_{i}}^{+}(x).$ For the case $A_{e_{i}}^{-}(x)\leq
B_{e_{i}}^{-}(x)<(\lambda _{e_{i}}\vee \mu _{e_{i}})(x)<B_{e_{i}}^{+}(x)\leq
A_{e_{i}}^{+}(x),$ it is contradiction to the fact $\left( \widetilde{F,}%
\text{ }I\right) $ and $\left( \widetilde{G,}\text{ }J\right) $ are ECSSs.
And if we take the case $A_{e_{i}}^{-}(x)\leq B_{e_{i}}^{-}(x)<(\lambda
_{e_{i}}\vee \mu _{e_{i}})(x)=B_{e_{i}}^{+}(x)\leq A_{e_{i}}^{+}(x),$ we get 
$(\lambda _{e_{i}}\vee \mu _{e_{i}})(x)\notin ((A_{e_{i}}\cap
B_{e_{i}})^{-}(x)$,$(A_{e_{i}}\cap B_{e_{i}})^{+}(x))$ because, $%
(A_{e_{i}}\cap B_{e_{i}})^{+}(x)=B_{e_{i}}^{+}(x)=(\lambda _{e_{i}}\vee \mu
_{e_{i}})(x).$ Thus $\left( \widetilde{F,}\text{ }I\right) \cap _{R}\left( 
\widetilde{G,}\text{ }J\right) $ is an ECSS for all $e_{i}\in I\cap J.$
\end{proof}

\begin{theorem}
Let $\left( \widetilde{F,}\text{ }I\right) =\left\{ \widetilde{F}%
(e_{i})=\left\{ <x,\text{ }A_{e_{i}\text{ }}(x),\text{ }\lambda _{e_{i}\text{
}}(x)>:\text{ }x\in X\right\} \text{ }e_{i}\in I,\right\} $ and $\left( 
\widetilde{G,}\text{ }J\right) =\left\{ \widetilde{G}(e_{i})=\left\{ <x,%
\text{ }B_{e_{i}\text{ }}(x),\text{ }\mu _{e_{i}\text{ }}(x)>:\text{ }x\in
X\right\} e_{i}\in J\right\} $ are cubic soft sets in $X$ such that 
\begin{eqnarray*}
&&\min \{\max \{A_{e_{i}}^{+}(x),B_{e_{i}}^{-}(x)\},\max
\{A_{e_{i}}^{-}(x),B_{e_{i}}^{+}(x)\}\} \\
&=&(\lambda _{e_{i}}\wedge \mu _{e_{i}})(x)=\max \{\min
\{A_{e_{i}}^{+}(x),B_{e_{i}}^{-}(x)\},\min \{A_{e_{i}}^{-}(x),B_{i\text{ }%
}^{+}(x)\}\}
\end{eqnarray*}%
for all $e_{i}\in I$ for all $e_{i}\in J$ and for all $x\in X,$ then $\left( 
\widetilde{F,}\text{ }I\right) \cap _{R}\left( \widetilde{G,}\text{ }%
J\right) $ is both an ECSS and an ICSS in $X$.
\end{theorem}

\begin{proof}
Consider $\left( \widetilde{F,}\text{ }I\right) \cap _{R}\left( \widetilde{G,%
}\text{ }J\right) =\left( \widetilde{H,}\text{ }C\right) $ where $C=I\cap J.$
Also $\widetilde{H}(e_{i})=\widetilde{F}(e_{i})\wedge _{R}\widetilde{G}%
(e_{i})$ is defined as%
\begin{equation*}
\widetilde{F}(e_{i})\wedge _{P}\widetilde{G}(e_{i})=\widetilde{H}%
(e_{i})=\left\{ 
\begin{array}{c}
<x,\text{ }r\min \{A_{e_{i}}(x),B_{e_{i}}(x)\},(\lambda _{e_{i}}\vee \mu
_{e_{i}})(x)>: \\ 
\text{ }x\in X,\text{ }e_{i}\in I\cap J.%
\end{array}%
\right\}
\end{equation*}%
For each $e_{i}\in I\cap J,$ take ${\Large \mathcal{\alpha }}_{e_{i}}=\min
\{\max \{A_{e_{i}}^{+}(x),B_{e_{i}}^{-}(x)\}$,$\max
\{A_{e_{i}}^{-}(x),B_{e_{i}}^{+}(x)\}\}$ and ${\Large \mathcal{\beta }}%
_{e_{i}}=\max \{\min \{A_{e_{i}}^{+}(x),B_{e_{i}}^{-}(x)\}$,$\min
\{A_{e_{i}}^{-}(x),B_{e_{i}}^{+}(x)\}\}.$ Then ${\Large \mathcal{\alpha }}%
_{e_{i}}$ is one of $A_{e_{i}}^{-}(x)$, $B_{e_{i}}^{-}(x)$, $%
A_{e_{i}}^{+}(x) $ and $B_{e_{i}}^{+}(x).$ We consider ${\Large \mathcal{%
\alpha }}_{e_{i}}=A_{e_{i}}^{-}(x)$ or $A_{e_{i}}^{+}(x)$ only, as remaining
cases are similar to this one. If ${\Large \mathcal{\alpha }}%
_{e_{i}}=A_{e_{i}}^{-}(x),$ then $B_{e_{i}}^{-}(x)\leq B_{e_{i}}^{+}(x)\leq
A_{e_{i}}^{-}(x)\leq A_{e_{i}}^{+}(x)$ and so ${\Large \mathcal{\beta }}%
_{e_{i}}=B_{e_{i}}^{+}(x).$ This implies that $A_{e_{i}}^{-}(x)={\Large 
\mathcal{\alpha }}_{e_{i}}=(\lambda _{e_{i}}\vee \mu _{e_{i}})(x)={\Large 
\mathcal{\beta }}_{e_{i}}=B_{e_{i}}^{+}(x).$

Thus $B_{e_{i}}^{-}(x)\leq B_{e_{i}}^{+}(x)=(\lambda _{e_{i}}\vee \mu
_{e_{i}})(x)=A_{e_{i}}^{-}(x)\leq A_{e_{i}}^{+}(x)$ which implies that $%
(\lambda _{e_{i}}\vee \mu _{e_{i}})(x)=B_{e_{i}}^{+}(x)=(A_{e_{i}}\cap
B_{e_{i}})^{+}(x).$ Hence $(\lambda _{e_{i}}\vee \mu _{e_{i}})(x)\notin
((A_{e_{i}}\cap B_{e_{i}})^{-}(x)$,$(A_{e_{i}}\cap B_{e_{i}})^{+}(x))$ and $%
(A_{e_{i}}\cap B_{e_{i}})^{-}(x)\leq (\lambda _{e_{i}}\vee \mu
_{e_{i}})(x)\leq (A_{e_{i}}\cap B_{e_{i}})^{+}(x).$ If ${\Large \mathcal{%
\alpha }}_{e_{i}}=A_{e_{i}}^{+}(x),$ then $B_{e_{i}}^{-}(x)\leq
A_{e_{i}}^{+}(x)\leq B_{e_{i}}^{+}(x)$ and so $(\lambda _{e_{i}}\vee \mu
_{e_{i}})(x)=A_{e_{i}}^{+}(x)=(A_{e_{i}}\cap B_{e_{i}})^{+}(x).$ Hence $%
(\lambda _{e_{i}}\vee \mu _{e_{i}})(x)\notin ((A_{e_{i}}\cap
B_{e_{i}})^{-}(x)$,$(A_{e_{i}}\cap B_{e_{i}})^{+}(x))$ and $(A_{e_{i}}\cap
B_{e_{i}})^{-}(x)\leq (\lambda _{e_{i}}\vee \mu _{e_{i}})(x)\leq
(A_{e_{i}}\cap B_{e_{i}})^{+}(x).$ Consequently, we note that $\left( 
\widetilde{F,}\text{ }I\right) \cap _{R}\left( \widetilde{G,}\text{ }%
J\right) $ is both an ECSS and an ICSS in $X$.
\end{proof}

\begin{theorem}
Let $\left( \widetilde{F,}\text{ }I\right) =\left\{ \widetilde{F}%
(e_{i})=\left\{ <x,\text{ }A_{e_{i}\text{ }}(x),\text{ }\lambda _{e_{i}\text{
}}(x)>:\text{ }x\in X\right\} \text{ }e_{i}\in I\right\} $ and $\left( 
\widetilde{G,}\text{ }J\right) =\left\{ \widetilde{G}(e_{i})=\left\{ <x,%
\text{ }B_{e_{i}}(x),\text{ }\mu _{e_{i}}(x)>:\text{ }x\in X\right\}
e_{i}\in J\right\} $ are internal cubic soft sets in $X$ such that $(\lambda
_{e_{i}}\wedge \mu _{e_{i}})(x)\leq \max
\{A_{e_{i}}^{-}(x),B_{e_{i}}^{-}(x)\}$ for all $e_{i}\in I$ for all $%
e_{i}\in J$ and for all $x\in X,$ then $\left( \widetilde{F,}\text{ }%
I\right) \cup _{R}\left( \widetilde{G,}\text{ }J\right) $ is an ECSS.
\end{theorem}

\begin{proof}
Given that $\left( \widetilde{F,}\text{ }I\right) =\left\{ \widetilde{F}%
(e_{i})=\left\{ <x,\text{ }A_{e_{i}}(x),\text{ }\lambda _{e_{i}}(x)>:\text{ }%
x\in X\right\} \text{ }e_{i}\in I\right\} $ and $\left( \widetilde{G,}\text{ 
}J\right) =\left\{ \widetilde{G}(e_{i})=\left\{ <x,\text{ }B_{e_{i}}(x),%
\text{ }\mu _{e_{i}}(x)>:\text{ }x\in X\right\} e_{i}\in J\right\} $ are
internal cubic soft sets in $X$.

Thus for all $e_{i}\in I,$ we have$\ A_{e_{i}}^{-}(x)\leq \lambda
_{e_{i}}(x)\leq A_{e_{i}}^{+}(x),$ and for all $e_{i}\in J$ we have $%
B_{e_{i}}^{-}(x)\leq \mu _{e_{i}}(x)\leq B_{e_{i}}^{+}(x).$

Since $\left( \widetilde{F,}\text{ }I\right) \cup _{R}\left( \widetilde{G,}%
\text{ }J\right) $ is defined as $\left( \widetilde{F,}\text{ }I\right) \cup
_{R}\left( \widetilde{G,}\text{ }J\right) =\left( \widetilde{H,}\text{ }%
C\right) ,$ where $C=I\cup J$ and%
\begin{equation*}
\widetilde{H}(e_{i})=\left\{ 
\begin{array}{c}
\begin{array}{c}
\widetilde{F}(e_{i})\text{ If }e_{i}\in I-J \\ 
\widetilde{G}(e_{i})\text{ If }e_{i}\in J-I%
\end{array}%
\text{ \ \ \ \ \ \ \ \ \ \ \ \ } \\ 
\widetilde{F}(e_{i})\vee _{R}\widetilde{G}(e_{i})\text{ If }e_{i}\in I\cap J,%
\end{array}%
\right.
\end{equation*}%
where $\widetilde{F}(e_{i})\vee _{R}\widetilde{G}(e_{i})$ is defined as%
\begin{equation*}
\widetilde{F}(e_{i})\vee _{R}\widetilde{G}(e_{i})=\widetilde{H}%
(e_{i})=\left\{ 
\begin{array}{c}
<x,\text{ }r\max \{A_{e_{i}}(x),B_{e_{i}}(x)\},(\lambda _{e_{i}}\wedge \mu
_{e_{i}})(x)>: \\ 
x\in X,\text{ }e_{i}\in I\cap J.%
\end{array}%
\right\}
\end{equation*}%
Given condition is $(\lambda _{e_{i}}\wedge \mu _{e_{i}})(x)\leq \max
\{A_{e_{i}}^{-}(x),B_{e_{i}}^{-}(x)\}$ for all $e_{i}\in I$ for all $%
e_{i}\in J$ and for all $x\in X.$

This implies that 
\begin{eqnarray*}
(\lambda _{e_{i}}\wedge \mu _{e_{i}})(x) &\notin &((A_{e_{i}}\cup
B_{e_{i}})^{-}(x),(A_{e_{i}}\cup B_{e_{i}})^{+}(x)) \\
&=&(\max \{A_{e_{i}}^{-}(x),B_{e_{i}}^{-}(x)\},\max
\{A_{e_{i}}^{+}(x),B_{e_{i}}^{+}(x)\}).
\end{eqnarray*}%
Hence $\left( \widetilde{F,}\text{ }I\right) \cup _{R}\left( \widetilde{G,}%
\text{ }J\right) $ is an ECSS.
\end{proof}

\begin{theorem}
Let $\left( \widetilde{F,}\text{ }I\right) =\left\{ \widetilde{F}%
(e_{i})=\left\{ <x,\text{ }A_{i\text{ }}(x),\text{ }\lambda _{e_{i}\text{ }%
}(x)>:\text{ }x\in X\right\} \text{ }e_{i}\in I\right\} $ and $\left( 
\widetilde{G,}\text{ }J\right) =\left\{ \widetilde{G}(e_{i})=\left\{ <x,%
\text{ }B_{e_{i}}(x),\text{ }\mu _{e_{i}}(x)>:\text{ }x\in X\right\}
e_{i}\in J\right\} $ are internal cubic soft sets in $X$ such that $(\lambda
_{e_{i}}\vee \mu _{e_{i}})(x)\geq \min \{A_{e_{i}}^{+}(x),B_{e_{i}}^{+}(x)\}$
for all $e_{i}\in I,$ for all $e_{i}\in J$ and for all $x\in X$. Then $%
\left( \widetilde{F,}\text{ }I\right) \cap _{R}\left( \widetilde{G,}\text{ }%
J\right) $ is an ECSS in $X$.
\end{theorem}

\begin{proof}
Given that $\left( \widetilde{F,}\text{ }I\right) =\left\{ \widetilde{F}%
(e_{i})=\left\{ <x,\text{ }A_{e_{i}}(x),\text{ }\lambda _{e_{i}}(x)>:\text{ }%
x\in X\right\} \text{ }e_{i}\in I\right\} $ and $\left( \widetilde{G,}\text{ 
}J\right) =\left\{ \widetilde{G}(e_{i})=\left\{ <x,\text{ }B_{e_{i}}(x),%
\text{ }\mu _{e_{i}}(x)>:\text{ }x\in X\right\} \text{ }e_{i}\in J\right\} $
are internal cubic soft sets in $X$.

Thus for all $e_{i}\in I,$ we have $A_{e_{i}}^{-}(x)\leq \lambda
_{e_{i}}(x)\leq A_{e_{i}}^{+}(x)$ and for all $e_{i}\in J,$ we have $%
B_{e_{i}}^{-}(x)\leq \mu _{e_{i}}(x)\leq B_{e_{i}}^{+}(x).$

Since $\left( \widetilde{F,}\text{ }I\right) \cap _{R}\left( \widetilde{G,}%
\text{ }J\right) $ is defined as $\left( \widetilde{F,}\text{ }I\right) \cap
_{R}\left( \widetilde{G,}\text{ }J\right) =\left( \widetilde{H,}\text{ }%
C\right) $ where $C=I\cap J$ and

$\widetilde{H}(e_{i})=\widetilde{F}(e_{i})\wedge _{R}\widetilde{G}(e_{i})$
is defined as%
\begin{equation*}
\widetilde{F}(e_{i})\wedge _{R}\widetilde{G}(e_{i})=\left\{ \widetilde{H}%
(e_{i})=\left\{ 
\begin{array}{c}
<x,\text{ }r\min \{A_{e_{i}}(x),B_{e_{i}}(x)\},(\lambda _{e_{i}}\vee \mu
_{e_{i}})(x)>: \\ 
x\in X,\text{ }e_{i}\in I\cap J.%
\end{array}%
\right\} \right\} .
\end{equation*}%
Given condition is that $(\lambda _{e_{i}}\vee \mu _{e_{i}})(x)\geq \min
\{A_{e_{i}}^{+}(x),B_{e_{i}}^{+}(x)\}$ for all $e_{i}\in I,$ for all $%
e_{i}\in J$ and for all $x\in X.$

This implies that 
\begin{eqnarray*}
(\lambda _{e_{i}}\vee \mu _{e_{i}})(x) &\notin &((A_{e_{i}}\cap
B_{e_{i}})^{-}(x),(A\cap B)^{+}(x)) \\
&=&(\min \{A_{e_{i}}^{-}(x),B_{e_{i}}^{-}(x)\},\min
\{A_{e_{i}}^{+}(x),B_{e_{i}}^{+}(x)\}).
\end{eqnarray*}%
Hence, $\left( \widetilde{F,}\text{ }I\right) \cap _{R}\left( \widetilde{G,}%
\text{ }J\right) $ is an ECSS in $X$.
\end{proof}

\section{Conclusion}

In order to deal with many complicated problems in the fields of
engineering, social science, economics, medical science etc involving
uncertainties, classical methods are found to be inadequate in recent times.
In 1999 \cite{8}, Molodstov proposed a new mathematical tool for dealing
with uncertainties which is free of the difficulties present theories. He
introduced the novel concept of soft sets and established the fundamental
results of the new theory. He also showed how soft set theory is free from
parameterization inadequacy syndrome of fuzzy set theory, rough set tTheory
and probability theory etc. In this paper we discuss a new approach to soft
set through applications of cubic set. By combine of cubic set and soft set,
we introduce a new mathematical model which is called cubic soft set. We
introduce two types of cubic soft set 1) \textit{Internal cubic soft set
(ICSS),} 2) \textit{External cubic soft set (ICSS).} We describe
P-(R-)order,P-(R-)union, P-(R-)intersection and P-OR, R-OR, P-AND and R-AND
are introduced, and related properties are investigated. We show that the
P-union and the P-intersection of internal cubic soft sets are also internal
cubic soft sets. We provide conditions for the P-union (resp.
P-intersection) of two external cubic soft sets to be an internal cubic soft
set. We give conditions for the P-union (resp. R-union and R-intersection)
of two external cubic soft sets to be an external cubic soft set. We
consider conditions for the R-intersection (resp.P-intersection) of two
cubic sof sets to be both an external cubic soft set and an internal cubic
soft set.

In future we will focuse on Applications of cubic soft in information
sciences and knowledge System. We will also study cubic soft relations. We
will apply cubic soft set to algebraic stractures. This work also extended
to topological space.

\end{document}